\documentclass[12pt]{amsart}
\usepackage{amscd,amsmath,amssymb,amsfonts,verbatim}
\usepackage[cmtip, all]{xy}

\newtheorem{thm}{Theorem}[section]
\newtheorem{prop}[thm]{Proposition}
\newtheorem{lem}[thm]{Lemma}
\newtheorem{cor}[thm]{Corollary}

\theoremstyle{definition}

\newtheorem{defn}[thm]{Definition}
\theoremstyle{remark}
\newtheorem{remk}[thm]{Remark}
\newtheorem{remks}[thm]{Remarks}

\newtheorem{exm}[thm]{Example}
\newtheorem{exms}[thm]{Examples}
\newtheorem{notat}[thm]{Notation}
\numberwithin{equation}{section}

{\hfill$\square$\end{defn}}
{\hfill$\square$\end{remk}}
{\hfill$\square$\end{remks}}
{\hfill$\square$\end{exm}}
{\hfill$\square$\end{exms}}
{\hfill$\square$\end{notat}}

\newcommand{\thmref}{Theorem~\ref}
\newcommand{\propref}{Proposition~\ref}
\newcommand{\corref}{Corollary~\ref}

\newcommand{\lemref}{Lemma~\ref}

\newcommand{\sW}{{\mathcal W}}

\newcommand{\A}{{\mathbb A}}

\newcommand{\G}{{\mathbb G}}

\renewcommand{\P}{{\mathbb P}}

\newcommand{\W}{{\mathbb W}}

\newcommand{\CH}{{\rm CH}}

\newcommand{\inj}{\hookrightarrow}

\newcommand{\codim}{{\rm codim}}

\newcommand{\Spec}{{\rm Spec \,}}

\newcommand{\Sch}{{\operatorname{\mathbf{Sch}}}}

\newcommand{\op}{{\text{\rm op}}}

\newcommand{\Sm}{{\mathbf{Sm}}}
\newcommand{\SmProj}{{\mathbf{SmProj}}}
\newcommand{\SmAff}{{\mathbf{SmAff}}}
\newcommand{\SmLoc}{{\mathbf{SmLoc}}}

\newcommand{\ds}{{/\kern-3pt/}}

\newcommand{\ess}{\text{\rm{ess}}}

\newcommand{\Supp}{{\operatorname{Supp}}}

\newcommand{\colim}{\mathop{\text{\rm colim}}}

\newcommand{\sgn}{{\operatorname{\rm sgn}}}

\newcommand{\TZ}{{\operatorname{Tz}}}
\renewcommand{\TH}{{\operatorname{TCH}}}
\newcommand{\un}{\underline}
\newcommand{\ov}{\overline}

\newcommand{\dgn}{{\operatorname{degn}}}
\renewcommand{\dim}{\text{\rm dim}}

\newcommand{\tuborg}{\left\{\begin{array}{ll}}
\newcommand{\sluttuborg}{\end{array}\right.}

\newcommand{\tch}{{\mathcal{TCH}}}

\begin{document}
\title{On additive higher Chow groups of affine schemes}
\author{Amalendu Krishna and Jinhyun Park}
\address{School of Mathematics, Tata Institute of Fundamental Research, 1 Homi Bhabha Road, Colaba, Mumbai, India}
\email{amal@math.tifr.res.in}
\address{Department of Mathematical Sciences, KAIST, 291 Daehak-ro Yuseong-gu, Daejeon, 34141, Republic of Korea (South)}
\email{jinhyun@mathsci.kaist.ac.kr; jinhyun@kaist.edu}

%\maketitle


\begin{abstract}We show that the multivariate additive higher Chow groups of a smooth affine $k$-scheme $\Spec (R)$ essentially of finite type over a perfect field $k$ of characteristic $\not = 2$ form a differential graded module over the big de Rham-Witt complex $\W_m\Omega^{\bullet}_{R}$. In the univariate case, we show that additive higher Chow groups of $\Spec (R)$ form a Witt-complex over $R$. We use these structures to prove an \'etale descent for multivariate additive higher Chow groups.
\end{abstract}

\maketitle

%\tableofcontents

\section{Introduction}
The additive higher Chow groups $\TH^q (X, n;m)$ emerged originally in \cite{BE2} in part as an attempt to understand certain relative higher algebraic $K$-groups of schemes in terms of algebraic cycles. Since then, several papers \cite{KL}, \cite{KP}, \cite{KP Jussieu}, \cite{KP2},  \cite{P2}, \cite{P1}, \cite{R} have studied various aspects of these groups. But lack of a suitable moving lemma for smooth \emph{affine} varieties has been a hindrance in studies of their local behaviors. Its projective sibling was known by \cite{KP}. During the period of stagnation, the subject has evolved into the notion of `cycles with modulus' $\CH^q (X|D, n)$ by Binda-Kerz-Saito in \cite{BS}, \cite{KS} associated to pairs $(X,D)$ of schemes and effective Cartier divisors $D$, setting a more flexible ground, while this desired moving lemma for the affine case was obtained by W. Kai \cite{Kai} (See Theorem \ref{thm:moving affine}).

The above developments now propel the authors to continue their program of realizing the relative $K$-theory $K_n (X \times \Spec k[t]/(t^{m+1}), (t))$ in terms of additive higher Chow groups. More specifically, one of the aims in the program considered in this paper is to understand via additive higher Chow groups, the part of the above relative $K$-groups which was proven in \cite{Bloch crys} to give the crystalline cohomology. This part turned out to be isomorphic to the de Rham-Witt complexes as seen in \cite{Illusie}. This article is the first of the authors' papers that relate the additive higher Chow groups to the big de Rham-Witt complexes $\mathbb{W}_m \Omega _R ^{\bullet}$ of \cite{HeMa} and to the crystalline cohomology theory. % studied, in \cite{Bloch crys} and \cite{Illusie}. 
This gives a motivic description of the latter two objects.

While the general notion of cycles with modulus for $(X,D)$ provides a wider picture, the additive higher Chow groups still have a non-trivial operation not shared by the general case. One such is an analogue of the Pontryagin product on homology groups of Lie groups, which turns the additive higher Chow groups into a differential graded algebra (DGA). This product is induced by the structure of algebraic groups on $\A^1$ and $\G_m$ and their action on $X \times \A^r=:X[r]$ for $r \geq 1$.

The usefulness of such a product was already observed in the earliest papers on additive 0-cycles by Bloch-Esnault \cite{BE2} and R\"ulling \cite{R}. This product on higher dimensional additive higher Chow cycles was given in \cite{KP2} for smooth projective varieties. In \S \ref{sec:DGA} of this paper, we extend this product structure in two directions: (1) toward multivariate additive higher Chow groups and (2) on smooth affine varieties. In doing so, we generalize some of the necessary tools, such as the following normalization theorem, proven as Theorem \ref{thm:normalization}. %Here is a summary of the results we obtain.
Necessary definitions are recalled in \S \ref{sec:definitions}.

\begin{thm}
Let $X$ be a smooth scheme which is either quasi-projective or essentially of finite type over a field $k$. Let $D$ be an effective Cartier divisor on $X$. Then each cycle class in $ \CH^q (X|D, n)$ has a representative, all of whose codimension $1$ faces are trivial.
\end{thm}

The above theorem for ordinary higher Chow groups was proven by Bloch and has been a useful tool in dealing with algebraic cycles. In this paper, we use the above theorem to construct the following structure of differential graded algebra and differential graded modules on the multivariate additive higher Chow groups, where Theorem \ref{thm:Witt intro} is proven in Theorems \ref{thm:Witt over k}, \ref{thm:sm Witt}, and \ref{cor:sm Witt}, while Theorem \ref{thm:Main-multi} is proven in Theorem \ref{thm:DGA}.
%This structure in turn is heavily used in \cite{KP crys}.

\begin{thm}\label{thm:Witt intro}
Let $X$ be a smooth scheme which is either affine essentially of finite type or projective over a perfect field $k$ of characteristic $\not = 2$
\begin{enumerate}
\item The additive higher Chow groups $\{ \TH^q(X, n; m) \}_{q, n, m \in \mathbb{N} }$  has a functorial structure of a restricted Witt-complex over $k$.
\item If $X=\Spec (R)$ is affine, then $\{ \TH^q(X, n; m) \}_{q, n, m \in \mathbb{N} }$  has a structure of a restricted Witt-complex over $R$.
\item For $X$ as in $(2)$, there is a natural map of restricted Witt-complexes $\tau_{n,m} ^R: \mathbb{W}_m \Omega_R ^{n-1} \to \TH^n (R, n;m)$.
\end{enumerate}
\end{thm}

\begin{thm}\label{thm:Main-multi}
Let $r \geq 1$. For a smooth scheme $X$ which is either affine essentially of finite type or projective over a perfect field $k$ of characteristic $\not = 2$, the multivariate additive higher Chow groups $\{\CH^q(X[r]|D_{\un{m}},n)\}_{q, n \ge 0}$ with  modulus $\un{m}= (m_1, \cdots, m_r)$, where $m_i \geq 1$, form a differential graded module over the DGA $\{ \TH^q(X, n; |\un{m}|-1) \}_{q, n \ge 1}$, where $| \un{m}| = \sum_{i=1} ^r m_i$. In particular, each $\CH^q(X[r]|D_{\un{m}},n)$ is a $\W_{(|\un{m}|-1)}(R)$-module, when $X= \Spec (R)$ is affine.
\end{thm}

%\begin{thm}\label{thm:Witt intro}
%Let $X$ be a smooth scheme which is either affine essentially of finite type or projective over a perfect field $k$ of characteristic $\not = 2$
%\begin{enumerate}
%\item The additive higher Chow groups $\{ \TH^q(X, n; m) \}_{q, n, m \in \mathbb{N} }$  has a functorial structure of a restricted Witt-complex over $k$.
%\item If $X=\Spec (R)$ is affine, then $\{ \TH^q(X, n; m) \}_{q, n, m \in \mathbb{N} }$  has a structure of a restricted Witt-complex over $R$.
%\item For $X$ as in $(2)$, there is a natural map of restricted Witt-complexes $\tau_{n,m} ^R: \mathbb{W}_m \Omega_R ^{n-1} \to \TH^n (R, n;m)$.
%\end{enumerate}
%\end{thm}

The above structures on the univariate and multivariate additive higher Chow groups suggest an expectation that these groups may describe the algebraic $K$-theory relative to nilpotent thickenings of the coordinate axes in an affine space over a smooth scheme. The calculations of such relative $K$-theory by Hesselholt in \cite{Hess-1} and \cite{Hess-2} show that any potential motivic cohomology which describes the above relative $K$-theory may have such a structure.

As part of our program of connecting the additive higher Chow groups with the relative $K$-theory, we show in \cite{KP crys} that the above map $\tau_{n,m} ^R$ is an isomorphism when $X$ is semi-local in addition, and we show how one deduces crystalline cohomology from additive higher Chow groups. The results of this paper form a crucial part in the process. 

Recall that the higher Chow groups of Bloch and algebraic $K$-theory do not satisfy {\'e}tale descent with integral coefficients.
%It is an interesting question to describe situations, under which these groups satisfy {\'e}tale descent integrally. 
As an application of \thmref{thm:Main-multi}, we show that the {\'e}tale descent is actually true for the multivariate additive higher Chow groups in the following setting:

\begin{thm}\label{thm:descent}
Let $r \geq 1$ and let $X$ be a smooth scheme which is either affine essentially of finite type or projective over a perfect field $k$ of characteristic $\not = 2$. Let $G$ be a finite group of order prime to ${\rm char}(k)$, acting freely on $X$ with the quotient $f: X \to X/G$. Then for all $q, n \ge 0$ and and $\un{m} = (m_1, \cdots, m_r)$ with $m_i \ge 1$ for $1 \le i \le r$, the pull-back map $f^*$ induces an isomorphism
\[
\CH^q(X/G[r]|D_{\un{m}},n) \xrightarrow{\simeq} {\rm H}^0(G, \CH^q(X[r]|D_{\un{m}},n)). 
\]

\end{thm}

Note that the quotient $X/G$ exists under the hypothesis on $X$. Since the corresponding descent is not yet known for the relative $K$-theory of nilpotent thickenings of the coordinate axes in an affine space over a smooth scheme, the above theorem suggests that this descent could be indeed true for the relative $K$-theory. 

\bigskip

\textbf{Conventions.} In this paper, $k$ will denote the base field which will be assumed to be perfect after \S\ref{sec:moving}. A $k$-scheme is a separated scheme of finite type over $k$. A $k$-variety is a reduced $k$-scheme. The product $X \times Y$ means usually $X \times _k Y$, unless said otherwise. We let $\Sch_k$ be the category of $k$-schemes, $\Sm_k$ of smooth $k$-schemes, and $\SmAff_k$ of smooth affine $k$-schemes. A scheme essentially of finite type is a scheme obtained by localizing at a finite subset (including $\emptyset$) of a finite type $k$-scheme. For $\mathcal{C} =  \Sch_k,\Sm_k, \SmAff_k$, we let $\mathcal{C}^{\ess}$ be the extension of the category $\mathcal{C}$ obtained by localizing at a finite subset (including $\emptyset$) of objects in $\mathcal{C}$. We let $\SmLoc_k$ be the category of smooth semi-local $k$-schemes essentially of finite type over $k$. So, $\SmAff_k ^{\ess} = \SmAff_k \cup \SmLoc_k$ for the objects. When we say a semi-local $k$-scheme, we always mean one that is essentially of finite type over $k$. Let $\SmProj_k$ be the category of smooth projective $k$-schemes.

\section{Recollection of  basic definitions}
\label{sec:definitions}

For $\P^1  = {\rm Proj}_k(k[s_0, s_1])$, we let $y = {s_1}/{s_0}$ its coordinate. Let $\square:= \P^1 \setminus \{1\}$.  For $n \geq 1$, let $(y_1, \cdots, y_n) \in \square^n$ be the coordinates. A face $F\subset \square^n$ means a closed subscheme defined by the set of equations of the form $\{y_{i_1} = \epsilon_{1}, \cdots , y_{i_s} = \epsilon_{s}\}$ for an increasing sequence $\{ i_j | 1 \leq j \leq s \} \subset \{ 1, \cdots, n \}$ and $\epsilon_j \in \{ 0, \infty \}$. We allow $s=0$, in which case $F = \square^n$. Let $\ov{\square}:= \mathbb{P}^1$. A face of $\ov{\square}^n$ is the closure of a face in $\square^n$. For $1 \leq i \leq n$, let $F_{n,i} ^1 \subset \ov{\square}^n$ be the closed subscheme given by $\{ y_i = 1 \}$. Let $F_n ^1:= \sum_{i=1} ^n F_{n,i} ^1$, which is the cycle associated to the closed subscheme $\ov{\square}^n \setminus \square^n$. Let $\square^0 = \ov{\square}^0 := \Spec (k)$. Let $\iota_{n,i, \epsilon}: \square^{n-1} \inj \square^{n}$ be the inclusion $(y_1, \cdots, y_{n-1} ) \mapsto (y_1, \cdots, y_{i-1}, \epsilon, y_i, \cdots, y_{n-1})$.

\subsection{Cycles with modulus}\label{sect:CWM}

Let $X \in \Sch_k ^{\ess}$. Recall (\cite[\S 2]{KPv}) that for effective Cartier divisors $D_1$ and $D_2$ on $X$, we say $D_1 \leq D_2$ if $D_1 + D= D_2$ for some effective Cartier divisor $D$ on $X$. A \emph{scheme with an effective divisor} (sed) is a pair $(X,D)$, where $X \in \Sch_k ^\ess$ and $D$ an effective Cartier divisor. A morphism $f: (Y, E) \to (X, D)$ of seds is a morphism $f: Y \to X$ in $\Sch_k^\ess$ such that $f^*(D)$ is defined as a Cartier divisor on $Y$ and $f^* (D) \leq E$. In particular, $f^{-1} (D) \subset E$. If $f: Y \to X$ is a morphism of $k$-schemes, and $(X, D)$ is a sed such that $f^{-1} (D ) = \emptyset$, then $f: (Y , \emptyset) \to (X, D)$ is a morphism of seds.

\begin{defn}[{\cite{BS}, \cite{KS}}]\label{defn:modulus}
Let $(X,D)$ and $(\ov{Y},E)$ be schemes with effective divisors. Let $Y = \ov{Y} \setminus E$. Let $V \subset X \times Y$ be an integral closed subscheme with closure $\ov{V} \subset X \times \ov{Y}$. We say $V$ \emph{has modulus $D$ (relative to $E$)} if
$
\nu^*_V(D \times \ov{Y}) \le \nu^*_V(X \times E)
$
on $\ov{V}^N$, where $\nu_V: \ov{V}^N \to \ov{V} \inj X \times \ov{Y}$ is the normalization followed by the closed immersion.
\end{defn}  

%In case $Y= \ov{Y} = \Spec(k)$, that $V$ has modulus $D$ on $X \times Y$ is equivalent to $V \cap D = \emptyset$. 
Recall the following containment lemma from \cite[Proposition 2.4]{KPv} (see also \cite[Lemma~2.1]{BS} and \cite[Proposition 2.4]{KP}):

\begin{prop}\label{prop:CL*}
Let $(X,D)$ and $(\ov{Y}, E)$ be schemes with effective divisors and $Y= \ov{Y} \setminus E$. If $V \subset X \times Y$ is a closed subscheme with modulus $D$ relative to $E$, then any closed subscheme $W \subset V$ also has modulus $D$ relative to $E$.
\end{prop}

\begin{defn}[{\cite{BS}, \cite{KS}}]\label{defn:partial complex}Let $(X,D)$ be a scheme with an effective divisor. For $s \in \mathbb{Z}$ and $n \geq 0$, let $\un{z}_s (X|D, n)$ be the free abelian group on integral closed subschemes $V \subset X \times \square^n$ of dimension $s+n$ satisfying the following conditions:

\begin{enumerate}
\item (Face condition) for each face $F \subset \square^n$, $V$ intersects $X \times F$ properly. 
\item (Modulus condition) $V$ has modulus $D$ relative to $F_n ^1$ on $X \times \square^n$. 
\end{enumerate}
\end{defn}

We usually drop the phrase ``relative to $F_n^1$'' for simplicity. A cycle in $\un{z}_s (X|D, n)$ is called an \emph{admissible cycle with modulus $D$}. One checks that $(n \mapsto \un{z}_s(X|D,n))$ is a cubical abelian group. In particular, the groups  $\un{z}_s (X|D, n)$ form a complex with the boundary map $\partial= \sum_{i=1} ^n (-1)^i (\partial_i ^{\infty} - \partial_i ^0)$, where $\partial_i ^{\epsilon} = \iota_{n, i, \epsilon} ^*$. 

\begin{defn}[{\cite{BS}, \cite{KS}}]
The complex $(z_s(X|D, \bullet), \partial)$ is the nondegenerate complex associated to $(n\mapsto \un{z}_s(X|D,n))$, i.e.,
$
z_s(X|D, n): = {\un{z}_s(X|D, n)}/{\un{z}_s(X|D, n)_{\dgn}}.
$
The homology $\CH_s(X|D, n): = {\rm H}_n (z_s(X|D, \bullet))$ for $n \geq 0$ is called \emph{higher Chow group of $X$ with modulus $D$}. If $X$ is equidimensional of dimension $d$, for $q \geq 0$, we write $\CH^q(X|D, n) = \CH_{d-q}(X|D, n)$.
\end{defn}

Here is a special case from \cite{KPv}:

\begin{defn}\label{defn:AHC}Let $X \in \Sch_k ^{\rm ess}$. For $r \geq 1$, let $X[r] := X \times \mathbb{A}^r$. When $(t_1, \cdots, t_r) \in \mathbb{A}^r$ are the coordinates, and $m_1, \cdots, m_r \geq 1$ are integers, let $D_{\un{m}}$ be the divisor on $X[r]$ given by the equation $\{ t_1 ^{m_1} \cdots t_r ^{m_r} = 0 \}$. The groups $\CH^q (X[r]|D_{\un{m}}, n)$ are called \emph{multivariate additive higher Chow groups} of $X$. For simplicity, we often say ``a cycle with modulus $\un{m}$'' for ``a cycle with modulus $D_{\un{m}}$.'' For an $r$-tuple of integers $\un{m} = (m_1, \cdots , m_r)$, we write $|\un{m}| = \sum_{i=1} ^r  m_i$. We shall say that $\un{m} \ge p$ if $m_i \ge p$ for each $i$.

When $r=1$, we obtain additive higher Chow groups, and as in \cite{KP2}, we often use the older notations $\TZ^q (X, n+1; m-1)$ for $z^q (X[1]|D_{m}, n)$ and $\TH^q (X, n+1; m-1)$ for $\CH^q (X[1]|D_m, n)$. In such cases, note that the modulus $m$ is shifted by $1$ from the above sense.
\end{defn}

\begin{defn}\label{defn:complex for moving}
Let $\mathcal{W}$ be a finite set of locally closed subsets of $X$ and let $e: \mathcal{W} \to \mathbb{Z}_{\geq 0}$ be a set function. Let $\un{z} ^q _{\mathcal{W}, e} (X|D, n) $ be the subgroup generated by integral cycles $Z \in \un{z}^q (X|D, n)$ such that for each $W \in \mathcal{W}$ and each face $F \subset \square^n$, we have $\codim _{W \times F} (Z \cap (W \times F)) \geq q - e(W)$. They form a subcomplex $\un{z}^q _{\mathcal{W},e}(X|D, \bullet)$ of $\un{z}^q (X|D, \bullet)$. Modding out by degenerate cycles, we obtain the subcomplex $z^q _{\mathcal{W},e} (X|D, \bullet)\subset z^q (X|D, \bullet)$. We write $z^q _{\mathcal{W}} (X|D, \bullet) := z^q_{\mathcal{W}, 0} (X|D, \bullet)$. For additive higher Chow cycles, we write $\un{\TZ} ^q _{\mathcal{W}} (X, n;m)$ for $\un{z}^q _{\mathcal{W}[1]} ( X[1] | D_{m+1}, n-1)$, where $\mathcal{W}[1] = \{ W[1] \ | \ W \in \mathcal{W} \}$.
\end{defn}

Here are some basic lemmas used in the paper:

%\begin{lem}[{\cite[Lemma~2.1]{KP}}]\label{lem:Pull-D}
%Let $X$ be a normal scheme and let $D_1$ and $D_2$ be effective Cartier divisors on $X$ such that $D_1 \ge D_2$ as Weil divisors. Let $Y \subset X$ be a closed subset which intersects $D_1$ and $D_2$ properly. Let $f: Y^N \to X$ be the composite of the inclusion and the normalization of $Y_{\rm red}$. Then $f^*(D_1) \ge f^*(D_2)$.
%\end{lem}

\begin{lem}[{\cite[Lemma 2.2]{KPv}}]\label{lem:cancel}
Let $f: Y \to X$ be a dominant map of normal integral $k$-schemes. Let $D$ be a Cartier divisor on $X$ such that the generic points of $\Supp(D)$ are contained in $f(Y)$. Suppose that $f^*(D) \ge 0$ on $Y$. Then $D \ge 0$ on $X$.
\end{lem}

\begin{lem}[{\cite[Lemma 2.9]{KPv}}]\label{lem:projective image}Let $f: Y \to X$ be a proper morphism of quasi-projective $k$-varieties. Let $D \subset X$ be an effective Cartier divisor such that $f(Y) \not \subset D$. Let $Z \in z^q (Y|f^* (D), n)$ be an irreducible cycle. Let $W = f(Z)$ on $X \times \square^n$. Then $W \in z^s (X|D, n)$, where $s= \codim_{X \times \square^n} (W)$. 
\end{lem}

\begin{lem}\label{lem:admislocal}Let $X$ be a $k$-scheme, and let $\{ U_i \}_{i \in I}$ be an open cover of $X$. Let $Z \in z^q (X \times \square^n)$ and let $Z_{U_i}$ be the flat pull-back to $U_i \times \square^n$. Then $Z \in z^q (X|D, n)$ if and only if for each $i \in I$, we have $Z_{U_i} \in z^q (U_i | D_{U_i}, n)$, where $D_{U_i}$ is the restriction of $D$ on $U_i$.
\end{lem}

\begin{proof}
The direction $(\Rightarrow)$ is obvious since flat pull-backs respect admissibility of cycles with modulus by \cite[Proposition 2.12]{KPv}. For the direction $(\Leftarrow)$, we may assume $Z$ is irreducible. In this case, it is easily checked that the face and the modulus conditions are both local on the base $X$. 
\end{proof}

\subsection{de Rham-Witt complexes}

\subsubsection{Ring of big Witt-vectors}\label{subsection:Witt ring}
Let $R$ be a commutative ring with unit. We recall the definition of the ring of big Witt-vectors of $R$ (see \cite[\S 4]{Hesselholt AM} or \cite[Appendix A]{R}). A \emph{truncation set} $S \subset \mathbb{N}$ is a non-empty subset such that if $s \in S$ and $t | s$, then $t \in S$. As a set, let $\mathbb{W}_S (R):= R^S$ and define the map $w: \mathbb{W}_S (R) \to R^S$ by sending $a = (a_s)_{s \in S}$ to $w(a) = (w(a)_s)_{s \in S}$, where $  w(a)_s := \sum_{ t | s} t a_t ^{ s/t}.$ When $R^S$ on the target of $w$ is given the component-wise ring structure, it is known that there is a unique functorial ring structure on $\mathbb{W}_S (R)$ such that $w$ is a ring homomorphism (see \cite[Proposition 1.2]{Hesselholt AM}). When $S= \{ 1, \cdots, m \}$, we write $\mathbb{W}_m (R): = \mathbb{W}_S (R)$. 

There is another description. Let $\mathbb{W}(R):= \mathbb{W}_{\mathbb{N}} (R)$. Consider the multiplicative group $(1 + tR[[t]])^{\times}$, where $t$ is an indeterminate. Then there is a natural bijection $\mathbb{W}(R) \simeq (1+ tR[[t]])^{\times}$, where the addition in $\mathbb{W}(R)$ corresponds to the multiplication of formal power series. For a truncation set $S$, we can describe $\mathbb{W}_S (R)$ as the quotient of $(1 + tR[[t]])^{\times}$ by a suitable subgroup $I_S$. See \cite[A.7]{R} for details. In case $S= \{ 1, \cdots, m \}$, we can write $\mathbb{W}_m (R) = (1+ tR[[t]])^{\times} / (1 + t^{m+1} R[[t]])^{\times}$ as an additive group.

For $a \in R$, the Teichm\"uller lift $[a] \in \mathbb{W}_S (R)$ corresponds to the image of $1-at \in (1 + t R[[t]])^{\times}$. This yields a multiplicative map $[-]: R \to \mathbb{W}_S (R)$. The additive identity element of $\mathbb{W}_m (R)$ corresponds to the unit polynomial $1$ and the multiplicative identity element corresponds to the polynomial $1-t$.

\subsubsection{de Rham-Witt complex}\label{sec:DRW}
Let $p$ be an odd prime and $R$ be a $\mathbb{Z}_{(p)}$-algebra.{\footnote{A definition of Witt-complex over a more general ring $R$ can be found in \cite[Definition~4.1]{Hesselholt AM}.} For each truncation set $S$, there is a differential graded algebra $\mathbb{W}_S \Omega_R ^{\bullet}$ called the big de Rham-Witt complex over $R$. This defines a contravariant functor on the category of truncation sets. This is an initial object in the category of $V$-complexes and in the category of Witt-complexes over $R$. For details, see \cite{HeMa} and \cite[\S 1]{R}. When $S$ is a finite truncation set, we have $\mathbb{W}_S \Omega_R ^{\bullet} = \Omega_{\mathbb{W}_S (R)/\mathbb{Z}}^{\bullet}/ N_S ^{\bullet},$ where $N_S ^{\bullet}$ is the differential graded ideal given by some generators (\cite[Proposition 1.2]{R}). In case $S= \{ 1, 2, \cdots, m \}$, we write $\mathbb{W}_m \Omega_R ^{\bullet}$ for this object. 

Here is another relevant object for this paper from \cite[Definition 1.1.1]{HeMa}; \emph{a restricted Witt-complex over $R$} is a pro-system of differential graded $\mathbb{Z}$-algebras $((E_m)_{m \in \mathbb{N}}, \mathfrak{R}: E_{m+1} \to E_m)$, with homomorphisms of graded rings $(F_r: E_{rm+r-1} \to E_m)_{m , r \in \mathbb{N}}$ called the \emph{Frobenius} maps, and homomorphisms of graded groups $(V_r: E_m \to E_{rm+r-1})_{m,r \in \mathbb{N}}$ called the \emph{Verschiebung} maps, satisfying the following relations for all $n, r,s  \in \mathbb{N}$:
\begin{enumerate}
\item [(i)] $\mathfrak{R}F_r= F_r \mathfrak{R}^r, \mathfrak{R}^r V_r = V_r \mathfrak{R}, F_1 = V_1= {\rm Id}, F_r F_s = F_{rs}, V_r V_s = V_{rs};$
\item [(ii)] $F_r V_r = r$. When $(r,s) = 1$, $F_r V_s = V_s F_r$ on $E_{rm+r-1};$
\item [(iii)] $V_r (F_r (x)y) = x V_r (y)$ for all $x \in E_{rm+r-1}$ and $y \in E_m;$ (projection formula)
\item [(iv)] $F_r d V_r = d$, where $d$ is the differential of the DGAs.
\end{enumerate}
Furthermore, we require that there is a homomorphism of pro-rings $(\lambda: \mathbb{W}_m (R) \to E_m ^0)_{m \in \mathbb{N}}$ that commutes with $F_r$ and $V_r$, satisfying
\begin{enumerate}
\item [(v)] $F_r d \lambda ([a]) = \lambda ([a]^{r-1}) d \lambda ([a])$ for all $a \in R$ and $r \in \mathbb{N}$.
\end{enumerate}

The pro-system $\{\mathbb{W}_m \Omega_R ^{\bullet} \}_{m \geq 1}$ is the initial object in the category of restricted Witt-complexes over $R$
(See \cite[Proposition 1.15]{R}).

\section{Normalization theorem}\label{sec:normalization}
Let $k$ be any field. The aim of this section is to prove Theorem \ref{thm:normalization}. Such results were known when $D=\emptyset$, or when $X$ is replaced by $X \times \mathbb{A}^1$ with $D= \{ t^{m+1} = 0 \}$ for $t \in \mathbb{A}^1$. We generalize it to higher Chow groups with modulus.

\begin{defn} Let $(X,D)$ be a scheme with an effective divisor. Let $z_N ^q (X|D, n) $ be the subgroup of cycles $\alpha\in z^q (X|D,n)$ such that $\partial_i ^0 (\alpha) = 0$ for all $1 \leq i \leq n$ and $\partial ^{\infty} _i (\alpha) = 0$ for $2 \leq i \leq n$. One checks that $\partial_1 ^{\infty} \circ \partial_1 ^{\infty} = 0$. Writing $\partial_1 ^{\infty}$ as $\partial^N$, we obtain a subcomplex $\iota: (z _N ^q (X|D, \bullet), \partial^N) \hookrightarrow  (z^q (X|D, \bullet), \partial)$. 
\end{defn}

\begin{thm}\label{thm:normalization}Let $X \in \Sm^{\rm ess}_k$ and let $D \subset X$ be an effective Cartier divisor. Then $\iota: z_N ^q (X|D, \bullet) \to z^q (X|D, \bullet)$ is a quasi-isomorphism. In particular, every cycle class in $\CH^q (X|D, n)$ can be represented by a cycle $\alpha$ such that $\partial_i ^{\epsilon} (\alpha) = 0$ for all $1 \leq i \leq n$ and $\epsilon = 0, \infty$.
\end{thm}

Let $\textbf{Cube}$ be the standard category of cubes (see \cite[\S 1]{LevineSM}) so that a cubical abelian group is a functor $\textbf{Cube}^{\rm op} \to (\mathbf{Ab})$. Recall also from \emph{loc.cit.} that an extended cubical abelian is a functor $\textbf{ECube}^{\rm op} \to (\mathbf{Ab})$, where $\textbf{ECube}$ is the smallest symmetric monoidal subcategory of \textbf{Sets} containing \textbf{Cube} and the morphism $\mu: \un{2} \to \un{1}$. The essential point of the proof of Theorem \ref{thm:normalization} is

\begin{thm}\label{thm:ecube} Let $X \in \Sm^{\rm ess}_k$ and $D \subset X$ be an effective Cartier divisor. Then $(\un{n} \mapsto z^q (X|D;n))$ is an extended cubical abelian group.
\end{thm}

 If Theorem \ref{thm:ecube} holds, then \cite[Lemma 1.6]{LevineSM} implies Theorem \ref{thm:normalization}. We suppose $(X,D)$ is as in Theorem \ref{thm:normalization} in what follows. The idea is similar to that of \cite[Appendix]{KP2}.

Let $q_1: \square^2 \to \square$ be the morphism $(y_1, y_2) \mapsto y_1 + y_2 - y_1 y_2$ if $y_1, y_2 \not = \infty$, and $(y_1, y_2) \mapsto \infty$ if $y_1$ or $y_2 = \infty$. Under the identification $\psi: \square \simeq \mathbb{A}^1$ given by $y \mapsto {1}/{(1-y)}$ (which sends $\{ \infty, 0 \}$ to $\{ 0, 1 \}$), this map $q_1$ is equivalent to $q_{1, \psi} : \mathbb{A}^2 \to \mathbb{A}^1$ given by $(y_1, y_2) \mapsto y_1 y_2$. For our convenience, we use this $\square_{\psi}:= (\mathbb{A}^1, \{ 0, 1 \})$ and cycles on $X \times \square_{\psi} ^n$. The boundary operator is $\partial = \sum_{i=1} ^n (-1)^i (\partial_i ^0 - \partial_i ^1 )$, and we replace $F_{n,i} ^1$ by $F_{n,i} ^{\infty} = \{ y_i = \infty \}$. We write $F_n ^{\infty}= \sum_{i=1} ^n F_{n,i} ^{\infty}$. We write $\ov{\square}_{\psi} = (\mathbb{P}^1, \{ 0, 1 \})$. The group of admissible cycles is $\un{z}^q _{\psi} (X|D, n)$. Consider $q_{n, \psi}: X \times \square^{n+1}_{\psi}  \to X \times \square^n_{\psi} $ given by $(x, y_1, \cdots, y_{n+1}) \mapsto (x, y_1, \cdots, y_{n-1}, y_n y_{n+1})$. 

\begin{prop}\label{prop:norm pullback}
For $Z \in z^q _{\psi}(X|D, n)$, we have $q_{n,\psi} ^* (Z) \in z^q _{\psi}  (X|D, n+1)$. 
\end{prop}

The delicacy of its proof lies in that the product map $q_{1, \psi} : \mathbb{A}^2 \to \mathbb{A}^1$ \emph{does not} extend to a morphism $(\mathbb{P}^1)^2 \to \mathbb{P}^1$ of varieties so that checking the modulus condition becomes nontrivial. We use a correspondence instead. For $n \geq 1$, let $i_n: W_n \hookrightarrow X \times \square_{\psi} ^{n+1} \times \ov{\square}^1_{\psi}$ be the closed subscheme defined by the equation $u_0 y_n y_{n+1} = u_1$, where $(y_1, \cdots, y_{n+1}) \in \square_{\psi} ^{n+1}$ and $(u_0; u_1) \in \ov{\square}^1_{\psi}$ are the coordinates. Let $y:= u_1 / u_0$. Its Zariski closure $\ov{W}_n \hookrightarrow X \times \ov{\square}_{\psi} ^{n+1} \times \ov{\square}_{\psi} ^1$ is given by the equation $u_0 u_{n,1} u_{n+1, 1} = u_1 u_{n,0} u_{n+1, 0}$, where $(u_{1,0}, u_{1,1}), \cdots, (u_{n+1, 0}, u_{n+1, 1})$ are the homogeneous coordinates of $\ov{\square}_{\psi}^{n+1}$ with $y_i = u_{i,1}/ u_{i,0}$. 

Consider $\theta_n: X \times \square_{\psi} ^{n+1} \times \ov{\square}_{\psi}^1 \to X \times \square_{\psi} ^n$ given by $(x, y_1, \cdots, y_{n+1}, (u_0; u_1)) \mapsto (x, y_1, \cdots, y_{n-1}, y_n y_{n+1})$, and let $\pi_n := \theta_n |_{{W}_n}$. To extend this $\pi_n$ to a morphism $\ov{\pi}_n$ on $\ov{W}_n$, we use the projection $\ov{\theta}_n : X \times \ov{\square} _{\psi}^{n+1} \times \ov{\square}_{\psi} ^1 \to X \times \ov{\square}_{\psi}^{n-1} \times \ov{\square}_{\psi}^1$, that drops the coordinates $(u_{n,0}; u_{n,1})$ and $(u_{n+1, 0}; u_{n+1, 1})$, and the projection $p_n: X \times \square_{\psi} ^{n+1} \times \ov{\square}_{\psi} ^1 \to X \times \square_{\psi} ^{n+1}$, that drops the last coordinate $(u_0; u_1)$.

\begin{lem}\label{lem:pi_n flat}
$(1)$ $W_n \cap \{ u_0 = 0 \} = \emptyset$, so that $W_n \subset X \times \square_{\psi}^{n+1} \times \square_{\psi} ^1$. $(2)$ $\ov{\theta}_n|_{W_n} = \pi_n$. Thus, we define $\ov{\pi}_n := \ov{\theta}_n|_{\ov{W}_n}$, which extends $\pi_n$. $(3)$ The varieties $W_n$ and $\ov{W}_n$ are smooth. $(4)$ Both $\pi_n$ and $\ov{\pi}_n$ are surjective flat morphisms of relative dimension $1$.
\end{lem}

\begin{proof}Its proof is almost identical to that of \cite[Lemma A.5]{KP2}. Part (1) follows from the defining equation of $W_n$, and (2) holds by definition. Let $\rho_n:= p_n |_{W_n}: W_n \to X \times \square_{\psi} ^{n+1}$. Since $X$ is smooth, using Jacobian criterion we check that $W_n$ is smooth. Furthermore, $\rho_n$ is an isomorphism with the obvious inverse. Under this identification, the morphism $\pi_n$ can also be regarded as the projection $(x, y_1, \cdots, y_n, y)\mapsto (x, y_1, \cdots, y_{n-1}, y)$ that drops $y_n$. In particular, $\pi_n$ is a smooth and surjective of relative dimension $1$. To check that $\ov{W}_n$ is smooth, one can do it locally on each open set where each of $u_{n,i}, u_{n+1, i}, u_i$ is nonzero for $i=0,1$. In each such open set, the equation for $\ov{W}_n$ takes the same form as for $W_n$, so that it is smooth again by Jacobian criterion. Similarly as for $\pi_n$, one sees $\ov{\pi}_n$ is of relative dimension $1$. Since $\ov{\theta}_n$ is projective and $\pi_n$ is surjective, the morphism $\ov{\pi}_n$ is projective and surjective. So, since $\ov{W}_n$ is smooth, the map $\ov{\pi}_n$ is flat by \cite[Exercise III-10.9, p.276]{Hartshorne}. Thus, we have (3) and (4).
\end{proof}

\begin{lem}\label{lem:mod norm}
Let $n \geq 1$ and let $Z \subset X \times \square_{\psi} ^n$ be a closed subscheme with modulus $D$. Then $Z':= (i_n)_* (\pi_n ^* (Z))$ also has modulus $D$.
\end{lem}

\begin{proof}Let $\ov{Z}$ and $\ov{Z}'$ be the Zariski closures of $Z$ and $Z'$ in $X \times \ov{\square}_{\psi}^n$ and $X \times \ov{\square}_{\psi} ^{n+1}$, respectively. By Lemma \ref{lem:pi_n flat} and the projectivity of $\ov{\theta}_n$, we see that $\ov{\theta}_n ({\ov{Z}'}) = \ov{Z}$. Consider the commutative diagram
\begin{equation}\label{eqn:mod-1}
\xymatrix{
{\ov{Z}'} ^N \ar[d] ^f \ar[r] ^{g} \ar@/^1.2pc/[rr] ^{\nu_{Z'}} & \ov{W}_n \ar@{^{(}->}[r] ^{\ov{i}_n \ \ \ \ \ \ \ \ } \ar[rd] ^{\ov{\pi}_n} & X \times \ov{\square} _{\psi}^{n+1} \times \ov{\square}_{\psi} ^1 \ar[d] ^{\ov{\theta}_n} \\
\ov{Z}^N \ar[rr]^{\nu_Z} & & X \times \ov{\square}_{\psi} ^n,}
\end{equation}
where $f$ is induced by the surjection $\ov{\theta}_n |_{\ov{Z}'} : \ov{Z}' \to \ov{Z}$, the maps $g$ and $\nu_Z$ are normalizations of $\ov{Z}'$ and $\ov{Z}$ composed with the closed immersions, and $\nu_{Z'}:= \ov{i}_n \circ g$. By the definition of $\ov{\theta}_n$, we have $\ov{\theta}_n ^* (D \times \ov{\square}_{\psi}^n)= D\times \ov{\square} ^{n+2}_{\psi}$, $\ov{\theta}_n ^* (F_{n,n} ^{\infty}) = F_{n+2, n+2} ^{\infty}$, while $\ov{\theta}_n ^* (F_{n,i} ^{\infty})  = F_{n+2, i} ^{\infty}$ for $1 \leq i \leq n-1$. By the defining equation of $\ov{W}_n$, we have ${\ov{\pi}}_n ^* F_{n,n} ^{\infty}  = \ov{i}_n^* F_{n+2, n+2} ^{\infty} =\ov{i}_n ^* \{ u _ 0 = 0 \} \leq \ov{i}_n ^* ( \{ u_{n,0} = 0 \} + \{ u_{n+1, 0} = 0 \} ) = \ov{i}_n ^* ( F_{n+2, n} ^{\infty} + F_{n+2, n+1} ^{\infty}).$

%First suppose $n \geq 2$. 
Thus, $\nu_{Z'} ^* \ov{\theta}_n ^* \sum_{i=1} ^n F_{n,i} ^{\infty}  = \sum_{i=1} ^{n-1} \nu_{Z'} ^* F_{n+2, i} ^{\infty} + g^* \ov{\pi}_n ^* F_{n,n} ^{\infty} $ $\leq \sum_{i=1} ^{n-1} \nu_{Z'} ^* F_{n+2, i} ^{\infty}  + g^* \ov{i}_n ^* ( F_{n+2, n} ^{\infty}  + F_{n+2, n+1} ^{\infty}) = \sum_{i=1} ^{n+1} \nu_{Z'}^* F_{n+2, i} ^{\infty} \leq \sum_{i=1} ^{n+2} \nu_{Z'} ^* F_{n+2, i} ^{\infty}.$ (In case $n=1$, we just ignore the terms with $\sum_{i=1} ^{n-1}$ in the above.)

That $Z$ has modulus $D$ means $\nu_Z ^* (D \times \ov{\square}_{\psi} ^n ) \leq \sum_{i=1} ^n \nu_Z ^*  F_{n,i} ^{\infty}$. Applying $f^* $ and using ~\eqref{eqn:mod-1}, we have $\nu^*_{Z'}(D \times \ov{\square}^{n+2}_{\psi}) = \nu_{Z'} ^* \ov{\theta}_n ^* (D \times \ov{\square}^n_{\psi}) \leq \nu_{Z'} ^* \ov{\theta}_n ^* \sum_{i=1} ^n F_{n,i} ^{\infty}$, which is bounded by $\sum_{i=1} ^{n+2} \nu_{Z'} ^* F_{n+2, i} ^{\infty}$ as we saw above. This means $Z'$ has modulus $D$. 
\end{proof}

\begin{defn}For any closed subscheme $Z \subset X \times \square_{\psi} ^n$, we define $W_n (Z):= p_{n*} i_{n*} \pi_n ^* (Z)$, which is closed in $X \times \square_{\psi} ^{n+1}$. 
\end{defn}

\begin{lem}\label{lem:pi norm}
Let $n \geq 1$. If a closed subscheme $Z \subset X \times \square_{\psi} ^n$ intersects all faces properly, then $W_n (Z)$ intersects all faces of $X \times \square_{\psi} ^{n+1}$ properly.
\end{lem}

\begin{proof}Our $W_n$ is equal to $\tau^* \tau_n ^* \tau_{n+1} ^* W_n ^X$, where $W_n ^X$ is that of \cite[Lemma 4.1]{LevineK}, and $\tau, \tau_n, \tau_{n+1}$ are the involutions ($x \mapsto 1-x$) for $y, y_n, y_{n+1}$, 
respectively. So, the lemma is a special case of \emph{loc.cit}. 
\end{proof}

\begin{proof}[Proof of Proposition \ref{prop:norm pullback}]
Consider the commutative diagram
$$
\xymatrix{
  W_n \ar[d]_{\pi_n} \ar[dr] ^{\rho_n = p_n|_{W_n}} \ar@{^{(}->}[r] ^{i_n \ \ \ \ \ } 
& X \times \square_{\psi} ^{n+1} \times \ov{\square}_{\psi} \ar[d]^{p_n} \\
 X \times \square_{\psi} ^n & X \times \square_{\psi} ^{n+1} 
\ar[l] _{q_{n, \psi}}.  }
 $$
By Lemma \ref{lem:pi_n flat}, $\rho_n$ is an isomorphism so that $\rho_{n*} i_n ^* p_n ^*= {\rm Id}$. Hence, $q_{n,\psi} ^* (Z) = \rho_{n*} i_n^* p_n ^* q_{n,\psi} ^* (Z) =^{\dagger} \rho_{n*} \pi_n ^* (Z) =^{\ddagger} p_{n*} i_{n*} \pi_n ^* (Z) = W_n (Z),$ where $\dagger$, $\ddagger$ are due to commutativity. So, we have reduced to showing that $W_n (Z) \in z_{\psi} ^q (X| D, n+1)$. But, by Lemmas \ref{lem:mod norm} and \ref{lem:pi norm}, we have $i_{n*} \pi_n ^* (Z) \in z^{q+1} _{\psi} (X \times \mathbb{P}^1|D \times \mathbb{P}^1, n+1)$. Now, for the projection $p_n$, by Lemma \ref{lem:projective image}, we have $ W_n (Z) = p_{n*} i_{n*} \pi_n ^* (Z) \in z^q _{\psi} (X |D, n+1)$. This proves Proposition \ref{prop:norm pullback}.
\end{proof}

\begin{proof}[Proof of Theorem \ref{thm:ecube}]
Since we know that $(\un{n} \mapsto z^q (X|D;n))$ is a cubical abelian group, every morphism $h: \un{r} \to \un{s}$ in \textbf{Cube} induces a morphism $h: \square^r \to \square^s$ which gives a homomorphism $h^*: z^q (X|D, s) \to z^q (X|D, r)$. Furthermore, the morphism $\mu: \un{2} \to \un{1}$ induces the morphism $q_1: \square^2 \to \square^1$ of varieties, and for each $Z \in z^q (X|D, 1)$, we have $q_1 ^* (Z)\in z^q (X|D, 2)$. Indeed, under the isomorphism $\psi: \square \simeq \mathbb{A}^1, y \mapsto 1/ (1-y)$, this is equivalent to show that $q_{1,\psi} ^*$ sends admissible cycles to admissible cycles, which we know by Proposition \ref{prop:norm pullback}.

So, it only remains to show the following ``stability under products'': if $h_i: \un{r_i} \to \un{s_i}$, $i=1,2$, are morphisms in \textbf{ECube} such that the corresponding morphisms $h_i: \square^{r_i} \to \square^{s_i}$ induce homomorphisms $h_i^*: z^q (X|D, s_i) \to z^q (X|D, r_i)$, for $i=1,2$ and all $q \geq 0$, then $h:=h_1 \times h_2: \square^{r_1 + r_2} \to \square^{s_1 + s_2}$ induces a homomorphism $h^* : z^q (X|D, s) \to z^q (X|D, r)$ for all $q \geq 0$, where $r=r_1 + r_2$ and $s=s_1 +s_2$. 

Since $h= h_1 \times h_2 = ({\rm Id}_{r_1} \times h_2) \circ (h_1 \times {\rm Id}_{r_2})$, we reduce to prove it when $h$ is either ${\rm Id}_{r_1} \times h_2$ or $h_1 \times {\rm Id}_{r_2}$. But the statement obviously holds for these cases.
%it is enough to prove the statement when either $h_1$ or $h_2$ is the identity morphism ${\rm Id}_r: \square^r \to \square^r$ for some $r \geq 0$. But, then the statement is obvious.
\end{proof}

\section{On moving lemmas}\label{sec:moving}
Let $k$ be any field. In this section, we discuss some of moving lemmas on algebraic cycles with modulus conditions. By a `moving lemma', we ask whether the inclusion $z^q _{\mathcal{W}} (Y|D, \bullet) \subset z^q (Y|D, \bullet)$ in Definition \ref{defn:complex for moving} is a quasi-isomorphism. It is known when $Y$ is smooth quasi-projective and $D=0$ (by \cite{Bl2}), and when $Y= X \times \mathbb{A}^1$, with $X$ smooth projective, $D=X \times \{ t^{m+1} = 0 \}$, and $\mathcal{W}$ consists of $W \times \mathbb{A}^1$ for finitely many locally closed subsets $W \subset X$ (by \cite{KP}). Recently, W. Kai \cite{Kai} proved it when $Y$ is smooth affine with a suitable condition. Kai's cases include the above case of $Y= X \times \mathbb{A}^1$, where $X$ is this time smooth affine. His proof applies to more general cases, possibly after Nisnevich sheafifications. 

In \S \ref{sec:Kai moving}, we sketch the argument of Kai in the case of multivariate additive higher Chow groups of smooth affine $k$-variety. In \S \ref{sec:KP moving}, we generalize the moving lemma of \cite{KP} in the case of pairs $(X \times S, X \times D)$ where $X$ is smooth projective. In \S \ref{sec:contravariant} and \ref{sec:presheaf TCH}, we discuss the standard pull-back property and its consequences. In \S \ref{sec:local moving}, we discuss a moving lemma for additive higher Chow groups of smooth semi-local $k$-schemes essentially of finite type.

\subsection{Kai's affine method for multivariate additive higher Chow groups}\label{sec:Kai moving}The moving lemma of W. Kai \cite{Kai} is the first moving result that applies to cycle groups with a \emph{non-zero} modulus over a smooth \emph{affine} scheme. We sketch the proof of the following special case on multivariate additive higher Chow groups that we need, which is a bit simpler than the general case considered by him. Following Definition \ref{defn:AHC}, we write $X[r] := X \times \mathbb{A}^r$.

\begin{thm}[W. Kai]\label{thm:moving affine}Let $X$ be a smooth affine variety over any field $k$. Let $\mathcal{W}$ be a finite set of locally closed subsets of $X$. Let $\mathcal{W}[r]:= \{ W [r]  \ | \ W \in \mathcal{W}\}$. Let $ \un{m}= (m_1, \cdots, m_r) \ge 1$. Then the inclusion $z^q _{\mathcal{W}[r]}(X[r] | D_{\un{m}} , \bullet) \hookrightarrow z^q (X[r] | D_{\un{m}}, \bullet)$ is a quasi-isomorphism.
\end{thm}

First recall some preparatory results:

\begin{lem}[{\cite[Lemma 4.5]{KP}}]\label{effective}Let $f: X \to Y$ be a dominant morphism of normal varieties. Suppose that $Y$ is integral with the generic point $\eta \in Y$, and let $X_{\eta}$ be the fiber over $\eta$, with the inclusion $j_{\eta}: X_{\eta} \inj X$. 

Let $D$ be a Weil divisor on $X$ such that $j_{\eta}^* (D) \geq 0$. Then there exists a non-empty open subset $U \subset Y$ such that $j_U^* (D) \geq 0$, where $j_U: f^{-1} (U) \inj X$ is the inclusion.
\end{lem}

The following generalizes \cite[Proposition 4.7]{KP}:

\begin{prop}[Spreading lemma]\label{prop:spread} Let $k \subset K$ be a purely transcendental extension. Let $(X,D)$ be a smooth quasi-projective $k$-scheme with an effective Cartier divisor, and let $\mathcal{W}$ be a finite collection of locally closed subsets of $X$. Let $(X_K, D_K)$ and $\mathcal{W}_K$ be the base changes via $\Spec (K) \to \Spec (k)$. Let $p_{K/k} : X_K \to X_k$ be the base change map. Then the pull-back map
$$p_{K/k} ^* : \frac{z^q (X|D, \bullet) }{z_{\mathcal{W}}^q (X|D, \bullet)} \to \frac{z^q (X_K|D_K, \bullet )}{z_{\mathcal{W}_K}^q (X_K|D_K, \bullet)}$$ is injective on homology.
\end{prop}

\begin{proof}
It is similar to\cite[Proposition 4.7]{KP}. We sketch its proof for the reader's convenience. If $k$ is finite, then we can use the standard pro-$\ell$-extension argument to reduce the proof to the case when $k$ is infinite, which we assume from now. We may also assume that ${\rm tr.deg}_k K < \infty$ and furthermore that ${\rm tr.deg}_k K =1$, by induction. So, we have $K= k (\mathbb{A}^1_k)$.

Suppose $Z \in z^q (X|D, n)$ is a cycle that satisfies $\partial Z \in z^q _{\mathcal{W}} (X|D, n-1)$, and $Z_K = \partial (B_K) + V_K$ for some $B_K \in z^q (X_K|D_K, n+1)$ and $V_K \in z^q _{\mathcal{W}_K} (X_K|D_K, n)$. Consider the inclusion $z^q (X_K|D_K, \bullet) \hookrightarrow z^q (X_K, \bullet )$. Then there is a non-empty open $U' \subset \mathbb{A}_k ^1$ such that $B_K = B_{U'}|_{\eta}, V_K= V_{U'}|_{\eta}$, $Z \times U' = \partial (B_{U'}) + V_{U'}$ for some $B_{U'} \in z^q (X \times U' , n+1)$, $V_{U'} \in z^q _{\mathcal{W} \times U'} (X \times U' , n)$, where $\eta$ is the generic point of $U'$.  Let $j_\eta: X \times \eta \to X \times U'$ be the inclusion, which is flat.

 Since $B_K, V_K$ satisfy the modulus condition, we have $j_{\eta} ^* (X \times U' \times F^1 _{n+1} - D \times U' \times \ov{\square}^{n+1}) \geq 0$ on $\ov{B}_K ^N$ and similarly for $\ov{V}_K ^N$. Furthermore, $\ov{B}_{U'} ^N \to U', \ov{V}_{U'} ^N \to U'$ are dominant. Thus by Lemma \ref{effective}, there is a non-empty open $U \subset U'$ such that $j_U^* (X \times U' \times F_{n+1} ^1 - D \times U' \times \ov{\square}^{n+1}) \geq 0$ on $\ov{B}_U ^N$ and similarly for $\ov{V}_U ^N$, for $j_U: X \times U \hookrightarrow X \times U'$. This proves that $B_U$ and $V_U$ have modulus $D \times U$. Hence, $B_U \in z^q (X \times U | D \times U, n+1)$ and $V_U \in z^q _{\mathcal{W}\times U} (X \times U|D \times U, n)$ with $Z \times U = \partial (B_U) + V_U$. 

Since $k$ is infinite, the set $U(k) \inj U$ is dense. We claim the following:

\noindent \textbf{Claim:} \emph{There is a point $u \in U(k)$ such that the pull-backs of $B_U$ and $V_U$ under the inclusion $i_u : X \times \{ u \} \inj X \times U$ are both defined in $z^q (X , n+1)$ and $z^q _{\mathcal{W} } (X , n)$, respectively.}

 Its proof requires the following elementary fact:

\noindent \textbf{Lemma:} \emph{Let $Y$ be any $k$-scheme. Let $B \in z^q (Y \times U)$ be a cycle. Then there exists a nonempty open subset $U'' \subset U$ such that for each $u \in U''(k)$, the closed subscheme $Y \times \{ u \}$ intersects $B$ properly on $Y \times U$, thus it defines a cycle $i_u ^* (B) \in z^q (Y)$, where $Y$ is identified with $Y \times \{ u \}$.}

Note that for each $u \in U(k)$, the subscheme $Y \times \{ u \} \subset Y \times U$ is an effective divisor so that its proper intersection with $B$ is equivalent to that $Y \times \{ u \}$ does not contain any irreducible component of $B$. If there exists a point $u_i \in U (k)$ such that $Y \times \{ u_i \}$ contains an irreducible component $B_i$ of $B$, then for any other $u \in U (k) \setminus \{u_i \}$, we have $(Y \times \{ u \}) \cap B_i = \emptyset$. So, for every irreducible component $B_i$ of $B$, there exists at most one $u_i \in U(k)$ such that $Y \times \{ u_i\}$ contains $B_i$. Let $S$ be the union of such points $u_i$, if they exist. Because there are only finitely many irreducible components of $B$, we have $|S|< \infty$. So, taking $U'' := U \setminus S$, we have proven \textbf{Lemma}.

We now prove \textbf{Claim}. Let $F \subset \square^{n+1}$ be any face, including the case $F= \square^{n+1}$. Since $B_U \in z^q (X\times U, n+1)$, by definition $X \times U \times F$ and $B_U$ intersect properly on $X \times U \times \square^{n+1}$, so their intersection gives a cycle $B_{U,F} \in z^q (X \times U \times F)$. By \textbf{Lemma} with $Y= X \times F$, there exists a nonempty open subset $U_F \subset U$ such that $B_{U,F}$ defines a cycle in $z^q (X \times \{u \} \times F)$ for every $u \in U_F (k)$. Let $\mathcal{U}_1:= \bigcap_F U_F$, where the intersection is taken over all faces $F$ of $\square^{n+1}$. This is a nonempty open subset of $U$. Similarly, let $F \subset \square^n$ be any face, including the case $F= \square^n$. Here, $V_U \in z^q _{\mathcal{W} \times U} (X \times U, n)$, and repeating the above argument involving \textbf{Lemma} with $Y= W \times F$ for $W \in \mathcal{W}$, we get a nonempty open subset $U_{W,F} \subset U$ such that we have an induced cycle in $z^q (W \times \{ u \} \times F)$ for every $u \in U_{W,F} (k)$. Let $\mathcal{U}_2:= \bigcap_{W, F} U_{W,F}$, where the intersection is taken over all pairs $(W,F)$, with $W \in \mathcal{W}$ and a face $F \subset \square^n$. Taking $\mathcal{U}:= \mathcal{U}_1 \cap \mathcal{U}_2$, which is a nonempty open subset of $U$, we now obtain \textbf{Claim} for every $ u \in \mathcal{U}(k)$. 

Finally, for such a point $u$ as in \textbf{Claim}, by the containment lemma (Proposition \ref{prop:CL*}), $i_u ^* (B_U)$ and $ i_u ^* (V_U)$ have modulus $D$. Hence, $i_u ^* (B_U) \in z^q (X|D, n+1)$ and $i_u ^* (V_U) \in z^q _{\mathcal{W}} (X|D, n)$. This finishes the proof.
\end{proof}

\begin{proof}[Sketch of the proof of Theorem \ref{thm:moving affine}]\textbf{Step 1}. We first show it when $X= \mathbb{A}^d_k$. Let $K= k(\mathbb{A}^d_k)$ and let $\eta \in X$ be the generic point. For simplicity, using the automorphism $y \mapsto 1/ (1-y)$ of $\mathbb{P}^1$, we replace $(\square, \{ \infty , 0 \})$ by $(\mathbb{A}^1, \{ 0, 1 \})$, and write $\square = \mathbb{A}^1$. For any $g \in \mathbb{A}^d$ and an integer $s >0$, define $\phi_{g,s}: \mathbb{A}^d _{k(g)} [r] \times_{k(g)} \square^1 _{k (g)} \to \mathbb{A}^d _{k(g)} [r]$ by $\phi_{g,s} (\un{x}, \un{t}, y) := (\un{x} + y (t_1 ^{m_1} \cdots t_r ^{m_r})^s g, \un{t})$, where $k(g)$ is the residue field of $g$. (N.B. In terms of W. Kai's homotopy, our $g \in \mathbb{A}^d$ corresponds to his $v=(g, 0, \cdots, 0) \in \mathbb{A}^d[r]= \mathbb{A}^{d+r}$.) For any cycle $V \in z^q (X[r] | D_{\un{m}}, n)$, define $H_{g, s} ^* (V):= (\phi_{g,s} \times {\rm Id}_{\square^n})^* p_{k(g)/k} ^* (V)$, where $p_{k(g)/k} : \mathbb{A}^d_{k(g)} [r] \times \square^n \to \mathbb{A}^d _k [r] \times \square^n$ is the base change map.

Using \cite[Lemma 1.2]{Bl1}, one checks that $H_{g,s} ^* (V)$ preserves the face condition for $V$. Moreover, if $V \in  z^q_{\sW}(X[r], n)$, then so does $H_{g,s} ^* (V)$. When $g = \eta$, another application of \cite[Lemma 1.2]{Bl1} shows that $H_{g,s} ^* (V)$ intersects with all $W[r] \times F$ properly, where $W \in \mathcal{W}$ and a $F \subset \square^n$ is a face. The argument for proving these face conditions follows the same steps as that of the proof of \cite[Lemma~5.5, Case~2]{KP} though the present case is slightly different so that we need to use \cite[Lemma 1.2]{Bl1} instead of \cite[Lemma 1.1]{Bl1} (see \cite[Lemma 3.5]{Kai} for more detail).

On the other hand, by \cite[Proposition 3.3]{Kai}, for each irreducible $V \in z^q (X[r]| D_{\un{m}}, n)$, there is an integer $s(V) \geq 0$ such that for any $s > s(V)$ and for any $g \in \mathbb{A}^d$, the cycle $H_{g,s} ^* (V)$ has modulus $D_{\un{m}}$. Let's call the smallest such integer $s(V)$, \emph{the threshold of $V$} for simplicity. (N.B. The existence of this number is one of the important points of W. Kai's contributions. While the reader can find its proof in \emph{loc.cit.}, the rough idea is that instead of translations by $g$ used in usual higher Chow groups, which correspond to $s=0$, W. Kai uses ``adjusted" translations as in the above definition of $\phi_{g,s}$, so that near the divisors $\{t_i = 0\}$, the effect of ``adjusted'' translation is also small, while away from the divisors $\{t_i =0\}$, the effect of ``adjusted" translation gets larger.)

So, as in \cite[\S 3.4]{Kai}, if we consider the subgroup $z^q _{\mathcal{W}[r], e}(X[r] |D_{\un{m}}, n) ^{\leq s}$  $\subset z^q _{\mathcal{W}[r], e}(X[r] | D_{\un{m}}, n)$ for $s>0$, consisting of cycles all of whose irreducible components $V$ have threshold $s(V) \leq s$, then $$\frac{z^q _{\mathcal{W}[r], e} ( X[r] |D_{\un{m}}, n)}{  z^q _{\mathcal{W}[r]} (X[r]|D_{\un{m}}, n)} = \underset{\rightarrow s}{\lim} \frac{z^q _{\mathcal{W}[r], e} ( X[r] |D_{\un{m}}, n) ^{\leq s}}{  z^q _{\mathcal{W}[r]} (X[r]|D_{\un{m}}, n) ^{\leq s}}.$$ Then one has the map 
$$H_{\eta, s}^* : \frac{z^q _{\mathcal{W}[r], e} ( X[r] |D_{\un{m}}, n) ^{\leq s}}{  z^q _{\mathcal{W}[r]} (X[r]|D_{\un{m}}, n) ^{\leq s}} \to \frac{z^q _{\mathcal{W}[r], e} ( X_K[r] |D_{\un{m}}, n+1)}{  z^q _{\mathcal{W}[r]} (X_K[r]|D_{\un{m}}, n+1)},$$
which gives a homotopy between the base change $p_{K/k} ^*$ and $H_{\eta, s} ^* |_{y_1=1}$. However, $H_{\eta, s} ^*|_{y_1 = 1}$ is zero on the quotient, while $p_{K/k} ^*$ is injective on homology by Proposition \ref{prop:spread}, after taking $s \to \infty$, so that the map $p_{K/k} ^*$ is in fact zero on homology. This means, the quotient $ z^q _{\mathcal{W}[r], e} ( X[r] |D_{\un{m}}, n) /  z^q _{\mathcal{W}[r]} (X[r]|D_{\un{m}}, n)$ is acyclic, proving the moving lemma for $X= \mathbb{A}^d_k$.

\textbf{Step 2}. If $X$ is a general smooth affine $k$-variety of dimension $d$, we use the standard generic linear projection trick. We choose a closed immersion $X \hookrightarrow \mathbb{A}^N$ for some $N \gg d$ and run the steps of \S 6 of \cite{KP} (with $\P^n$ replaced by $\A^N$ everywhere) \emph{mutatis mutandis} to conclude the proof of the moving lemma for $X$ from that of affine spaces. We leave the details for the reader.
\end{proof}

\subsection{Projective method for multivariate additive higher Chow groups}\label{sec:KP moving}
The following theorem generalizes the moving lemma for additive higher Chow groups of smooth projective schemes \cite[Theorem 4.1]{KP} to a general setting which includes the multivariate additive higher Chow groups.

\begin{thm}\label{thm:moving proj}Let $(S,D)$ be a smooth quasi-projective $k$-variety with an effective Cartier divisor. Let $X$ be a smooth projective $k$-variety. Let $\mathcal{W}$ be a finite collection of locally closed subsets of $X$. We let $\mathcal{W} \times S:= \{ W \times S | W \in \mathcal{W} \}$. Then the inclusion $z^q _{\mathcal{W} \times S} (X\times S| X \times D, \bullet) \hookrightarrow z^q (X \times S| X \times D, \bullet)$ is a quasi-isomorphism. In particular, when $\un{m}= (m_1, \cdots, m_r) \ge 1$, and $(S,D) = (\mathbb{A}^r, D_{\un{m}})$, the moving lemma holds for multivariate additive higher Chow groups of smooth projective varieties over $k$.
\end{thm}

\begin{proof}Most arguments of \cite[Theorem 4.1]{KP} work with minor changes, so we sketch the proof.

\textbf{Step 1.} We first prove the theorem when $X= \mathbb{P}^d_k$. The algebraic group $SL_{d+1, k}$ acts on $\mathbb{P}^d$. Let $K= k(SL_{d+1, k})$. Then there is a $K$-morphism $\phi: \square_K ^1 \to  SL_{d+1,K}$ such that $\phi (0) = {\rm Id},$ and $\phi (\infty) = \eta$, where $\eta$ is the generic point of $SL_{d+1, k}$. See \cite[Lemma 5.4]{KP}. For such $\phi$, consider the composition $H_n$ of morphisms
$$\mathbb{P}^d \times S \times \square_K ^{n+1} \overset{\mu_{\phi}}{\to} \mathbb{P}^d \times S \times \square_K ^{n+1} \overset{ {\rm pr}_K'}{\to} \mathbb{P}^d \times S \times \square_K ^n \overset{ p_{K/k}}{\to} \mathbb{P}^d \times S \times \square_k ^n,$$
where $\mu_{\phi} (\un{x}, s, y_1, \cdots, y_{n+1}) = (\phi(y_1) \un{x}, s, y_1, \cdots, y_{n+1})$, ${\rm pr}_K '$ is the projection dropping $y_1$, and $p_{K/k}$ is the base change. We claim that $H_n^*$ carries $z_{\mathcal{W}\times S} ^q (\mathbb{P}^d \times S| \mathbb{P}^d \times D, n)$ to $z_{\mathcal{W}\times S }^q (\mathbb{P}^d_K \times S | \mathbb{P}^d _K \times D, n+1)$, i.e., for an irreducible cycle $Z \in z_{\mathcal{W}\times S } ^q (\mathbb{P}^d \times S | \mathbb{P}^d \times S, n)$, we show that $Z':= H_n ^* (Z) \in z_{\mathcal{W}\times S }^q (\mathbb{P}^d_K \times S | \mathbb{P}^d _K \times D, n+1)$.

To do so, we first claim that $Z'$ intersects with $W\times S \times F_K$ properly for each $W \in \mathcal{W}$ and each face $F \subset \square^{n+1}$. 

(1) In case $F= \{ 0 \} \times F'$ for some face $F' \subset \square^n$, because $\phi (0) = {\rm Id}$, we have $Z' \cap (W \times S \times F_K) \simeq Z_K \cap (W \times S \times F_K ')$. Note that $\dim ( W \times S \times F_K) = \dim (W \times S \times F'_K)$. Hence, $\codim _{W \times S \times F_K} (Z' \cap (W \times S \times F_K)) = \dim (W \times S \times F_K) - \dim (Z' \cap (W \times S \times F_K)) = \dim (W \times S \times F' _K) - \dim (Z_K \cap (W \times S \times F'_K)) = \dim (W \times S \times F') - \dim (Z \cap (W \times S \times F')) = \codim_{ W \times S \times F'} (Z \cap (W \times S \times F')) \geq q$, because $Z \in z^q _{\mathcal{W} \times S} (\mathbb{P}^d \times S |\mathbb{P}^d \times D , n)$. 

(2) In case $F = \{ \infty \} \times F'$ for some face $F' \subset \square^n$, $\dim (W \times S \times F_K) = \dim (W \times S \times F'_K)$ and $Z' \cap (W \times S \times F_K) \simeq \eta \cdot (Z_K) \cap (W \times S \times F'_K)$, where $SL_{d+1,k}$ acts on $\mathbb{P}^d \times S \times F'$, naturally on $\mathbb{P}^d$ and trivially on $S \times F'$. Let $A:= W \times S \times F'$ and $B:= Z \cap (\mathbb{P}^d \times S \times F')$. Thus, $\codim_{W \times S \times F_K} (Z' \cap (W \times S \times F_K)) = \dim (W \times S \times F_K) - \dim (Z' \cap (W \times S \times F_K)) = \dim (W \times S \times  F'_K) - \dim (\eta \cdot (Z_K) \cap (W \times S \times F_K')) = ^{\dagger} \dim (A_K) - \dim (\eta \cdot B_K \cap A_K) = \codim _{A_K} (\eta \cdot B_K \cap A_K),$ where $\dagger$ holds because $Z \cap A = B \cap A$. By applying \cite[Lemma 1.1]{Bl1} to $G=SL_{d+1, k}$, and the above $A, B$ on $\mathcal{X}:= \mathbb{P}^d \times S \times F'$, there is a non-empty open subset $U \subset G$ such that for all $g \in U$, the intersection $(g  \cdot A) \cap B$ is proper on $\mathcal{X}$. By shrinking $U$, we may assume $U$ is invariant under inverse map, so $g = \eta^{-1} \in U$. Thus, $\codim_{A_K} ((\eta \cdot B_K )\cap A_K) \geq \codim_{\mathcal{X}_K} (\eta \cdot B_K)$. Since $\codim_{\mathcal{X}_K} (\eta \cdot B_K) = \codim _{\mathcal{X}_K} B_K$ and $\codim_{\mathcal{X}_K} B_K = q$, we get $\codim_{W \times S \times F_K} (Z' \cap ( W \times S \times F_K)) = \codim_{A_K} ((\eta \cdot B_K) \cap A_K) \geq \codim _{\mathcal{X}_K} B_K = q$.

(3) In case $F = \square \times F'$ for some face $F' \subset \square^n$, the projection $Z' \cap (W \times S \times \square \times F'_K) \to \square_K$ is flat, being a dominant map to a curve, so $\dim ( Z' \cap (W \times S \times \square \times F'_K)) = \dim (Z' \cap (W \times S \times \{ \infty \} \times F'_K)) + 1$. We also have $\dim (W \times S \times \square \times F'_K) = \dim (W \times S \times \{ \infty \} \times F'_K)+1$. Hence, we deduce $\codim_{W \times S \times F_K} (Z' \cap (W \times S \times F_K)) = \dim (W \times S \times \square \times F_K) - \dim (Z' \cap (W \times S \times \square \times F'_K)) = \codim _{W \times S \times \{ \infty \} \times F'_K} (Z' \cap (W \times S  \times \{ \infty \} \times F'_K)) \geq^{\dagger} q$, where $\dagger$ follows from case (2). This shows $Z'$ intersects all faces properly.

Now we show that $Z'$ has modulus $\mathbb{P}^d \times D$. We drop all exchange of the factors, for simplicity. For $p: \mathbb{P}^d\to \Spec (k)$, we take $V= p(Z)$ on $S \times \square^n$. Because $Z \subset p^{-1} (p(Z)) = \mathbb{P}^r \times V$, we have $Z' = \mu_{\phi} ^* (\square^1_K\times Z ) \subset \mu_{\phi} ^* (\mathbb{P}^d \times \square_K ^1 \times V) = \mathbb{P}^d \times \square_K ^1 \times V :=Z_1$. By Lemma \ref{lem:projective image}, $V$ is admissible on $S \times \square^n$. So, $p^* [V] = \mathbb{P}^d \times V$ is admissible on $\mathbb{P}^d \times S \times \square^n$. In particular, $\mathbb{P}^d \times V$ has modulus $\mathbb{P}^d \times D$. Hence, $Z_1 = \mathbb{P}^d \times \square_K ^1 \times V$ also has modulus $\mathbb{P}^d _K  \times D$. Now, $Z' \subset Z_1$ shows that $Z'$ has modulus $\mathbb{P}^d_K  \times D$ by Proposition \ref{prop:CL*}. Thus, we proved $Z' \in z_{\mathcal{W}\times S }^q (\mathbb{P}^d_K \times S | \mathbb{P}^d _K \times D, n+1)$.

Going back to the proof, one checks that $H_{\bullet} ^*: z^q (\mathbb{P}^d \times S | \mathbb{P}^d \times D, \bullet) \to z^q (\mathbb{P}^d_K \times S  | \mathbb{P}^d \times D, \bullet+1)$ is a chain homotopy satisfying $\partial H^* (Z)  + H^* \partial (Z) = Z_K - \eta \cdot (Z_K)$, and the same holds for $z_{\mathcal{W}\times S}$ by a straightforward computation (see \cite[Lemma 5.6]{KP}). Furthermore, for each admissible $Z$, we have $\eta \cdot Z_K \in z^q _{\mathcal{W}_K \times S } (\mathbb{P}_K ^d \times S | \mathbb{P}^d _K \times D, n)$, by the above proof of proper intersection of $Z'$ with $W \times S \times F_K$, where $F = \{ \infty \} \times F'$ for a face $F' \subset \square^n$. Hence, the base change $p_{K/k} ^* : z^q (\mathbb{P}_k ^d \times S | \mathbb{P}_k ^d \times D, \bullet) / z^q_{\mathcal{W} \times S } (\mathbb{P}^d_k \times S |\mathbb{P}^d _k \times D,  \bullet) \to z^q (\mathbb{P}_K ^d \times S | \mathbb{P}^d _K \times D, \bullet) / z^q _{\mathcal{W}_K \times S } (\mathbb{P}_K ^d \times S | \mathbb{P}_K ^d \times D, \bullet)$ is homotopic to $\eta \cdot p_{K/k} ^*$, which is zero on the quotient. That is, $p_{K/k} ^*$ on the above quotient complex is zero on homology. However, by the spreading argument (Proposition \ref{prop:spread}), $p_{K/k} ^*$ is injective on homology. (N.B. We used here an elementary fact that $k(SL_{d+1,k})$ is purely transcendental over $k$. To check this fact, first note that by definition $k[SL_{d+1,k}]\simeq k[\{T_{i,j}| 1 \leq i,j \leq d+1 \}]/ (\det (M) -1)$ for the $(d+1, d+1)$-matrix $M=[T_{ij}]$ consisting of indeterminates $T_{i,j}$ for $1 \leq i,j \leq d+1$. Here by Cramer's rule we can write $\det(M) -1 = \alpha T_{d+1, d+1} - \beta -1$, where $\alpha = \det (M_{d+1, d+1})$, $\beta = \sum_{1 \leq j \leq d} (-1) ^{d+1+j} \det (M_{d+1, j})$ and $M_{ij}$ is the $(i,j)$-minor of $M$. Here both $\alpha$ and $ \beta$ do not have $T_{d+1, d+1}$. Hence $k[SL_{d+1,k}]\simeq k[\{T_{ij}|1\leq i,j \leq d+1, (i,j) \not = (d+1, d+1)\}, \frac{ \beta +1}{\alpha}]$. Thus, $k(SL_{d+1, k})\simeq k(\{T_{ij}|1\leq i,j \leq d+1, (i,j) \not = (d+1, d+1)\})$, which is purely transcendental over $k$.) Hence, the quotient complex $z^q (\mathbb{P}^d \times S |\mathbb{P}^d \times D, \bullet)/ z^q _{\mathcal{W}\times S} (\mathbb{P}^d \times S| \mathbb{P}^d \times D,  \bullet)$ is acyclic, i.e., the moving lemma holds for $(\mathbb{P}^d \times S, \mathbb{P}^d \times D)$, finishing Step 1.

 \textbf{Step 2.} Now let $X$ be a general smooth projective variety of dimension $d$. In this case, we choose a closed immersion $X \hookrightarrow \mathbb{P}^N$ for some $N \gg d$. We now run the linear projection argument of \cite[\S 6]{KP} again without any extra argument to deduce the proof of the moving lemma for $X$ from that of the projective spaces. We leave out the details.
\end{proof}

\subsection{Contravariant functoriality}\label{sec:contravariant}
The following contravariant functoriality of multivariate additive higher Chow groups is an immediate application of the moving lemma and the proof is identical to that of \cite[Theorem~7.1]{KP}.

\begin{thm}\label{thm:contravariant} Let $f:X \to Y$ be a morphism of $k$-varieties, with $Y$ smooth affine or smooth projective. Let $r \geq 1$ and $\un{m} = (m_1, \cdots, m_r) \ge 1$. Then there exists a pull-back $f^*: \CH^q (Y[r] |D_{\un{m}},n) \to \CH^q (X[r] | D_{\un{m}}, n)$.

If $g: Y \to Z$ is another morphism with $Z$ smooth affine or smooth projective, then we have $(g\circ f)^* = f^* \circ g^*$.
\end{thm}

\begin{remk}As a special case, when $r=1$, we have the pull-back map $f^*: \TH ^q (Y, n; m) \to \TH^q (X, n; m)$.
\end{remk}

\subsection{The presheaf $\mathcal{TCH}$}\label{sec:presheaf TCH}For the rest of the section, we concentrate on additive higher Chow groups. Let $m \geq 0$. By Theorem \ref{thm:contravariant}, we see that $T^q_{n,m} := \TH^q (-, n; m)$ is a presheaf of abelian groups on the category $\SmAff_k$, but we do not know if it is a presheaf on the categories $\Sm_k$ or $\Sch_k$. However, we can exploit Theorem \ref{thm:contravariant} further to {define} a new presheaf on $\Sm_k$ and $\Sch_k$. The idea of this detour occurred to the authors while working on \cite{sst}. We do it for somewhat more general circumstances.

Let $\mathcal{C}$ be a category and $\mathcal{D}$ be a full subcategory. Let $F$ be a presheaf of abelian groups on $\mathcal{D}$, i.e. $F: \mathcal{D}^{\op} \to (\textbf{Ab})$ is a functor to the category of abelian groups. %(N.B. For the target category $(\textbf{Ab})$, we may more generally use a cocomplete category, too.) 
For each object $X \in \mathcal{C}$, let $(X \downarrow \mathcal{D})$ be the category whose objects are the morphisms $X \to A$ in $\mathcal{C}$, with $A \in \mathcal{D}$, and a morphism from $h_1 : X \to A$ to $h_2: X \to B$, with $A,B \in \mathcal{D}$, is given by a morphism $g: A \to B$ in $\mathcal{C}$ such that $g \circ h_1 = h_2$. The functor $F: \mathcal{D}^{\op} \to (\textbf{Ab})$ induces the functor $(X \downarrow \mathcal{D})^{\op} \to (\textbf{Ab})$ given by $(X\overset{h}{\to} A) \mapsto F(A)$, also denoted by $F$. 

\begin{defn}\label{defn:tch_colimit}Suppose that for each $X \in \mathcal{C}$, the category $(X \downarrow \mathcal{D})$ is cofiltered. Then define $\mathcal{F} (X):= \underset{(X \downarrow \mathcal{D})^{\op}}{\colim} F$. 

In particular, when $\mathcal{C}= \Sch_k$ and $\mathcal{D} = \SmAff_k$, one checks that $(X \downarrow \SmAff_k)$ is cofiltered, and for $X \in \Sch_k$, we define $\tch ^q (X,n;m):= \underset{(X \downarrow \SmAff_k)^{\op}}{\colim} T^q _{n , m}$.
\end{defn}

%Let $X \in \Sch_k$. Let $(X \downarrow \SmAff_k)$ be the category whose objects are the $k$-morphisms $X \to A$, with $A \in \SmAff_k$, and a morphism from $h_1: X \to A$ to $h_2: X \to B$, with $A, B \in \SmAff_k$, is given by a $k$-morphism $g: A \to B$ such that $g\circ h_1 = h_2$. One checks that $(X \downarrow \SmAff_k)$ is cofiltered.

%The functor $T^q _{n,m} : \SmAff_k ^{\op} \to (\mathbf{Ab})$ induces the functor $ (X \downarrow \SmAff_k)^{\op} \to (\mathbf{Ab})$, $(X \xrightarrow{h} A) \mapsto T_{n,m} ^q (A)$, also denoted by $T^q _{n, m}$. Here, $(\mathbf{Ab})$ is the category of abelian groups.

%\begin{defn}\label{defn:tch_colimit}
%For $X \in \Sch_k$, define $\tch ^q (X,n;m):= \underset{(X \downarrow \SmAff_k)^{\op}}{\colim} T^q _{n , m}$, which is a filtered colimit.
%\end{defn}

\begin{prop}\label{prop:presh}Let $\mathcal{C}$ be a category and $\mathcal{D}$ be a full subcategory such that for each $X \in \mathcal{C}$, the category $(X \downarrow \mathcal{D})$ is cofiltered. Let $F$ be a presheaf of abelian groups on $\mathcal{D}$ and let $\mathcal{F}$ be as in Definition \ref{defn:tch_colimit}.

Let $f: X \to Y$ be a morphism in $\mathcal{C}$. Then for $X \in \mathcal{C}$, the association $X \mapsto \mathcal{F}(X)$ satisfies the following properties:
\begin{enumerate}
\item There is a canonical homomorphism $\alpha_X: \mathcal{F} (X) \to F(X)$.
\item If $X \in \mathcal{D}$, then $\alpha_X$ is an isomorphism, and $\alpha: \mathcal{F} \to F$ defines an isomorphism of presheaves on $\mathcal{D}$.
\item There is a canonical pull-back $f^* : \mathcal{F}(Y) \to \mathcal{F}(X)$. If $g: Y \to Z$ is another morphism in $\mathcal{C}$, then we have $(g \circ f)^* = f^* \circ g^*$. So, $\mathcal{F}$ is a presheaf of abelian groups on $\mathcal{C}$. In particular, $\tch^q (-, n;m)$ is a presheaf of abelian groups on $\Sch_k$, which is isomorphic to $\TH^q (-, n;m)$ on $\SmAff_k$.
\end{enumerate}
\end{prop}

\begin{proof}(1) Let $(X \xrightarrow{h} A) \in (X \downarrow \mathcal{D})^{\op}$. By the given assumption, we have the pull-back $h^* : F(A) \to F(X)$. Regarding $F(X)$ as a constant functor on $(X\downarrow \mathcal{D})^{\op}$, this gives a morphism of functors $F\to F(X)$. Taking the colimits over all $h$, we obtain $\mathcal{F} (X) \to F(X)$, where $\alpha_X = \colim_h h^*$.

(2) When $X \in \mathcal{D}$, the category $(X \downarrow \mathcal{D})^{\op}$ has the terminal object ${\rm Id}_X: X \to X$. Hence, the colimit $\mathcal{F} (X)$ is just $F(X)$. 

(3) A morphism $f: X \to Y$ in $\mathcal{C}$ defines a functor $f^{\sharp}: (Y \downarrow \mathcal{D})^{\op} \to (X \downarrow \mathcal{D})^{\op}$ given by $( Y\overset{h}{ \to} A) \mapsto (X \overset{f}{\to} Y \overset{h}{\to} A).$ Thus, taking the colimits of the functors induced by $F$, we obtain $f^*: \mathcal{F} (Y) \to \mathcal{F} (X)$. For another morphism $g: Y \to Z$, that $(g \circ f)^* = f^* \circ g^*$ can be checked easily using the universal property of the colimits. 

In the special case when $\mathcal{C} = \Sch_k$ and $\mathcal{D} = \SmAff_k$ with $F = \TH^q (-, n;m)$, by Theorem \ref{thm:contravariant} we know that $F$ is a presheaf on $\SmAff_k$. So, the above general discussion holds.
\end{proof}

\begin{remk}
Since additive higher Chow groups have pull-backs for flat maps (see \cite[Lemma 4.7]{KL}), it follows that for $X \in \Sm_k$, $\alpha_{(-)}$ defines a map of presheaves $\tch^q (-, n;m) \to \TH^q (-, n;m)$ on the small Zariski site $X_{\rm Zar}$ of $X$. Proposition \ref{prop:presh}(2) says that this map is an isomorphism for affine open subsets of $X$. Thus, this map of presheaves on $X_{\rm Zar}$ induces an isomorphism of their Zariski sheafifications.
\end{remk}

\subsection{Moving lemma for smooth semi-local schemes}\label{sec:local moving}
One remaining objective in Section \ref{sec:moving} is 
%to define additive higher Chow cycles over $k$-schemes essentially of finite type, and 
to prove the following semi-local variation of Theorem \ref{thm:moving affine}:

\begin{thm}\label{thm:moving local}Let $Y \in \SmLoc_k$. Let $\mathcal{W}$ be a finite set of locally closed subsets of $Y$. Then the inclusion $\TZ^q _{\mathcal{W}} (Y, \bullet;m ) \hookrightarrow \TZ^q (Y, \bullet;m)$ is a quasi-isomorphism.
\end{thm}

We begin with some basic results related to cycles over semi-local schemes. Recall that when $A$ is a ring and $\Sigma= \{ p_1, \cdots, p_N \}$ is a finite subset of $\Spec (A)$, the localization at $\Sigma$ is the localization $A \to S^{-1} A$, where $S= \bigcap_{i=1} ^N (A \setminus p_i)$. For a quasi-projective $k$-scheme $X$ and a finite subset $\Sigma$ of (not necessarily closed) points of $X$, the localization $X_{\Sigma}$ is defined by reducing it to the case when $X$ is affine by the following elementary fact (see \cite[Proposition 3.3.36]{Liu}) that we use often.

\begin{lem}\label{lem:FA} Let $X$ be a quasi-projective $k$-scheme. Given any finite subset $\Sigma \subset X$ and an open subset $U \subset X$ containing $\Sigma$, there exists an affine open subset $V \subset U$ containing $\Sigma$.
\end{lem}

For $X\in \Sch_k$ and a point $x \in X$, the open neighborhoods of $x$ form a cofiltered category and we have functorial flat pull-back maps $(j_U^V)^*: \un{\TZ}^q (V, n; m) \to \un{\TZ}^q (U, n;m)$ for $j^V _U : U \hookrightarrow V$ in this category. 

\begin{lem}\label{lem:local spread 0} Let $X \in \Sch_k$ and let $x \in X$ be a scheme point. Let $Y= \Spec (\mathcal{O}_{X,x})$. Then we have $\colim_{ x \in U} \un{\TZ}^q (U, n;m) \xrightarrow{\simeq} \un{\TZ}^q (Y, n;m)$, where the colimit is taken over all open neighborhoods $U$ of $x$.
\end{lem}

\begin{proof}Replacing $x$ by an affine open neighborhood of $x \in X$, we may assume that $X$ is affine and write $X= \Spec (A)$. Let $\mathfrak{p}_x \subset A$ be the prime ideal that corresponds to the point $x$ and let $S:= A \setminus \mathfrak{p}_x$, so that $Y= \Spec (S^{-1} A)$. To facilitate our proof, using the automorphism $y \mapsto 1/ (1-y)$ of $\mathbb{P}^1$, we identify $\square$ with $\mathbb{A}^1$ and take $\{0, 1 \} \subset \mathbb{A}^1$ as the faces. So, $X \times B_n = X \times \mathbb{A}^1 \times \mathbb{A}^{n-1} = \Spec (A[ t, y_1, \cdots, y_{n-1}])$. 

Let $\alpha \in \un{\TZ} ^q (Y, n;m)$. We need to find an open subset $U \subset X$ containing $x$ such that the closure of $\alpha$ in $U \times \mathbb{A}^1 \times \mathbb{A}^{n-1}$ is admissible. For this, we may assume $\alpha$ is irreducible, i.e., it is a closed irreducible subscheme $Z \subset Y \times \mathbb{A}^1 \times \mathbb{A}^{n-1}$. Let $\ov{Z}$ be its Zariski closure in $X \times \mathbb{A}^1 \times \mathbb{A}^{n-1}$. Let $\mathfrak{p}$ be the prime ideal of $B:=A[t, y_1, \cdots, y_{n-1}]$ such that $V(\mathfrak{p}) = \ov{Z}$.

For the proper intersection with faces, let $\mathfrak{q}\subset B$ be the prime ideal $(y_{i_1} - \epsilon_1, \cdots, y_{i_s} - \epsilon_s)$, where $1 \leq i_1< \cdots < i_s \leq n-1$ and $\epsilon_j \in \{ 0, 1 \}$. Let $\mathfrak{P}$ be a minimal prime of $\mathfrak{p} + \mathfrak{q}$. One checks immediately from the behavior of prime ideals under localizations that there is $a \in S$ such that either $\mathfrak{P} B[a^{-1}] = B[a^{-1}]$ or ${\rm ht} (\mathfrak{P} B[a^{-1}]) \geq q + s$. This means, over $U_{\mathfrak{q}}:= \Spec (A[a^{-1}])$, either the intersection of $\ov{Z}_{U_\mathfrak{q}}$ with $V (\mathfrak{q})$ is empty, or has codimension $\geq q + s$. Applying this argument to all faces, we can take $U_1:= \bigcap_{\mathfrak{q}} U_{\mathfrak{q}}$. Then $\ov{Z}_{U_1}$ intersects all faces of $U_1 \times \mathbb{A}^1 \times \mathbb{A}^{n-1}$ properly.

For the modulus condition, let $\nu: \widehat{Z}^N \to \widehat{Z} \hookrightarrow X \times \mathbb{P}^1 \times (\mathbb{P}^1)^{n-1}$ be the normalization composed with the closed immersion of the further Zariski closure $\widehat{Z}$ of $\ov{Z}$. Let $F_n^{\infty} = \sum_{i=1} ^{n-1} \{ y_i = \infty \}$ be the divisor at infinity. For an open set $j: U \hookrightarrow X$, the modulus condition of $\ov{Z}_U$ means $(m+1) [ j^* \nu^* \{ t=0  \}] \leq [ j^* \nu^* (F_n ^{\infty})]$ on $\widehat{Z}_U ^N$. Note that there exist only finitely many prime Weil divisors $P_1, \cdots, P_\ell$ on $\widehat{Z}^N$ such that ${\rm ord} _{P_i} ( \nu^* (F_n ^{\infty}) - (m+1) \nu^* \{ t = 0 \}) < 0$. Their images $Q_i$ under the normalization map $\widehat{Z}^N \to \widehat{Z}$ are still irreducible proper closed subsets of $\widehat{Z}$, thus of $X \times \mathbb{P}^1 \times (\mathbb{P}^1)^{n-1}$. Since $Z= \ov{Z}_Y$ has the modulus condition on $Y \times B_n$ by the given assumption, we have $(Y \times \widehat{B}_n) \cap Q_i = \emptyset$ for each $1 \leq i \leq \ell$. Thus, there is an affine open subset $U_2 \subset X$ containing $x$ such that $(U_2 \times \widehat{B}_n) \cap Q_i= \emptyset$ for each $1 \leq i \leq \ell$. Now, by construction, $\ov{Z}_{U_2}$ on $U_2 \times B_n$ satisfies the modulus condition. So, taking an affine open subset $U \subset U_1 \cap U_2$ containing $x$, we have $\ov{Z}_U \in \un{\TZ} ^q (U, n;m)$. That $(\ov{Z}_U)_Y= Z$ is obvious. 
\end{proof}

We can extend this colimit description to semi-local schemes:

\begin{lem}\label{lem:local spread}Let $Y$ be a semi-local $k$-scheme obtained by localizing at a finite set $\Sigma$ of scheme points of a quasi-projective $k$-variety $X$. For a cycle $Z$ on $Y \times B_n$, let $\ov{Z}$ be its Zariski closure in $X \times B_n$.

Then $Z \in \un{\TZ}^q (Y, n;m)$ if and only if there exists an affine open subset $U \subset X$ containing $\Sigma$, such that $\ov{Z}_U \in \un{\TZ}^q (U, n;m)$, where $\ov{Z}_U$ is the pull-back of $ \ov{Z}$ via the open immersion $U \to X$.
\end{lem}

\begin{proof}
The direction $(\Leftarrow)$ is obvious by pulling back via the flat morphism $Y \hookrightarrow U$. For the direction $(\Rightarrow)$, by Lemma \ref{lem:local spread 0},  for each $x \in \Sigma$ we have an affine open neighborhood $U_x \subset X$ of $x$ such that $\ov{Z}_{U_x} \in \un{\TZ}^q (U_x, n;m)$. Take $W= \bigcup _{x \in \Sigma} U_x$. This is an open subset of $X$ containing $\Sigma$. By Lemma \ref{lem:admislocal}, we have $\ov{Z}_W \in \un{\TZ} ^q (W, n;m)$. On the other hand, by Lemma \ref{lem:FA}, there exists an affine open subset $U \subset W$ containing $\Sigma$. By taking the flat pull-back via the open immersion $U \hookrightarrow W$, we get $\ov{Z}_U \in \un{\TZ} ^q (U, n;m)$.
\end{proof}

\begin{lem}\label{lem:bdry spread}Let $Y$ be a semi-local integral $k$-scheme obtained by localizing at a finite set $\Sigma$ of scheme points of an integral quasi-projective $k$-scheme $X$. Let $Z \in \TZ^q (Y, n;m)$, $W \in \TZ^q (Y, n+1;m)$, and let $\ov{Z}$, $\ov{W}$ be their Zariski closures in $X \times B_n$ and $X \times B_{n+1}$, respectively. For every open subset $U\subset X$, the subscript $U$ means the pull-back to $U$. Then we have the following:
\begin{enumerate}
\item If $\partial Z = 0$, we can find  an affine open subset $U \subset X$ containing $\Sigma$ such that $\ov{Z}_U \in \TZ^q (U, n;m)$ and $\partial \ov{Z}_U = 0$. 
\item If $Z= \partial W$, we can find an affine open subset $U \subset X$ containing $\Sigma$ such that $\ov{Z}_U \in \TZ^q (U, n;m)$, $\ov{W}_U \in \TZ^q (U, n+1;m)$ and $\ov{Z}_U = \partial \ov{W}_U$.
\end{enumerate}
\end{lem}

\begin{proof}Note that (1) is a special case of (2), so we prove (2) only. Let $Z':= \ov{Z} - \partial \ov{W} \in z^q (X \times B_n)$. If $Z'$ is $0$ as a cycle, then take $U_0=X$. If not, let $Z_1', \cdots, Z_s '$ be the irreducible components of $Z'$. Since $Z = \partial W$, each component $Z'_i$ has empty intersection with $Y \times B_n$. So, each $\pi ((Z_i ')^c)$ is a non-empty open subset of $X$ containing $\Sigma$, where $\pi: X \times B_n \to X$ is the projection, which is open. Take $U_0 = \bigcap_{i=1} ^s \pi ( (Z_i ')^c)$. %(If $Z'=0$, then we take $U_0 = X$.)

On the other hand, Lemma \ref{lem:local spread} implies that there exist open sets $U_1, U_2 \subset X$ containing $\Sigma$ such that $\ov{Z}_{U_1} \in \TZ^q (U_1, n;m)$ and $\ov{W}_{U_2} \in \TZ^q (U_2, n+1;m)$. Choose an affine open subset $U \subset U_0 \cap U_1 \cap U_2$ containing $\Sigma$, using Lemma \ref{lem:FA}. Then part (2) holds over $U$ by construction.
\end{proof}

\begin{proof}[{Proof of Theorem \ref{thm:moving local}}]
We show that the chain map $\TZ^q _{\mathcal{W}} (Y, \bullet;m ) \hookrightarrow \TZ^q (Y, \bullet;m)$ is a quasi-isomorphism. Let $X$ be a smooth affine $k$-variety with a finite subset $\Sigma\subset X$ such that $Y = \Spec (\mathcal{O}_{X, \Sigma})$. 

For surjectivity on homology, let $Z \in \un{\TZ}^q (Y, n;m)$ be such that $\partial Z = 0$. Let $\ov{Z}$ be the Zariski closure of $Z$ in $X \times B_n$. Here, $\partial \ov{Z}$ may not be zero, but by Lemma \ref{lem:bdry spread}(1), there exists an affine open subset $U\subset X$ containing $\Sigma$ such that we have $\partial \ov{Z}_U =0$, where $\ov{Z}_U$ is the pull-back of $\ov{Z}$ to $U$. Let $\mathcal{W}_U = \{ {W}_U | W \in \mathcal{W} \}$, where $W_U$ is the Zariski closure of $W$ in $U$. Then the quasi-isomorphism $\TZ_{\mathcal{W}_U} ^q (U, \bullet; m) \hookrightarrow \TZ^q (U, \bullet;m)$ of Theorem \ref{thm:moving affine} shows that there are some $C \in \TZ^q (U, n+1;m)$ and $Z'_U \in \TZ^q _{\mathcal{W}_U} (U, n;m)$ such that $\partial C = \ov{Z}_U - Z'_U$. Let $\iota: Y \hookrightarrow U$ be the inclusion. So, by applying the flat pull-back $\iota^*$ (which is equivariant with respect to taking faces), we obtain $\partial (\iota ^* C) = Z - \iota ^* Z'_U$, and here $\iota^* Z'_U \in \TZ^q _{\mathcal{W}} (Y, n;m)$, \emph{i.e.}, $Z$ is equivalent to a member in $\TZ^q _{\mathcal{W}} (Y, n;m)$. 

For injectivity on homology, let $Z \in \TZ^q _{\mathcal{W}} (Y, n;m)$ be such that $ Z = \partial Z'$ for some $Z' \in \TZ^q (Y, n+1;m)$. Let $\ov{Z}$ and $\ov{Z}'$ be the Zariski closures of $Z$ and $Z'$ on $X \times B_n$ and $X \times B_{n+1}$, respectively. Then by Lemma \ref{lem:bdry spread}(2), there exists a nonempty open affine subset $U \subset X$ containing $\Sigma$ such that $\ov{Z}_U = \partial \ov{Z}'_U$. Then the quasi-isomorphism $\TZ_{\mathcal{W}_U} ^q (U, \bullet; m) \hookrightarrow \TZ^q (U, \bullet;m)$ of Theorem \ref{thm:moving affine} shows that there exists $Z'' \in \TZ^q_{\mathcal{W}_U}  (U, n+1;m)$ such that $\ov{Z}_U = \partial Z''$. Pulling back via $\iota:Y \hookrightarrow U$ then shows $Z= \partial (\iota^* Z'')$, with $\iota^* Z'' \in \TZ^q _{\mathcal{W}} (Y, n+1;m)$. %\qed
%Thus, we proved the theorem.
\end{proof}

Using an argument identical to Theorem \ref{thm:contravariant} (see \cite[Theorem 7.1]{KP}), we get:

\begin{cor}\label{cor:contra local}Let $f:Y_1 \to Y_2$ be a morphism in $\Sch^{\rm ess}_k$, where $Y_2 \in \SmLoc_k$. Then there is a natural pull-back $f^*: \TH^q (Y_2 , n;m) \to \TH^q (Y_1, n;m)$. 
\end{cor}

\section{The Pontryagin product}\label{sec:DGA}
Let $R$ be a commutative ring and let $(A,d_A)$ be a differential graded algebra over $R$. Recall that (\emph{left}) differential graded module $M$ over $A$ is a left $A$-module $M$ with a grading $M = \oplus_{ n \in \mathbb{Z}} M_n$ and a differential $d_M$ such that $A_mM_n \subset M_{m+n}$, $d_M(M_n) \subset M_{n+1}$ and $d_M(ax) = d_A(a) x + (-1)^n a d_M(x)$ for $a \in A_n$ and $x \in M$. A homomorphism of differential graded modules $f:M \to N$ over $A$ is an $A$-module map which is compatible with gradings and differentials. %The category of differential graded modules over $A$ is a complete and cocomplete abelian category.

In this section, we show that the multivariate additive higher Chow groups have a product structure that resembles the Pontryagin product. We construct a differential operator on these groups in the next section and show that the product and the differential operator together turn multivariate additive higher Chow groups groups into a differential graded module over $\W_m\Omega^{\bullet}_R$ for suitable $m$, when $X = \Spec(R)$ is in $\SmAff_k ^{\ess}$. This generalizes the DGA-structure on additive higher Chow groups of smooth projective varieties in \cite{KP2}. The base field $k$ is perfect in this section.

\subsection{Some cycle computations}\label{sec:cycle computations}
We generalize some of \cite[\S 3.2.1, 3.2.2, 3.3]{KP2}. Let $(X,D)$ be a $k$-scheme with an effective divisor.

Recall that a permutation $\sigma\in \mathfrak{S}_n$ acts naturally on $\square^n$ via $\sigma (y_1, \cdots, y_n) := (y_{\sigma(1)}, \cdots, y_{\sigma(n)})$. This action extends to cycles on $X \times \square^n$ and $X \times \ov{\square}^n$.

Let $ n, r \geq 1$ be given. Consider the finite morphism $\chi_{n,r} : X \times \square^n \to X \times \square^n$ given by $(x, y_1, \cdots, y_n) \mapsto (x, y_1 ^r, y_2, \cdots, y_n)$. Given an irreducible cycle $Z \subset X \times \square^n$, define $Z\{r \}:= (\chi_{n,r})_* ([Z]) = [k (Z): k (\chi_{n,r} (Z))] \cdot [ \chi_{n,r} (Z)].$ We extend it $\mathbb{Z}$-linearly.

\begin{lem} \label{lem:brace} If $Z$ is an admissible cycle with modulus $D$, then so is $Z\{r \}$.
\end{lem}

\begin{proof}The proof is almost identical to that of \cite[Lemma 3.11]{KP2}, except that the divisor $(m+1) \{ t = 0 \}$ there should be replaced by $D \times \ov {\square}^n$. We give its argument for the reader's convenience.

We may assume $Z$ is irreducible. It is enough to show that $\chi_{n,r} (Z)$ is admissible with modulus $D$. We first check that it satisfies the face condition of Definition \ref{defn:partial complex}. When $n=1$, the proper faces of $\square$ are of codimension $1$, and for $\epsilon \in \{ 0, \infty\}$, we have $\partial_1 ^{\epsilon} (\chi_{n,r} (Z)) = r \partial_1 ^{\epsilon} (Z)$. When $n \geq 2$, for $\epsilon \in \{ 0, \infty\}$, we have $\partial_1 ^{\epsilon} (\chi_{n,r} (Z)) = r \partial_1 ^{\epsilon} (Z)$ and $\partial_i ^{\epsilon} (\chi_{n,r} (Z)) =\chi_{n-1, r} (\partial _i ^{\epsilon} (Z))$ if $ i \geq 2$. For faces $F \subset \square^n$ of higher codimensions, we consequently have $F \cdot (\chi_{n,r} (Z)) = r (F \cdot Z)$ if $F$ involves the equations $\{ y_1 = \epsilon \}$, and $F \cdot (\chi_{n,r} (Z)) = \chi_{n-c, r} (F\cdot Z)$, otherwise, where $c$ is the codimension of $F$. Since the intersection $F \cdot Z$ is proper, so is $\chi_{n-c, r} (F \cdot Z)$ by induction on the codimension of faces. This shows $\chi_{n,r} (Z)$ satisfies the face condition.

To show that $W:= \chi_{n,r} (Z)$ has modulus $D$, consider the commutative diagram
$$
\xymatrix{ \ov{Z}^N \ar[r] ^{\nu_Z}  \ar[d] ^{\ov{\chi}_{n,r} ^N} & \ov{Z} \ar[d] ^{\ov{\chi}_{n,r}} \ar@{^{(}->}[r] ^{\iota_Z \ \ \ } & X \times \ov{\square}^n \ar[d] ^{\chi_{n,r}}\\
\ov{W}^N \ar[r]^{\nu_W} & \ov{W} \ar@{^{(}->}[r] ^{\iota_W \ \ \ } & X \times \ov{\square}^n,}
$$
where $\ov{Z}$, $\ov{W}$ are the Zariski closures of $Z$ and $W$ in $X \times \ov{\square}^n$ and $\nu_Z, \nu_W$ are the respective normalizations. The morphisms $\chi_{n,r}$, $\ov{\chi}_{n,r}$ are the natural induced maps, and $\ov{\chi}_{n,r} ^N$ is induced by the universal property of normalization. Since $Z$ has modulus $D$, we have the inequality 
\begin{equation}\label{eqn:brace}
[\nu_Z^* \iota_Z ^* (D \times \ov{\square}^n)] \leq \sum_{i=1} ^n [\nu_Z^* \iota_Z ^*\{ y_i = 1 \}] %\leq r[\nu_Z ^* \iota_Z^* \{ y_1 = 1 \}] + \sum_{i=2} ^n [\nu_Z^* \iota_Z ^*\{ y_i = 1 \}].
\end{equation}
By the definition of $\chi_{n,r}$, we have $\chi_{n,r} ^* (D \times \ov{\square}^n) = D \times \ov{\square}^n$, $\chi_{n,r} ^* \{ y_1 = 1 \} \geq \{ y_1 = 1\}$, and $\chi_{n,r} ^* \{ y_i = 1 \} = \{ y_i = 1 \}$ for $i \geq 2$. Hence \eqref{eqn:brace} implies that $[\nu_Z^* \iota_Z ^* \chi_{n,r} ^* (D \times \ov{\square}^n)] \leq \sum_{i=1} ^n [\nu_Z^* \iota_Z ^* \chi_{n,r} ^*\{ y_i = 1 \}]$. By the commutativity of the diagram, this implies that ${\ov{\chi}_{n,r} ^N}^* \left( \sum_{i=1} ^n \nu_W ^* \iota_W ^* \{ y_i = 1 \} - \nu_W ^* \iota_W ^* (D \times \ov{\square}^n)\right) \geq 0$. By Lemma \ref{lem:cancel}, this implies $\sum_{i=1} ^n \nu_W ^* \iota_W ^* \{ y_i = 1 \} - \nu_W ^* \iota_W ^* (D \times \ov{\square}^n ) \geq 0$, which means $W$ has modulus $D$. This completes the proof.
\end{proof}

Let $n, i \geq 1$. Suppose $X$ is smooth quasi-projective essentially of finite type over $k$. Let $(x, y_1, \cdots, y_n, y, \lambda)$ be the coordinates of $X \times \ov{\square}^{n+2}$. Consider the closed subschemes $ V_X ^i$ on $X \times \square^{n+2}$ given by the equation $(1-y) (1-\lambda) = 1- y_1$ if $i=1$, and $(1-y) (1-\lambda) = (1-y_1) (1 + y_1 + \cdots + y_1 ^{i-1} - \lambda (1 + y_1 + \cdots + y_1 ^{i-2}))$ if $i \geq 2$. 

Let $\widehat{V}_X ^i$ be the Zariski closure of $V_X ^i$ in $X \times \ov{\square} ^{n+2}$. Let $\pi_1: X \times \ov{\square} ^{n+2} \to X \times \ov{\square} ^{n+1}$ be the projection that drops $y_1$, and let $\pi_1 ':=\pi_1 |_{{V}_X ^i}$. As in \cite[Lemma 3.12]{KP2}, one sees that $\pi_1'$ is proper surjective. For an irreducible cycle $Z \subset X \times \square^n$, define (see \cite[Definition 3.13]{KP2}) $\gamma_Z ^i :={ \pi_1 ' }_*(V_X ^i \cdot (Z \times \square^2))$ as an abstract algebraic cycle. One checks that it is also the Zariski closure of $\nu^i (Z \times \square)$, where $\nu^i: X \times \square^n \times \square \to X \times \square^{n+1}$ is the rational map given by $\nu^i (x, y_1, \cdots, y_n, y) = (x, y_2, y_3, \cdots, y_n, y, \frac{y-y_1 ^i}{y-y_1 ^{i-1}})$. We extend the definition of $\gamma_Z^i$ $\mathbb{Z}$-linearly.

\begin{lem}\label{lem:gamma^i}
Let $Z \in z^q (X|D, n)$. Then $\gamma_Z ^i \in z^q (X|D, n+1)$.
\end{lem}

\begin{proof}
Once we have Lemma \ref{lem:brace}, the proof of Lemma \ref{lem:gamma^i} is very similar to that of \cite[Lemma 3.15]{KP2}, except we replace $(m+1) \{ t = 0 \}$ by $D \times \ov{\square}^{n+1}$. We give its argument for the reader's convenience.

We may assume $Z$ is irreducible. To keep track of $n$, we write $\gamma_{Z, n} ^i = \gamma_Z^i$. We first check that it satisfies the face condition of Definition \ref{defn:partial complex}. Let $\epsilon \in \{ 0, \infty\}$. 

Let $F \subset \square^{n+1}$ be a face. If $F$ involves the equation $\{ y_j = \epsilon\}$ for $j=n, n+1$, then by direction computations, we see that $\partial_n ^0 (\gamma_{Z,n} ^i) = \sigma \cdot Z, \partial_{n+1} ^0 (\gamma_{Z,n} ^i) = \sigma \cdot (Z\{ i \})$ for the cyclic permutation $\sigma = (1, 2, \cdots, n)$, and $\partial_n ^{\infty} (\gamma_{Z,n} ^i) = 0$, $\partial_{n+1} ^{\infty} (\gamma_{Z,n} ^i) = \sigma \cdot (Z \{ i-1 \})$. Since $Z$ is admissible with modulus $D$, so are $Z \{ i \}$ and $Z \{ i-1\}$ by Lemma \ref{lem:brace}. In particular, all of $\sigma \cdot Z, \sigma \cdot (Z \{ i \})$, and $\sigma \cdot ( Z \{ i-1 \})$ intersect all faces properly. Hence $\gamma_{Z,n}^i$ intersects $F$ properly.

In case $F$ does not involve the equations $\{y_j = \epsilon \}$ for $j=n, n+1$, we prove it by induction on $n \geq 1$. By direction calculations, for $j < n$, we have $\partial_j ^{\epsilon} (\gamma_{Z,n} ^i) = \gamma_{\partial_j ^{\epsilon} Z, n-1} ^i$ so that the dimension of $\partial _j ^i (\gamma_{Z, n} ^i)$ is at least one less by the induction hypothesis. Repeated applications of this argument for all other defining equations of $F$ then give the result. 

It remains to show that $\gamma_Z ^i$ has modulus $D$. Every irreducible component of $\gamma_Z ^i$ is of the form $W'= \pi_1' (Z')$, where $Z'$ is an irreducible component of $V_X ^i \cdot (Z \times \square^2)$. We prove $W'$ has modulus $D$. Consider the following commutative diagram
$$
\xymatrix{ {\ov{Z}'}^N \ar[d] ^{\pi_1 ^N} \ar[r] ^{\nu_{Z'}} & \widehat{V}_X ^i \ar@{^{(}->}[r] ^{\iota \ \ \ }\ar[d] ^{\ov{\pi}_1'} \ar[dr] ^{\pi_1 '} & X \times \ov{\square}^{n+2} \ar[d] ^{\pi_1} \\
{\ov{W}'}^N \ar[r] ^{\nu} & \ov{W}' \ar@{^{(}->}[r] ^{\iota_{W'} \ \ \ } & X \times \ov{\square}^{n+1},}
$$
where $\nu_{Z'}$ is the normalization of the Zariski closure $\ov{Z}'$ of $Z'$ in $\widehat{V}_X ^i$, $\nu$ is the normalization of the Zariski closure $\ov{W}'$ of $W'$ in $X \times \ov{\square}^{n+1}$, and $\pi_1^N$, $\ov{\pi}_1'$ are the induced morphisms. We use $(x, y_1, \cdots, y_n, y, \lambda) \in X \times \ov{\square}^{n+2}$ and $(x, y_2, \cdots, y_n, y, \lambda) \in X \times \ov{\square}^{n+1}$ as the coordinates. From the modulus $D$ condition of $Z$, we deduce 
\begin{equation}\label{eqn:gamma_Z}
\nu_{Z'} ^* \iota^*  (D \times \ov{\square} ^{n+2}) \leq \sum_{j=1} ^n \nu_{Z'} ^* \iota^* \{ y_j = 1 \}. 
\end{equation}
Note that the above does not involve the divisors $\{ y = 1 \}$ and $\{ \lambda = 1 \}$. Since $V_X ^i$ is an effective divisor on $X \times \square^{n+2}$ defined by the equation $(1-y_1)(*) = (1-y)(1-\lambda)$ for some polynomial $(*)$, we have $[\nu_{Z'} ^* \iota^* \{ y_1 = 1 \}] \leq [ \nu_{Z'} ^* \iota^* \{ y = 1 \}] + [ \nu_{Z'} ^* \iota^* \{ \lambda = 1 \}]$. 

Since the above diagram commutes, from \eqref{eqn:gamma_Z} we deduce ${ \pi_1 ^N}^* \nu^* \iota_{W'} ^* (D \times \ov{\square} ^{n+1}) \leq { \pi_1 ^N} ^* \left( \sum_{j=2} ^{n}  \nu^* \iota_{W'} ^*\{ y_j = 1 \} + \{ y = 1 \} + \{ \lambda = 1 \}\right)$. Hence by Lemma \ref{lem:cancel}, we deduce $\nu^* \iota_{W'} ^* (D \times \ov{\square} ^{n+1}) \leq   \sum_{j=2} ^{n} \nu^* \iota_{W'} ^*\{ y_j = 1 \} + \{ y = 1 \} + \{ \lambda = 1 \}$, which means $W'$ has modulus $D$. This finishes the proof.
\end{proof}

\begin{lem}\label{lem:permutation}
Let $n \geq 2$ and let $Z \in z^q (X|D, n)$ such that $\partial_i ^{\epsilon} (Z) = 0$ for all $1 \leq i \leq n$ and $\epsilon \in \{ 0, \infty \}$. Let $\sigma \in \mathfrak{S}_n$. Then there exists $\gamma_Z ^\sigma \in z^q (X|D, n+1)$ such that $Z = (\sgn (\sigma)) (\sigma \cdot Z) + \partial (\gamma_Z ^{\sigma}).$
\end{lem}

\begin{proof}Its proof is almost identical to that of \cite[Lemma 3.16]{KP2}, except that we use Lemma \ref{lem:gamma^i} instead of \cite[Lemma 3.15]{KP2}. We give its argument for the reader's convenience.

First consider the case when $\sigma$ is the transposition $\tau= (p, p+1)$ for $1 \leq p \leq n-1$. We do it for $p=1$ only, i.e. $\tau= (1,2)$. Other cases of $\tau$ are similar. Let $\xi$ be the unique permutation such that $\xi \cdot (x, y_1, \cdots, y_{n+1}) = (x, y_n, y_1, y_{n+1}, y_2, \cdots, y_{n-1})$. Consider the cycle $\gamma_Z ^{\tau}:= \xi \cdot \gamma_Z ^1$, where $\gamma_Z ^1$ is as in Lemma \ref{lem:gamma^i}. Being a permutation of an admissible cycle, so is this cycle $\gamma_Z ^\xi$. Furthermore, by direction calculations, we have $\partial_1 ^{\infty} (\gamma_Z ^{\tau}) = 0$, $\partial_1 ^0 (\gamma_Z ^{\tau}) = \tau \cdot Z$, $\partial_3 ^{\infty} (\gamma_Z ^{\tau}) = 0$ and $\partial_3 ^{0} (\gamma_Z ^{\tau}) = Z$. On the other hand, for $\epsilon \in \{ 0, \infty\}$, $\partial_2 ^{\epsilon} (\gamma_Z ^{\tau})$ is a cycle obtained from $\gamma_{\partial_2 ^{\epsilon} (Z)} ^1$ by a permutation action. So, it is $0$ because $\partial_2 ^{\epsilon} (Z) = 0$ by the given assumptions. Similarly for $j \geq 4$, we have $\partial_j ^{\epsilon} (\gamma_Z ^\tau) = 0$. Hence $\partial (\gamma_Z ^{\tau}) = Z + \tau \cdot Z$, as desired.

Now let $\sigma \in \mathfrak{S}_n$ be any. By a basic result from group theory, we can express $\sigma = \tau_r \tau_{r-1} \cdots \tau_2 \tau_1$, where each $\tau_i$ is a transposition of the form $(p, p+1)$ as considered before. Let $\sigma _0 := {\rm Id}$ and $\sigma_{\ell} := \tau_{\ell} \tau_{\ell-1} \cdots \tau_1$ for $1 \leq \ell \leq r$. For each such $\ell$, by the previous case considered, we have $(-1)^{\ell -1} \sigma _{\ell -1} \cdot Z + (-1)^{\ell -1} \tau _{\ell} \cdot \sigma _{\ell -1} \cdot Z = \partial ( ( -1)^{\ell -1} \gamma _{\sigma _{\ell -1} \cdot Z} ^{\tau _{\ell}}).$ Since $\tau _{\ell} \cdot \sigma_{\ell -1} = \sigma _{\ell}$, by taking the sum of the above equations over all $1 \leq \ell \leq r$, after cancellations, we obtain $Z + (-1)^{r-1} \sigma \cdot Z = \partial (\gamma _Z ^{\sigma})$, where $\gamma_Z ^{\sigma}:= \sum_{\ell=1} ^r (-1)^{\ell -1} \gamma_{\sigma _{\ell -1} \cdot Z} ^{\tau_{\ell}}$. Since $(-1)^r = \sgn (\sigma)$, we obtain the desired result.
\end{proof}

\subsection{Pontryagin product}\label{sec:Pont}
Let $X \in \Sch^{\rm ess}_k$ be an equidimensional scheme. For $\un{m} = (m_1, \cdots , m_r) \ge 1$, let $\CH(X[r]|D_{\un{m}}):= \oplus_{q,n} \CH^q(X[r]|D_{\un{m}},n)$. For $m \ge 1$, we let $\TH(X;m) = \oplus_{q,n} \TH^q(X,n; m) = \oplus_{q,n} \CH^q(X[1]|D_{m+1},n-1)$. The objective of \S \ref{sec:Pont} is to prove the following result which generalizes \cite[\S 3]{KP2}.

\begin{thm}\label{thm:algebra}Let $k$ be a perfect field. 
Let $m \ge 0$ and let $\un{m}= (m_1, \cdots, m_r) \ge 1$. Let $X, Y$ be both either in $\SmAff_k ^{\ess}$ or in $\SmProj_k$. Then we have the following:
\begin{enumerate}
\item $\TH(X;m)$ is a graded commutative algebra with respect to a product $\wedge_X$.
\item $\CH(X[r]|D_{\un{m}})$ is a graded module over $\TH(X; |\un{m}|-1)$.
\item For $f: Y \to X$ with $d= \dim Y - \dim X$, $f^*: \CH(X[r]|D_{\un{m}}) \to \CH(Y[r]|D_{\un{m}})$ and $f_*: \CH(Y[r]|D_{\un{m}}) \to \CH(X[r]|D_{\un{m}})[-d]$ (if $f$ is proper in addition) are morphisms of graded $\TH(X;|\un{m}|-1)$-modules. 
\end{enumerate}
\end{thm}

The proof requires a series of results and will be over after Lemma \ref{lem:asscomm}.

%\begin{lem}[{\emph{cf.} \cite[Lemma 2.11]{KP2}}]\label{lem:modulus concat}
\begin{lem}\label{lem:modulus concat}
Let $X_1, X_2 \in \Sch_k ^{\ess}$. For $i=1, 2$ and $r_i \ge 1$, let $V_i$ be a cycle on $X_i \times \mathbb{A}^{r_i} \times \square^{n_i}$ with modulus $\un{m}_i = (m_{i1}, \cdots, m_{ir_i})$, respectively. Then $V_1 \times V_2$, regarded as a cycle on $X_1 \times X_2 \times \mathbb{A}^{r_1 +r_2} \times \square^{n_1 + n_2}$ after a suitable exchange of factors, has modulus $(\un{m}_1, \un{m}_2)$. 
\end{lem}

\begin{proof}
We may assume that $V_1$ and $V_2$ are irreducible. It is enough to show that each irreducible component $W \subset V_1 \times V_2$ has modulus $(\un{m}_1, \un{m}_2)$. Let $\iota_i: \ov{V}_i \hookrightarrow X_i \times \mathbb{A}^{r_i} \times \ov{\square}^{n_i}$ be the Zariski closure of $V_i$, and let $\nu_{\ov{V}_i} : \ov{V}_i ^N \to \ov{V}_i$ be the normalization for $i=1,2$.  Since $k$ is perfect, \cite[Lemma 3.1]{KL} says that the morphism $\nu:= \nu_{\ov{V}_1} \times \nu_{\ov{V}_2} : \ov{V}_1 ^N \times \ov{V}_2 ^N \to \ov{V}_1 \times \ov{V}_2 = \ov{V_1 \times V_2}$ is the normalization. Hence, the composite $\ov{W}^N \overset{\nu_W}{\to} \ov{W} \overset{\iota}{\hookrightarrow} \ov{V}_1 \times \ov{V}_2$, where $\ov{W}$ is the Zariski closure of $W$ and $\nu_W$ is the normalization of $\ov{W}$, factors into $\ov{W}^N \overset{\iota^N}{\to} \ov{V}_1 ^N \times \ov{V}_2 ^N \overset{\nu}{\to} \ov{V}_1 \times \ov{V}_2$, where $\iota^N$ is the natural inclusion.

Let $(t_1, \cdots, t_{r_1}, t_1 ', \cdots, t_{r_2}' , y_1, \cdots, y_{n_1 + n_2}) \in \mathbb{A}^{r_1+ r_2} \times \ov{\square}^{n_1 + n_2}$ be the coordinates. Consider two divisors $D^1:= \sum_{i=1} ^{n_1} \{ y_i = 1 \} - \sum_{j=1} ^{r_1} m_{1j} \{ t_j = 0 \}, D^2 := \sum_{ i=n_1 + 1} ^{n_1 + n_2} \{ y_i = 1 \} - \sum_{j=1} ^{r_2} m_{2j} \{ t_j ' = 0 \}$. By the modulus conditions satisfied by $V_1$ and $V_2$, we have $((\iota_1 \times 1) \circ (\nu_{\ov{V}_1} \times 1))^* D^1 \geq 0$ and $((1 \times \iota_2) \circ (1 \times \nu_{\ov{V}_2}))^* D^2 \geq 0$. Thus, we have $\nu^* (\iota_1 \times \iota_2)^* (D^1 + D^2) \geq 0$ on $\ov{V}_1 ^N \times \ov{V}_2 ^N$ so that $(\iota^N)^* \nu^* (\iota_1 \times \iota_2)^* (D^1 + D^2) \geq 0$ on $\ov{W}^N$. Since $\iota \circ \nu_W = \nu \circ \iota^N$, this is equivalent to $\nu_W ^* \iota^* (\iota_1 \times \iota_2 )^* (D^1 + D^2) \geq 0$, which shows $W$ has modulus $(\un{m}_1, \un{m}_2)$.
\end{proof}

\begin{defn}\label{defn:mu}
Let $r \ge 1$ be an integer and define $\mu: X_1 \times \mathbb{A}^1 \times \square^{n_1} \times X_2 \times \mathbb{A}^r \times \square^{n_2} \to X_1 \times X_2 \times \mathbb{A}^{r} \times \square^{n_1 + n_2}$ by $(x_1, t, \{y_j\}) \times (x_2, \{t_i\}, \{y_j'\}) \mapsto (x_1, x_2, \{tt_i\}, \{y_j\}, \{y_j'\})$.
\end{defn}
The map $\mu$ is flat, but not proper. But, the following generalization of \cite[Lemma 3.4]{KP2} gives a way to take a push-forward:

\begin{prop}\label{prop:mu finite}
Let $V_1 \subset X_1 \times \mathbb{A}^1 \times \square^{n_1}$ and $V_2 \subset X_2 \times \mathbb{A}^r \times \square^{n_2}$ be closed subschemes with moduli $m$ and $\un{m} \ge 1$, respectively. Then $\mu|_{V_1 \times V_2}$ is finite.
\end{prop}

\begin{proof}
Since $\mu$ is an affine morphism, the proposition is equivalent to show that $\mu|_{V_1 \times V_2}$ is projective.

Set $X =  X_1 \times X_2 \times \square^{n_1+n_2}$. Let $\Gamma \inj X_1 \times X_2 \times \A^1 \times \A^{r} \times \A^r \times \square^{n_1+n_2} = X \times \A^1 \times \A^{r} \times \A^r$ denote the graph of the morphism $\mu$ and let $\ov{\Gamma} \inj  X \times \P^1 \times (\P^1)^{r} \times (\P^1)^r=X \times P_1 \times P_2 \times P_3$ be its closure, where $P_1 = \mathbb{P}^1$ and $P_2 = P_3 = (\mathbb{P}^1)^r$. Let $p_i$ be the projection of $X \times \P^1 \times (\P^1)^{r} \times (\P^1)^r$ to $X \times P_i$ for $1 \le i \le 3$. Set $\ov{\Gamma}^0 = p^{-1}_3(X \times \A^r)$. Then $p_3: \ov{\Gamma}^0 \to X \times \A^r$ is projective.

Using the homogeneous coordinates of $P_1 \times P_2 \times P_3$, one checks easily that $Z:= \ov{\Gamma} ^0 \setminus \Gamma \subset E \cup (\bigcup_{i=1} ^r E_i)$ (the union is taken inside $X \times P_1 \times P_2 \times P_3$), where $E= X \times \{ \infty \} \times (\{ 0 \})^r \times \mathbb{A}^r$ and $E_i = X \times \{ 0 \} \times (  (\mathbb{P}^1)^{i-1} \times \{ \infty \} \times (\mathbb{P}^1)^{r-i}) \times \mathbb{A}^r$. 

Let $V= V_1 \times V_2$. Let $\Gamma_V$ be the graph $\Gamma$ restricted to $V$ and let $\ov{\Gamma}_V$ be its Zariski closure in $X \times P_1 \times P_2 \times P_3$. Since $p_3: \ov{\Gamma} ^0 \to X \times \mathbb{A}^r$ is projective, so is the map $\ov{\Gamma}^0_V := \ov{\Gamma}_V \cap \ov{\Gamma}^0 \to X \times \A^r$. So, if we show $\ov{\Gamma} ^0 _V \cap Z = \emptyset$, then $V \simeq \Gamma_V = \ov{\Gamma}_V ^0$ is projective over $X \times \mathbb{A}^r$, which is the assertion of the proposition. 
%Let's prove that $\ov{\Gamma} ^0 _V \cap Z = \emptyset$.

To show $\ov{\Gamma} ^0 _V \cap Z = \emptyset$, consider the projections $X \times P_1 \times P_2 \times P_3 \overset{p_1}{\to} X \times P_1 \overset{\pi_1}{\to} X_1 \times P_1 \times \square^{n_1}$. Since the closure $\ov{V}_1$ has modulus $m \ge 1$ on $X_1 \times P_1 \times \square^{n_1}$, we have $\ov{V}_1 \cap (X_1 \times \{0\} \times \square^{n_1}) = \emptyset$. In particular, $\ov{\Gamma}_V \cap E_i \inj (\pi_1 \circ p_1)^{-1}(\ov{V}_1 \cap (X_1 \times \{0\} \times \square^{n_1})) = \emptyset$ for $1 \le i \le r$.

To show that $\ov{\Gamma}^0_V \cap E = \emptyset$, consider the projections $X \times P_1 \times P_2 \times P_3 \overset{p_2}{\to} X \times P_2 \overset{\pi_2}{\to} X_2 \times P_2 \times \square^{n_2}$. Since the closure $\ov{V}_2 $ has modulus $\un{m} \geq 1$ on $X_2 \times P_2 \times \square^{n_2}$, we have $\ov{V}_2 \cap (X_2 \times (\{0\})^r \times \square^{n_2}) = \emptyset$. In particular, $\ov{\Gamma}_V \cap E \inj (\pi_2 \circ p_2)^{-1}(\ov{V}_2 \cap (X_2 \times (\{0\})^r \times \square^{n_2})) = \emptyset$. This finishes the proof.
\end{proof}

%In the following lemma, we show independently that $\mu|_{V_1 \times V_2}$ is quasi-finite. It is a simple application of the modulus condition for additive cycles.

%\begin{lem}\label{lem:prod-qfin}
%$V_1 \subset X_1 \times \mathbb{A}^1 \times \square^{n_1}$ and
%$V_2 \subset X_2 \times \mathbb{A}^r \times \square^{n_2}$ be closed subschemes with moduli $m, \un{m} \ge 1$, respectively. 
%Then $\mu|_{V_1 \times V_2}$ is quasi-finite.
%\end{lem}
%\begin{proof}
%We set $V = V_1 \times V_2$ and let $W$ denote the scheme-theoretic image of $V$ under the map $\mu$. It suffices to show that the fiber of the map $\mu|_V: V \to W$ over every geometric point is finite.
%So we can assume $x = (x_1, x_2, t, \{a_l\}, \{y_i\}, \{y_j\}) \in X_1 \times X_2 \times \A^1 \times \A^r \times \square^{n_1 + n_2}$.  The fiber $\mu^{-1}(x)$ is contained in the union of closed subsets $A \cup B$, where 
%$A = \{(x_1, x_2, t, \{t_l\}, \{y_i\}, \{y_j\})| t_l = 0 \ 
%\mbox{for \ some} \ l\}$ and $B = (x_1, x_2, \{y_i\}, \{y_j\}) \times 
%\{(t, t_1, t_2) \in \A^1 \times \A^2| tt_1 = a_1, tt_2 = a_2\}$.
%
%The modulus condition for $V_2$ implies that $V \cap A = \emptyset$. Since $B$ is an one-dimensional integral closed subscheme, we must have $B \subset V$ if $V \cap B$ is not finite. In this case, we must have $(x_1, \{y_i\}) \times \A^1 = \ov{p(B)} \subset  V_1$, where 
%$p: X_1 \times X_2 \times \A^1 \times \A^r \times \square^{n_1 + n_2} \to X_1 \times \A^1 \times \square^{n_1}$ is the projection. This contradicts the modulus condition for $V_1$. We conclude that $V \cap B$ is finite. 
%\end{proof}

\begin{lem}\label{lem:mu mod}
Let $X\in \Sch_k^{\ess}$ and let $V$ be a cycle on $X \times \mathbb{A}^{1} \times \mathbb{A}^r  \times \square^n$ with modulus $(|\un{m}|, \un{m})$, where $\un{m} = (m_{1}, \cdots, m_{r}) \ge 1$. Suppose $\mu|_V$ is finite. Then the closed subscheme $\mu(V)$ on $X \times \mathbb{A}^r \times \square^n$ has modulus $\un{m}$. 
\end{lem}

\begin{proof}
This is a straightforward generalization of \cite[Proposition 3.8]{KP2} and is a simple application of Lemma \ref{lem:cancel}. We skip the detail. We only remark that it is crucial for the proof that the $\A^1$-component of the modulus is at least $|\un{m}|$. 
\end{proof}

\begin{defn}\label{defn:mu-times}
For any irreducible closed subscheme $V \subset X \times \mathbb{A}^{1} \times \mathbb{A}^r \times \square^n$ such that $\mu|_V: V \to \mu(V)$ is finite, where $\mu$ is as in Definition \ref{defn:mu}, define $\mu_* (V)$ as the push-forward $\mu_* (V) = \deg ( \mu|_V) \cdot [ \mu (V)]$. Extend it $\mathbb{Z}$-linearly. 

If $V_1$ is a cycle on $X_1 \times \mathbb{A}^1 \times \square^{n_1}$ and $V_2$ is a cycle on $X_2 \times \mathbb{A}^r \times \square^{n_2}$ such that $\mu|_{V_1 \times V_2}$ is finite, we define the external product $V_1 \times_{\mu} V_2:= \mu_* (V_1 \times V_2)$. If $p_i= \dim \ V_i$, then $\dim (V_1 \times_{\mu} V_2) = p_1 + p_2$. If $X_1 \times X_2$ is equidimensional and if $q_i$ is the codimension of $V_i$, then $V_1 \times_{\mu} V_2$ has codimension $q_1 + q_2 - 1$.
\end{defn}

\begin{lem}\label{lem:proper int}
Let $V_1 \in z^{q_1} (X_1 [1] |D_{m}, n_1)$ and $V_2 \in z^{q_2} (X_2 [r] |D_{\un{m}}, n_2)$ with $X_1, X_2 \in \Sch_k^{\ess}$ and $m, \un{m} \ge 1$. Then $V_1 \times_{\mu}  V_2$ intersects all faces of $X_1 \times X_2 \times \mathbb{A}^r \times \square^{n_1 + n_2}$ properly.
\end{lem}

\begin{proof}
We may assume that $V_1$ and $V_2$ are irreducible. $V_1 \times V_2$ clearly intersects all faces of $X_1 \times X_2 \times \A^1 \times \mathbb{A}^r \times \square^{n_1 + n_2}$ properly. It follows from \propref{prop:mu finite} that $\mu|_{V_1 \times V_2}$ is finite. In this case, the proper intersection property of $\mu(V_1 \times_{\mu}  V_2)$ follows exactly like that of the finite push-forwards of Bloch's higher Chow cycles.
\end{proof}

\begin{cor}\label{cor:prod mid summ}
Let $X_1, X_2, X_3 \in \Sch_k^{\ess}$ be equidimensional and let $\un{m} \ge 1$. Then there is a product
\[
\times_{\mu}: z^{q_1} (X_1[1]|D_{|\un{m}|}, n_1) \otimes 
z^{q_2} (X_2[r] |D_{\un{m}}, n_2) \to 
z^{q_1 + q_2 - 1}((X_1 \times X_2)[r]|D_{\un{m}}, n_1 + n_2)
\]
which satisfies the relation $\partial (\xi \times_{\mu} \eta )= \partial (\xi) \times_{\mu} \eta + (-1)^{n_1} \xi \times_{\mu} \partial (\eta)$. It is associative in the sense that $(\alpha_1 \times_{\mu} \alpha_2) \times_{\mu} \beta = \alpha_1 \times_{\mu} (\alpha_2 \times_{\mu} \beta)$ for $\alpha_i \in z^{q_i} (X_i [1] | D_{|\un{m}|}, n_i)$ for $i = 1,2$ and $\beta \in z ^{q_3} (X_3 [r] | D_{\un{m}}, n_3)$. In particular, it induces operations $\times_{\mu}: \CH^{q_1} (X_1[1]|D_{|\un{m}|}, n_1) \otimes \CH^{q_2} (X_2[r] |D_{\un{m}}, n_2) \to \CH^{q_1 + q_2 - 1}((X_1 \times X_2)[r]|D_{\un{m}}, n_1 + n_2)$. 
\end{cor}

\begin{proof}
The existence of $\times_{\mu}$ on the level of cycle complexes follows from the combination of Proposition \ref{prop:mu finite}, Lemma~\ref{lem:mu mod} and Lemma \ref{lem:proper int}. The associativity follows from that of the Cartesian product $\times$ and the product $\mu: \mathbb{A}^1 \times \mathbb{A}^1 \to \mathbb{A}^1$.

By definition, one checks $\partial (\xi \times \eta) = \partial (\xi) \times \eta + (-1)^{n_1} \xi \times  \partial (\eta)$. So, by applying $\mu_*$, we get the required relation. That $\times_{\mu}$ descends to the homology follows.
\end{proof}

\begin{defn}\label{defn:internal product}
Let $\un{m}= (m_1, \cdots, m_r) \ge 1$ and let $X$ be in $\SmAff_k ^{\ess}$ or in $\SmProj_k$. For cycle classes $\alpha_1 \in \CH^{q_1} (X[1]|D_{|\un{m}|}, n_1)$ and $\alpha_2 \in \CH^{q_2} (X[r]|D_{\un{m}}, n_2)$, define the internal product $\alpha_1 \wedge_X \alpha_2$ to be $\Delta_X ^* (\alpha_1 \times_{\mu} \alpha_2)$ via the diagonal pull-back $\Delta_X ^*: \CH^{q_1 + q_2 -1} ((X \times X)[r]|D_{\un{m}}, n_1 + n_2) \to \CH^{q_1 + q_2 -1} (X[r]|D_{\un{m}}, n_1 + n_2)$. This map exists by Theorem \ref{thm:contravariant} and Corollary \ref{cor:contra local}. 
\end{defn}

\begin{lem}\label{lem:asscomm}
$\wedge_X$ is associative in the sense that $(\alpha_1 \wedge_X \alpha_2) \wedge_X \beta = \alpha_1 \wedge_X(\alpha_2 \wedge_X \beta)$ for $\alpha_1, \alpha_2 \in \CH (X[1]|D_{|\un{m}|})$ and $\beta \in \CH (X[r]|D_{\un{m}})$. $\wedge_X$ is also graded-commutative on $\CH (X[1]|D_{|\un{m}|})$.
\end{lem}

\begin{proof}
The associativity holds by Corollary \ref{cor:prod mid summ}. For the graded-commutativity, first note by Theorem \ref{thm:normalization} that we can find representatives $\alpha_1$ and $\alpha_2$ of the given cycle classes whose codimension $1$ faces are all trivial. Let $\sigma$ be the permutation that sends $(1, \cdots, n_1, n_1 +1, \cdots, n_1 + n_2)$ to $(n_1 + 1, \cdots, n_1 + n_2, 1, \cdots, n_1)$ so that $\sgn (\sigma) = (-1)^{n_1+ n_2}$. It follows from Lemma \ref{lem:permutation} that $\alpha_1 \wedge_X \alpha_2 = (-1)^{n_1+ n_2} \alpha_2 \wedge_X \alpha_1 + \partial (W)$ for some admissible cycle $W$, as desired.
\end{proof}

\begin{proof}[Proof of Theorem \ref{thm:algebra}]
The proof of (1) and (2) is just a combination of the above discussion under the observation that $\TH^q(X,n;m) = \CH^q(X[1]|D_{m+1}, n-1)$ for $m \ge 0$ and $n \ge 1$. To prove (3) for $f^*$, consider the commutative diagram

\begin{equation}\label{eqn:pull-back-wedge}
\xymatrix@C1pc{
Y[r] \times \square^n \ar[d]^f  \ar[r]^<<{\Delta_{Y}} & (Y \times Y)[r] \times \square^n \ar[d]^{f \times f} & (Y \times Y)[r+1] \times \square^n \ar[d]^{f \times f} \ar[l]_{\mu_{Y}} \\
X[r] \times \square^n \ar[r]^<<<{\Delta_{X}} & (X \times X)[r] \times \square^n & (X \times X)[r+1] \times \square^n \ar[l]_{\mu_{X}}.}
\end{equation}

There is a finite set $\sW$ of locally closed subsets of $X$ such that $f^*: z^{q_1}_{\sW} (X[1]|D_{|\un{m}|}, \bullet) \to z^{q_1} (Y[1]|D_{|\un{m}|}, \bullet)$ and $f^*: z^{q_2}_{\sW} (X[r]|D_{\un{m}}, \bullet) \to z^{q_2} (Y[r]|D_{\un{m}}, \bullet)$ can be defined as taking cycles associated to the inverse images. Moreover, it is enough to consider the product of cycles in $z^{q_1}_{\sW} (X[1]|D_{|\un{m}|}, \bullet)$ and $z^{q_2}_{\sW} (X[r]|D_{\un{m}}, \bullet)$ by the moving lemmas Theorems~\ref{thm:moving affine} and ~\ref{thm:moving proj}. For irreducible cycles $V_1 \in z^{q_1} (X[1]|D_{|\un{m}|}, n_1)$ and $V_2 \in z^{q_2} (X[r]|D_{\un{m}}, n_2)$, the map $\mu_{Y}$ is finite when restricted to $f^*(V_1) \times f^*(V_2)$ by \lemref{prop:mu finite}. In particular, $\mu_{Y}(f^*(V_1) \times f^*(V_2)) \in z^{q_1+q_2-1}((Y \times Y)[r]|D_{\un{m}}, n_1 + n_2)$.

Since the right square in the diagram ~\eqref{eqn:pull-back-wedge} is transverse, it follows that $f^*(\mu_{X}(V_2 \times V_2)) = \mu_{Y}(f^*(V_1) \times f^*(V_2))$ as cycles. The desired commutativity of the product with $f^*$ now follows from the commutativity of the left square in ~\eqref{eqn:pull-back-wedge} and the composition law of \thmref{thm:contravariant}.

The proof of (3) for $f_*$ is just the projection formula, whose proof is identical to the one given in \cite[Theorem~3.19]{KP2} in the case when $X_1, X_2 \in \SmProj_k$.
\end{proof}

As applications, we obtain: 

\begin{cor}\label{cor:Witt-module}
Let $X$ be in $\SmAff_k ^{\ess}$ or in $\SmProj_k$. Then for $q, n \ge 0$ and $\un{m} \ge 1$, the group $\CH^q(X[r]|D_{\un{m}},n)$ is a $\W_{(|\un{m}|-1)}(k)$-module. 
\end{cor}

\begin{proof}
Applying \thmref{thm:algebra} to $X$ and the structure map $X \to \Spec(k)$, it follows that $\CH(X[r]|D_{\un{m}})$ is a graded module over $\TH(k; |\un{m}|-1)$. By Corollary \ref{cor:prod mid summ}, this yields a $\TH^1(k, 1; |\un{m}|-1)$-module structure on each $\CH^q(X[r]|D_{\un{m}},n)$. The corollary now follows from the fact that there is a ring isomorphism $\W_m(k) \overset{\sim}{\to} \TH^1(k,1;m)$ for every $m \ge 1$ by \cite[Corollary~3.7]{R}.
\end{proof}

We can explain the homotopy invariance of the groups $\CH^q(X, n)$ in terms of additive higher Chow groups as follows.

\begin{cor}\label{cor:mod-1}
For $X \in \Sch^{\rm ess}_k$ which is equidimensional and for $q,n \ge 0$, we have $\CH^q(X[1]|D_1, n) = 0$.
\end{cor}

\begin{proof}
By Corollary~\ref{cor:prod mid summ}, we have a map $\times_{\mu}: \CH^1(pt[1]|D_1, 0) \otimes \CH^q(X[1]|D_1, n) \to \CH^q(X[1]|D_1, n)$ and it follows from the definition of $\times_{\mu}$ that $[1] \times_{\mu} \alpha = \alpha$ for every $\alpha \in \CH^q(X[1]|D_1, n)$, where $[1] \in \CH^1(pt[1]|D_1, 0)$ is the cycle given by the closed point $1 \in \A^1(k)$. It therefore suffices to show that the homology class of $1$ is zero. To do so, we may use the identification $(\square, \{ \infty, 0 \}) \simeq (\mathbb{A}^1, \{ 0, 1 \})$ given by $y \mapsto 1/ (1-y)$ again. Then the cycle $C \subset \A^2$ given by $\{(t,y)\in \A^2| ty =1\}$ is an admissible cycle in $z^1(pt[1]|D_1, 1)$ such that $\partial_1([C]) = [1]$ and $\partial_0([C]) = 0$.
\end{proof}

\section{The structure of differential graded modules}
In this section, we construct a differential operator on the graded module of \S \ref{sec:DGA} of multivariate additive higher Chow groups over the univariate additive higher Chow groups, generalizing \cite[\S 4]{KP2}. We assume that $k$ is perfect and ${\rm char}(k) \neq 2$.

 \subsection{Differential}\label{sec:differential}

Let $X$ be a smooth quasi-projective scheme essentially of finite type over $k$. Let $r \ge 1$ and let $\un{m}= (m_1, \cdots, m_r) \ge 1$. Let $(\mathbb{G}_m ^r)^{\times} := \{ ( t_1, \cdots, t_r ) \in \mathbb{G}_m ^r \ | \ t_1 \cdots t_r \neq 1 \}$. Consider the morphism $\delta_n : (\mathbb{G}_m^r)^{\times}  \times \square^n \to \mathbb{G}_m^r  \times \square^{n+1},$ $ (t_1, \cdots, t_r, y_1, \cdots, y_n) \mapsto (t_1, \cdots, t_r, \frac{1}{t_1 \cdots t_r}, y_1, \cdots, y_n).$ It induces $\delta_n : X \times (\mathbb{G}_m^r ) ^\times  \times \square^n \to X \times  \mathbb{G}_m^r \times \square^{n+1}$.

Recall a closed subscheme $Z \subset X \times \mathbb{A}^r \times \square^n$ with modulus $\un{m}$ does not intersect the divisor $\{ t_1 \cdots t_r = 0 \}$. So, it is closed in $X \times \mathbb{G}_m ^r \times \square^n$. For such $Z$, we define $Z^{\times}:= Z |_{X \times (\mathbb{G}_m^r)^{\times} \times \square^n}$.

\begin{lem}\label{lem:deltaZ mod-00}
For a closed subscheme $Z \subset X \times \mathbb{A}^r \times \square^n$ with modulus $\un{m}$, the image $\delta_n (Z^{\times})$ is closed in $X \times \mathbb{G}_m ^r \times \square^{n+1}$. 
\end{lem}

\begin{proof}
It is enough to show that $\delta_n: X \times (\mathbb{G}_m ^r)^{\times} \times \square^n \to X \times \mathbb{G}_m ^r \times \square^{n+1}$ is a closed immersion. It reduces to show that the map $(\mathbb{G}_m ^r)^{\times} \to \mathbb{G}_m ^r \times (\mathbb{P}^1 \setminus \{ 1 \})$ given by $(t_1, \cdots, t_r) \mapsto (t_1, \cdots, t_r, {1}/{(t_1 \cdots t_r)})$ is a closed immersion. This is obvious because the image coincides with the closed subscheme given by the equation $t_1 \cdots t_r y = 1 $, where $(t_1, \cdots, t_r , y ) \in \mathbb{G}_m ^r \times \square$ are the coordinates.
\end{proof}

\begin{defn}[{\emph{cf.} \cite[Definition 4.3]{KP2}}]
For a closed subscheme $Z \subset X \times \mathbb{A}^r \times \square^n$ with modulus $\un{m}$, we write $\delta_n (Z):= \delta_n (Z^{\times})$. If $Z$ is a cycle, we define $\delta_n(Z)$ by extending it $\mathbb{Z}$-linearly. We may often write $\delta (Z)$ if no confusion arises.
\end{defn}

\begin{lem}\label{lem:deltaZ mod}
Let $Z$ be a cycle on $X \times \mathbb{A}^r \times \square^n$ with modulus $\un{m}$. Then $\delta_n(Z)$ is a cycle on $X \times \mathbb{A}^r \times \square^{n+1}$ with modulus $\un{m}$.
\end{lem}

\begin{proof}
We may suppose that $Z$ is irreducible. Let $V = \delta_n (Z)$, which is \emph{a priori} closed in $X \times \mathbb{G}_m ^r \times \square^{n+1}$. If the closure $V'$ of $V$ in $X \times \mathbb{A}^r \times \square^{n+1}$ has modulus $\un{m}$, then it does not intersect the divisor $\{t_1 \cdots t_r = 0 \}$ of $X \times \mathbb{A}^r \times \square ^{n+1}$, so $V=V'$, and $V$ is closed in $X \times \mathbb{A}^r \times \square^{n+1}$ with modulus $\un{m}$. So, we reduce to show that $V'$ has modulus $\un{m}$.

Let $\ov{Z}$ and $\ov{V}$ be the Zariski closures of $Z$ and $V'$ in $X \times \mathbb{A}^r \times \ov{\square}^n$ and $X \times \mathbb{A}^r \times \ov{\square}^{n+1}$, respectively. Observe that $\delta_n$ extends to $\ov{\delta}_n: X \times \mathbb{A}^r \times \ov{\square}^n \to X \times \mathbb{A}^r \times \ov{\square}^{n+1}$, which is induced from $\mathbb{A}^r \overset{\Gamma}{\to} \mathbb{A}^r \times \ov{\square} \overset{{\rm Id}\times\sigma}{\to} \mathbb{A}^r \times \ov{\square}$, where $\Gamma$ is the graph morphism of the composite $\mathbb{A}^r {\to} \mathbb{A}^1 \hookrightarrow \ov{\square}$ of the product map followed by the open inclusion, $(t_1, \cdots, t_r) \mapsto (t_1 \cdots t_r) \mapsto (t_1 \cdots t_r ; 1)$, while $\sigma: \ov{\square} \to \ov{\square}$ is the antipodal automorphism $(a;b) \mapsto (b;a)$, where $(a; b) \in \ov{\square}=\mathbb{P}^1$ are the homogeneous coordinates. Since $\Gamma$ is a closed immersion and ${\rm Id} \times \sigma$ is an isomorphism, the morphism $\ov{\delta}_n$ is projective. Hence, the dominant map $\delta_n|_{Z^{\times}} : Z ^{\times} \to V$ induces $\ov{\delta}_n|_{\ov{Z}} : \ov{Z} \to \ov{V}$. In particular, we have a commutative diagram
\begin{equation}\label{eqn:delta diag}
\xymatrix{
\ov{Z}^N \ar[d]^{\widetilde{\delta}_n} \ar[r]^{\nu_Z} & \ov{Z} \ar[d] ^{\ov{\delta}_n|_{\ov{Z}}} \ar@{^(->}[r]^{\iota_Z \ \ \ \ \ \ \ } & X \times \mathbb{A}^r \times \ov{\square}^n \ar[d] ^{\ov{\delta}_n} \\
\ov{V}^N \ar[r] ^{\nu_{V}} & \ov{V} \ar@{^(->}[r] ^{\iota_V \ \ \ \ \ \ \ } & X \times \mathbb{A}^r \times \ov{\square}^{n+1},}
\end{equation}
where $\iota_Z, \iota_V$ are the closed immersions, $\nu_Z, \nu_V$ are normalizations, and $\widetilde{\delta}_n$ is given by the universal property of normalization for dominant maps. 

By definition, $\ov{\delta}_n ^* \{ t_j = 0 \} = \{ t_j = 0 \}$ for $1 \leq j \leq r$. First consider the case $n \geq 1$. Then $\ov{\delta}_n ^* F_{n+1, i} ^1 = F_{n, i-1} ^1$ for $2 \leq i \leq n+1$. Now, $\widetilde{\delta}_n ^* \nu_V ^* \iota_V ^* ( \sum_{i=1} ^{n+1} F_{n+1, i} ^1 - \sum_{j=1} ^r m_j \{ t_j= 0 \})  \geq \widetilde{\delta}_n ^* \nu_V^* \iota_V ^* (  \sum_{i=2} ^{n+1} F_{n+1, i} ^1 - \sum_{j=1} ^r m_j \{ t_j= 0 \}) =^{\dagger}  \nu_Z ^* \iota_Z ^* \ov{\delta}_n ^* ( \sum_{i=2} ^{n+1} F_{n+1, i} ^1 - \sum_{j=1} ^r m_j \{ t_j= 0 \}) = \nu_Z ^* \iota_Z ^* (\sum_{i=2} ^{n+1} F_{n, i-1} ^1 - \sum_{j=1} ^r m_j  \{ t_j = 0 \}) = \nu_Z ^* \iota_Z ^* ( \sum_{i=1} ^n F_{n,i} ^1 - \sum_{j=1} ^r m_j \{ t_j = 0 \}) \geq ^{\ddagger} 0,$ where $\dagger$ holds by the commutativity of \eqref{eqn:delta diag} and $\ddagger$ holds as $Z$ has modulus $\un{m}$. Using Lemma \ref{lem:cancel}, we can drop $\widetilde{\delta}_n ^*$, i.e., $V'$ has modulus $\un{m}$.

When $n=0$, we have for $1 \leq j \leq r$, $\widetilde{\delta}_0 ^* \nu_V^* \iota_V ^* \{ t_j = 0 \} = \nu_Z ^* \iota_Z ^* \ov{\delta}_0 ^* \{ t_j = 0 \} = \nu_Z ^* \iota_Z ^* \{ t_j = 0 \}$, which is $0$ because $\ov{Z} \cap \{ t_j = 0 \}  = \emptyset$. Hence, $\widetilde{\delta}_0 ^* \nu_V ^* \iota_V ^* ( F_{1,1} ^1 - \sum_{j=1} ^r m_j \{ t_j = 0 \}) = \widetilde{\delta}_0 ^* \nu_V ^* \iota_V ^* F_{1,1} ^1 \geq 0$. Dropping $\widetilde{\delta}_0 ^*$, we get $V'$ has modulus $\un{m}$.
\end{proof}

%\begin{prop}[{\emph{cf.} \cite[Proposition 4.4]{KP2}}]
\begin{prop}\label{prop:delta admissible}
Let $Z \in z^q (X[r]|D_{\un{m}}, n)$. Then $\delta (Z) \in z^{q+1} (X[r] | D_{\un{m}}, n+1)$. Furthermore, $\delta$ and $\partial$ satisfy the equality $\delta \partial + \partial \delta = 0$.
\end{prop}

\begin{proof}
We may assume that $Z$ is an irreducible cycle. Let $\partial_{n, i} ^{\epsilon}$ be the boundary given by the face $F_{n, i} ^{\epsilon}$ on $X \times \mathbb{A}^r \times \square^n$, for $1 \leq i \leq n$ and $\epsilon = 0, \infty$. 

\noindent \textbf{Claim:}
For $\epsilon = 0, \infty$, (i) $\partial _{n+1, 1} ^{\epsilon} \circ \delta_n = 0$, (ii) $\partial_{n+1, i} ^{\epsilon} \circ \delta_n = \delta_{n-1} \circ \partial_{n, i-1} ^{\epsilon}$ for $2 \leq i \leq n+1$.

For (i), we show that $\delta_n (Z) \cap \{ y_1 = \epsilon \} = \emptyset$ for $\epsilon = 0, \infty$. Since $\delta_n (Z) \subset V( t_1 \cdots t_r y_1 = 1 )$, we have $\delta_n(Z) \cap \{ y_1 = 0 \} = \emptyset$. On the other hand, if $\delta_n (Z)$ intersects $\{ y_1 = \infty \}$, then some $t_i$ must be zero on $Z$, i.e., $Z$ intersects $\{ t_ i = 0 \}$ for some $1 \leq i \leq r$. However, since $Z$ has modulus $\un{m}$, this can not happen. Thus, $\delta_n (Z) \cap \{ y_1 = \infty \} = \emptyset$. This shows (i). For (ii), by the definition of $\delta_n$, the diagram
$$
\xymatrix{
(\mathbb{G}_m ^r)^{\times} \times \square^{n-1} \ar[r] ^{\iota_{i-1} ^{\epsilon}} \ar[d] ^{\delta_{n-1}} & (\mathbb{G}_m ^r)^{\times} \times \square^n \ar[d]^{\delta_n}\\
\mathbb{G}_m ^r \times \square^n \ar[r] ^{\iota_i ^{\epsilon}} & \mathbb{G}_m ^r \times \square^{n+1}}
$$
is Cartesian. Thus, $\delta_{n-1} ((\iota_{i-1} ^* (Z)) = (\iota_i ^\epsilon)^* (\delta_n (Z))$ by \cite[Proposition 1.7]{Fulton}, i.e., (ii) holds. This proves the claim.

By Lemma \ref{lem:deltaZ mod}, we know $\delta_n (Z)$ has modulus $\un{m}$. Since $Z$ intersects all faces properly, so does $\delta_n (Z)$ by applying (i) and (ii) of the above claim repeatedly. For $\partial \delta + \delta \partial = 0$, note that $\partial \delta_n (Z) = \sum_{i=1} ^{n+1} (-1)^i ( \partial_{n+1, i} ^{\infty} \delta_n (Z) - \partial_{n+1, i} ^0  \delta_n (Z) )=^{\dagger} \sum_{i=2} ^{n+1} (-1)^i ( \delta_{n-1} \partial_{n, i-1} ^{\infty} (Z) - \delta_{n-1}  \partial_{n, i-1} ^0 (Z) ) = - \sum_{i=1} ^n (-1)^i ( \delta_{n-1}  \partial_{n,i} ^{\infty} (Z) - \delta_{n-1} \partial_{n, i-1} ^0 (Z))= - \delta_{n-1} \sum_{i=1} ^n (-1)^i ( \partial_{n,i} ^{\infty} (Z) - \partial_{n,i} ^0 (Z))  = - \delta_{n-1} \circ \partial (Z),$ where $\dagger$ holds by the claim. 
\end{proof}

Lemma \ref{lem:diff} and Corollary \ref{cor:diff-*} below, which generalize \cite[\S 4.2]{KP2}, have much simpler proofs than \emph{loc.cit.}

\begin{lem}\label{lem:diff} 
Let $Z \in z^q (X[r]|D_{\un{m}}, n)$ be such that $\partial _i ^{\epsilon} (Z) = 0$ for $1 \leq i \leq n$ and $\epsilon = 0, \infty$. Then $2 \delta^2 (Z)$ is the boundary of an admissible cycle with modulus $\un{m}$.
\end{lem}

\begin{proof} 
Note that $\delta^2 (Z)$ is an admissible cycle on $X \times \mathbb{A}^r \times \square^{n+2}$ with modulus $\un{m}$, by Proposition \ref{prop:delta admissible}. For the transposition $\tau = (1,2)$ on the set $\{1, \cdots, n+2\}$, we have $\tau \cdot \delta^2 (Z) = \delta^2 (Z)$, by the definition of $\delta$. On the other hand, we have $\tau \cdot \delta^2 (Z)  = - \delta ^2 (Z) + \partial (\gamma)$ for some admissible cycle $\gamma$, by Lemma \ref{lem:permutation}. Hence, we have $- \delta^2 (Z) + \partial (\gamma) = \delta^2 (Z)$, i.e., $2 \delta^2 (Z) = \partial (\gamma)$, as desired.
\end{proof}

\begin{cor}\label{cor:diff-*}
Let $k$ be a perfect field of characteristic $\not = 2$ and let $X$ be in $\SmAff_k ^{\ess}$ or in $\SmProj_k$. Let $\un{m} \geq 1$. Then $\delta^2 = 0$ on $\CH^q (X[r]|D_{\un{m}}, n)$.
\end{cor}

\begin{proof}
If $r = \un{m} = 1$, by \corref{cor:mod-1}, there is nothing to prove. So, suppose either $r \geq 2$ or $|\un{m}| \ge 2$. But, if $r \geq 2$, then we automatically have $|\un{m}| \geq 2$, so we just consider the latter case.

Given $\alpha \in \CH^q (X[r]|D_{\un{m}}, n)$, by Theorem \ref{thm:normalization}, we can find a representative $Z \in z^q (X[r]|D_{\un{m}}, n)$ such that $\partial _i ^{\epsilon} (Z) = 0$ for $1 \leq i \leq n$ and $\epsilon = 0, \infty$. Then by Lemma \ref{lem:diff}, we have $2 \delta^2(\alpha) = 0$. 

On the other hand, by Corollary \ref{cor:Witt-module}, the group $\CH^q (X[r]|D_{\un{m}}, n)$ is a $\W_{(|\un{m}|-1)} (k)$-module. As $|\un{m}| \ge 2$ and ${\rm char}(k) \neq 2$, it follows that $2 \in (\W_{(|\un{m}|-1)}(k))^{\times}$. In particular, $ \delta^2(\alpha) = 0$.
\end{proof}

\subsection{Leibniz rule}\label{sec:Leibniz}
We now discuss the Leibniz rule, generalizing \cite[\S 4.3]{KP2}. Let $X \in \Sch^{\rm ess}_k$. Let $(x, t, t_1, \cdots, t_r, y_1, \cdots, y_{n+2}) \in X \times \mathbb{A}^{r+1} \times \square^{n+2}$ be the coordinates. Let $T \subset X \times \mathbb{A}^{r+1} \times \square^{n+2}$ be the closed subscheme defined by the equation $ty_{n+1}= y_{n+2} (tt_1 \cdots t_ry_{n+1}-1)$. 

\begin{defn}[{\emph{cf.} \cite[Definition 4.9]{KP2}}]\label{defn:C_Z}
Given a closed subscheme $Z \subset X \times \mathbb{A}^{r+1} \times \square^n$, define $C_Z:= T \cdot (Z \times \square^2)$ on $X \times \mathbb{A}^{r+1} \times \square^{n+2}$. This is extended $\mathbb{Z}$-linearly to cycles. 
\end{defn}

\begin{lem}\label{lem:mod-Leib}
Let $Z$ be a cycle on $X  \times \mathbb{A}^{r+1} \times \square^n$ with modulus $\un{m} = (m_1, \cdots, m_{r+1})$. Then $C_Z$ has modulus $\un{m}$  on $X \times \mathbb{A}^{r+1} \times \square^{n+2}$.
\end{lem}

\begin{proof}
We may assume $Z$ is irreducible. We show that each irreducible component $V \subset C_Z$ has modulus $\un{m}$. Let $\ov{Z}$ and $\ov{V}$ be the Zariski closures of $Z$ and $V$ in $X \times \mathbb{A}^{r+1} \times \ov{\square}^n$ and $X \times \mathbb{A}^{r+1} \times \ov{\square}^{n+2}$, respectively. The projection ${\rm pr}: X \times \mathbb{A}^{r+1} \times \ov{\square}^{n+2} \to X \times \mathbb{A}^{r+1} \times \ov{\square}^n$ that ignores the last two $\ov{\square}^2$ is projective, while its restriction to $X \times \mathbb{A}^{r+1} \times \square^{n+2}$ maps $V$ into $Z$. So, ${\rm pr}$ maps $\ov{V}$ to $\ov{Z}$, giving a commutative diagram
\begin{equation}\label{eqn:modulus C_Z}
\xymatrix{
\ov{V}^N \ar[r] ^{\nu_V} \ar[d] ^{{\rm pr} ^N} & \ov{V} \ar@{^(->}[r] ^{\iota_V\ \ \ \ \ \ \ \ \ } \ar[d] ^{{\rm pr}|_{\ov{V}}} & X \times \mathbb{A}^{r+1} \times \ov{\square}^{n+2} \ar[d] ^{\rm pr} \\
\ov{Z}^N \ar[r] ^{\nu_Z} & \ov{Z} \ar@{^(->}[r] ^{\iota_Z\ \ \ \ \ \ \ \ \ } & X \times \mathbb{A}^{r+1} \times \ov{\square}^n,}
\end{equation}
where $\iota_V$ and $\iota_Z$ are the closed immersions, $\nu_V$ and $\nu_Z$ are normalizations, and ${\rm pr} ^N$ is induced by the universal property of normalization for dominant maps. The modulus condition for $V$ is now easily verified using the pull-back of the modulus condition for $Z$ on $\ov{Z} ^N$ and the fact that ${\rm pr} ^* \{ t_j = 0 \} = \{ t_j = 0 \}$ for all $ j$ and ${\rm pr} ^* F_{n,i} ^1 = F_{n+2, i} ^1$ for all $ i$.
\end{proof}

\begin{cor}\label{cor:C_Z finite} 
Let $X_1, X_2 \in \Sch^{\rm ess}_k$. Let $V_1 \subset X_1 \times \mathbb{A}^1 \times \square^{n_1}$ and $V_2 \subset X_2 \times \mathbb{A}^r \times \square^{n_2}$ be closed subschemes with moduli $|\un{m}|$ and $\un{m}$, respectively with $\un{m} \ge 1$. 

Under the exchange of factors $X_1 \times \mathbb{A}^1 \times \square^{n_1} \times X_2 \times \mathbb{A}^r \times \square^{n_2} \simeq X_1 \times X_2 \times \mathbb{A}^{r+1} \times \square^n,$ where $n = n_1 + n_2$, consider the cycle $C_{V_1 \times V_2}$ on $X_1 \times X_2 \times \mathbb{A}^{r+1} \times \square^{n+2}$. Then $\mu|_{C_{V_1 \times V_2}}$ is finite. In particular, $\mu_* (C_{V_1 \times V_2})$ as in Definition \ref{defn:mu-times} is well-defined, and has modulus $\un{m}$. 
\end{cor}

\begin{proof}
We set $V = V_1 \times V_2$. From the definition of $\mu$, the map $\mu: V \times \square^2 \to X_1 \times X_2 \times \A^r \times \square^{n+2}$ is of the form $\mu|_V \times {\rm Id}_{\square^2}$. By Proposition \ref{prop:mu finite}, the map $\mu|_V$ is finite, thus so is $\mu|_{V} \times {\rm Id}_{\square^2}: V \times \square^2 \to X_1 \times X_2 \times \A^r \times \square^{n+2}$. Hence, its restriction to $C_{V} = T \cdot (V \times \square^2)$ is also finite. The modulus condition for $\mu_*(C_V)$ follows from Lemmas~\ref{lem:mu mod} and ~\ref{lem:mod-Leib}.
\end{proof}

\begin{defn}[{\emph{cf.} \cite[Definition 4.12]{KP2}}]\label{defn:cyclic Leibniz}
Let $V_1 \in z^{q_1} (X_1 [1] | D_{|\un{m}|}, n_1)$ and $V_2 \in z^{q_2} (X_2 [r] | D_{\un{m}}, n_2)$ with $X_1, X_2 \in \Sch^{\rm ess}_k$. Let $n= n_1 + n_2$ and define $V_1 \times_{\mu'} V_2$ be the cycle $\sigma \cdot \mu_* (C_{V_1 \times V_2})$, where $\sigma = (n+2, n+1, \cdots, 1)^2 \in \mathfrak{S}_{n+2}$.
\end{defn}

\begin{lem}\label{lem:Leibniz modulus}
Let $V_1, V_2$ be as in Definition~\ref{defn:cyclic Leibniz}. Then $V_1 \times_{\mu'} V_2 \in z^{q_1 + q_2 -1} ( (X_1 \times X_2)[r] |D_{\un{m}}, n_1 + n_2 + 2)$.
\end{lem}

\begin{proof}
By Corollary \ref{cor:C_Z finite}, the cycle $\mu_* (C_{V_1 \times V_2})$ has modulus $\un{m}$, thus so does $W:= V_1 \times_{\mu'} V_2$. It remains to prove that $W$ intersects all faces properly. Let $\sigma_{n_1} = (n_1 +1, n_1, \cdots, 1) \in \mathfrak{S}_{n+1}$. Then by direct calculations, we have
\begin{equation}\label{eqn:Leibniz faces}
\tuborg 
\partial_1 ^{\infty}W= \sigma_{n_1} (V_1 \times_{\mu} \delta (V_2)), \partial_1 ^0 W = 0,  \partial_2 ^{\infty} W = \delta ( V_1 \times_{\mu} V_2), \partial_2 ^0 W = \delta (V_1) \times_{\mu} V_2, \\
\partial_i ^{\epsilon} W = 
	\tuborg \partial_{i-2} ^{\epsilon} (V_1) \times_{\mu'} V_2, & \mbox{ for } 3 \leq i \leq n_1 +2, \\
	V_1 \times_{\mu'} \partial_{i - n_1 -2} ^{\epsilon}(V_2), & \mbox{ for } n_1 + 3 \leq i \leq n+2, 
		\sluttuborg 
\epsilon \in \{ 0, \infty \}. 
\sluttuborg
\end{equation} 
Since each $V_i$ is admissible, using \eqref{eqn:Leibniz faces}, Lemma \ref{lem:proper int}, Proposition \ref{prop:delta admissible} and induction on the codimension of faces, we deduce that $W$ intersects all faces properly.
\end{proof}

\begin{prop}\label{prop:Leibniz1}
Let $X_1, X_2 \in \Sm^{\rm ess}_k$.
%such that both are either in $\SmAff_k ^{\ess}$ or in $\SmProj_k$.
Let $\xi \in z^{q_1} (X_1[1] |D_{|\un{m}|}, n_1)$ and $\eta \in z^{q_2} (X_2 [r] |D_{\un{m}}, n_2)$. Let $n = n_1 + n_2$ and $q = q_1 + q_2$. Suppose that all codimension one faces of $\xi$ and $\eta$ vanish. Then in the group $z^{q-1}((X_1 \times X_2)[r]|D_{\un{m}}, n+1)$, the cycle $\delta( \xi \times _{\mu} \eta) - \delta \xi \times_{\mu} \eta - (-1)^{n_1} \xi \times_{\mu} \delta \eta$ is the boundary of an admissible cycle.
\end{prop}

\begin{proof}
By \eqref{eqn:Leibniz faces}, for $3 \leq i \leq n_1 +2$, we have $\partial_{i} ^{\epsilon} (\xi \times_{\mu'} \eta) = \partial_{i-2} ^{\epsilon} (\xi ) \times_{\mu'} \eta = 0$, while for $n_1 + 3 \leq i \leq n+2$, we have $\partial_i ^{\epsilon} (\xi \times_{\mu'} \eta) = \xi \times_{\mu'} \partial_{i-n_1 -2} ^{\epsilon} (\eta) = 0$. Hence, $ \partial (\xi \times _{\mu'} \eta) = \sum_{i=1} ^{n+2} (-1)^i (\partial_i ^{\infty} - \partial_i ^0) (\xi \times_{\mu'} \eta)  =  \delta (\xi \times_{\mu} \eta) - \{ \sigma_{n_1} \cdot (\xi \times_{\mu} \delta \eta) + \delta \xi \times_{\mu} \eta \}$ by \eqref{eqn:Leibniz faces} for $i=1, 2$. Equivalently,
 \begin{equation}\label{eqn:LLeib}
 \delta (\xi \times_{\mu} \eta) - \delta \xi \times_{\mu} \eta - \sigma_{n_1} \cdot ( \xi \times_{\mu} \delta \eta) = \partial (\xi \times_{\mu'} \eta).
 \end{equation}
But, for $\xi \times_{\mu} \delta \eta$, notice that
\begin{equation}\label{eqn:Leibniz normal}
\partial_i ^{\epsilon} ( \xi \times_{\mu} \delta \eta) = \tuborg \partial_i ^{\epsilon} \xi \times_{\mu} \delta \eta=0, & \mbox{ for } 1 \leq i \leq n_1, \\
 \xi \times_{\mu} \partial_{i- n_1} ^{\epsilon} (\delta \eta), & \mbox{ for } n_1 + 1 \leq i \leq n+1, \sluttuborg \epsilon \in \{ 0, \infty \}.
\end{equation}
We have $\partial_{1} ^{\epsilon} (\delta \eta) = 0$ when $i = n_1 +1$ by Claim (i) of Proposition \ref{prop:delta admissible}, and $\partial_{i - n_1} ^{\epsilon} (\delta \eta) = \delta (\partial_{i - n_1 -1} ^{\epsilon} \eta) = \delta(0) = 0$ when $n_1 + 2 \leq i \leq n+1$ by Claim (ii) of Proposition \ref{prop:delta admissible}. Hence, $\xi \times_{\mu} \delta \eta$ is a cycle with trivial codimension $1$ faces, so, by Lemma \ref{lem:permutation}, for some admissible cycle $\gamma$, we have $\sigma_{n_1} \cdot (\xi \times_{\mu} \delta \eta) = \sgn (\sigma_{n_1}) (\xi \times_{\mu} \delta \eta) + \partial (\gamma) = (-1)^{n_1} \xi \times_{\mu} \delta \eta + \partial (\gamma)$. Putting this back in \eqref{eqn:LLeib}, we obtain $\delta (\xi \times_{\mu} \eta) - \delta \xi \times_{\mu} \eta - (-1)^{n_1} \xi \times_{\mu} \delta \eta = \partial (\xi \times_{\mu'} \eta) - \partial (\gamma),$ as desired.
\end{proof}

The above discussion summarizes as follows:

\begin{thm}\label{thm:DGA}
Let $X$ be in $\SmAff_k ^{\ess}$ or in $\SmProj_k$ over a perfect field $k$ with ${\rm char}(k) \not = 2$. Let $r \geq 1$ and $\un{m} = (m_1, \cdots , m_r) \ge 1$. Then the following hold:
\begin{enumerate}
\item $(\CH (X[1]|D_{|\un{m}|}), \wedge_X, \delta)$ forms a commutative differential graded $\W_{(|\un{m}|-1)}\Omega^{\bullet}_{k}$-algebra. 
\item $(\CH (X[r]|D_{\un{m}}), \delta)$ forms a differential graded $(\CH (X[1]|D_{|\un{m}|}), \wedge_X, \delta)$-module.
\end{enumerate}

In particular, $(\CH (X[r]|D_{\un{m}}), \delta)$ is a differential graded $\W_{(|\un{m}|-1)}\Omega^{\bullet}_{k}$-module.
\end{thm}

\begin{proof}
The commutative differential graded algebra structure on $\CH(X[1]|D_{|\un{m}|})$ and the differential graded module structure on $\CH(X[r]|D_{\un{m}})$ over $\CH(X[1]|D_{|\un{m}|})$ follows by combining \thmref{thm:algebra}, \corref{cor:diff-*} and Proposition~\ref{prop:Leibniz1} using \thmref{thm:normalization}.

The structure map $p:X \to \Spec(k)$ turns $(\CH (X[1]|D_{|\un{m}|}),  \wedge_X, \delta)$ into a differential graded algebra over $(\CH (pt[1]|D_{|\un{m}|}),  \wedge_{pt}, \delta)$ via $p^*$. Since $\oplus_{n \ge 0}\CH^{n+1}(pt[1]|D_{|\un{m}|}, n)$ forms a differential graded sub-algebra of $(\CH (pt[1]|D_{|\un{m}|}),  \wedge_{pt}, \delta)$.
%isomorphic to $\mathbb{W}_{|\un{m}|-1} \Omega_k ^{\bullet}$ by \cite{R}, 
The map of commutative differential graded algebras $\W_{(|\un{m}|-1)}\Omega^{\bullet}_{k} \to \oplus_{n \ge 0}\CH^{n+1}(pt[1]|D_{|\un{m}|}, n)$ (see \cite{R}) finishes the proof of the theorem.
\end{proof}

As a consequence of \thmref{thm:DGA} (use \corref{cor:mod-1} when $|\un{m}| =1$), we obtain the following property of multivariate additive higher Chow groups.
%This property was expected because these groups are supposed be the motivic cohomology which compute the relative $K$-theory of $X \times \A^r$ relative to the union of various infinitesimal thickenings of the coordinate hyperplanes.

\begin{cor}\label{cor:DGA-rational}
 Let $r \geq 1$ and $\un{m} \geq 1$ and let $X$ be in $\SmAff_k ^{\ess}$ or in $\SmProj_k$. Then each $\CH^q(X[r]|D_{\un{m}}, n)$ is a $k$-vector space provided ${\rm char}(k) = 0$.
\end{cor}

\section{Witt-complex structure on additive higher Chow groups}\label{sec:Witt-complex}
Let $k$ be a perfect field of characteristic $\not =2$. In this section, a \emph{smooth affine $k$-scheme} means an object in $\SmAff_k ^{\ess}$, i.e., an object of either $\SmAff_k$ or $\SmLoc_k$. 

R\"ulling proved in \cite{R} that the additive higher Chow groups of $0$-cycles over $\Spec(k)$ form a  restricted Witt-complex over $k$. When $X$ is a smooth projective variety over $k$, it was proven in \cite{KP2} that additive higher Chow groups of $X$ form a restricted Witt-complex over $k$. Our objective is to prove the stronger assertion that the additive higher Chow groups of $\Spec(R) \in \SmAff_k ^{\ess}$ have the structure of a restricted Witt-complex over $R$. 

Since we exclusively use the case $r=1$ only, we use the older notations $\TZ^q (X, n; m)$ and $\TH ^q (X, n;m)$ instead of $z^q (X[1]|D_{m+1}, n-1)$ and $\CH^q (X[1]|D_{m+1}, n-1)$. For $X \in \Sch_k^{\ess}$, we let $\TH(X;m): = \oplus_{n,q} \TH^q (X, n;m) $ and $\TH^M (X; m): = \oplus_{n} \TH^n (X, n;m)$. The superscript $M$ is for \emph{Milnor}. Let $\TH(X):= \oplus_m \TH(X;m)$ and $\TH^M (X):= \oplus_m \TH^M (X;m)$. We similarly define $\tch(X;m)$, $\tch^M (X;m)$, $\tch (X)$, and $\tch^M (X)$ for $X \in \Sch_k$ using Definition \ref{defn:tch_colimit}.

\subsection{Witt-complex structure over $k$}\label{sec:Witt over k} 
In this section, we show that the additive higher Chow groups for an object of $\SmAff_k^{\ess}$ form a functorial restricted Witt-complex over $k$. For $r \geq 1$, let $\phi_r: \A^1 \to \A^1$ be the morphism $x \mapsto x^r$, which induces $\phi_r: \Spec (R) \times B_n \to \Spec (R) \times B_n$. By \cite[\S 5.1, 5.2]{KP2}, we have the Frobenius $F_r: \TH^q (R, n; rm+r-1) \to \TH^q (R, n;m)$ and the Verschiebung $V_r: \TH^q (R, n;m) \to \TH^q (R, n;rm+r-1)$ given by $F_r = \phi_{r *}$ and $V_r = \phi^* _r$. We also have a natural inclusion $\mathfrak{R}: \TZ^q (R, \bullet ; m+1) \to \TZ^q (R, \bullet; m)$  for any $m \geq 1$, which induces $\mathfrak{R}: \TH^q (R, n; m+1) \to \TH^q (R, n;m)$, called the \emph{restriction}. Finally, by \thmref{thm:DGA}, there is a differential $\delta : \TZ^q (R, \bullet; m) \to \TZ^q (R, \bullet +1; m)$, which induces $\delta: \TH^q (R, n;m) \to \TH^q (R, n+1;m)$.

\begin{thm}\label{thm:Witt over k}
Let $X \in \SmAff_k ^{\ess}$ and $m \geq 1$. Then $\TH (X ; m)$ is a DGA and $\TH^M (X;m)$ is its sub-DGA. Furthermore, with respect to the operations $\delta, \mathfrak{R}, F_r, V_r$ in the above together with $\lambda= f^*: \mathbb{W}_m (k)= \TH^1 (k, 1;m) \to \TH^1 (X, 1;m)$ for the structure morphism $f: X \to \Spec (k)$, $\TH(X)$ is a restricted Witt-complex over $k$ and $\TH^M (X)$ is a restricted sub-Witt-complex over $k$. These structures are functorial.
\end{thm}

\begin{proof}
In \cite[Theorem 1.1, Scholium 1.2]{KP2}, it was stated that $\TH(X;m)$ and $\TH^M (X;m)$ are DGAs, and that $\TH(X)$ and $\TH^M (X)$ are restricted Witt-complexes over $k$ with respect to the above $\delta, \mathfrak{R}, F_r, V_r$, provided the moving lemma holds for $X$. But this is now shown in Theorems \ref{thm:moving affine} and ~\ref{thm:moving local}. We give a very brief sketch of this structure and its functoriality.

The functoriality of the restriction operator $\mathfrak{R}$ recalled  above, was stated in \cite[Corollary 5.19]{KP2}, which we easily check here: let $f: X \to Y$ be a morphism in $\SmAff_k^{\ess}$ and consider the following commutative diagram:
$$
\xymatrix{ \TZ^q _{\mathcal{W}} (Y, \bullet; m+1) \ar[r]^{f^*} \ar @{^{(}->}[d] & \TZ^q (X, \bullet; m+1) \ar @{^{(}->} [d] \\
\TZ^q_{\mathcal{W}} (Y, \bullet ; m) \ar[r] ^{f^*} & \TZ^q (X, \bullet; m),}
$$
where $\mathcal{W}$ is a finite set of locally closed subsets of $Y$, and the horizontal maps are chain maps given by the inverse images as in the proof of Theorem \ref{thm:contravariant} and Corollary \ref{cor:contra local}. The diagram and Theorems \ref{thm:moving affine} and \ref{thm:moving local} imply that $f^* \mathfrak{R} = \mathfrak{R} f^*$ because the vertical inclusions induce $\mathfrak{R}$ by definition.

For each $r\geq 1$, the Frobenius $F_r$ and Verschiebung $V_r$ recalled in the above are functorial as proven in \cite[Lemmas 5.4, 5.9]{KP2}, and that $F_r$ is a graded ring homomorphism is proven in \cite[Corollary 5.6]{KP2}.

Finally, the properties (i), (ii), (iii), (iv), (v) in Section \ref{sec:DRW}, are all proven in \cite[Theorem 5.13]{KP2}, where none requires the projectivity assumption. 
\end{proof}

\begin{cor}\label{cor:Witt over k}
Let $m \geq 1$ be an integer. Then $\tch ( - ; m)$ and $\tch^M (-;m)$ define presheaves of DGAs on $\Sch_k$, and the pro-systems $\tch (-)$ and $\tch^M (-)$ define presheaves of restricted Witt-complexes over $k$ on $\Sch_k$.
\end{cor}

\begin{proof}Let $X \in \Sch_k$. By definition, $\tch (X;m)$ is the colimit over all $(X \to A) \in (X\downarrow \SmAff_k)^{\op}$ of $\TH(A; m)$. But the category of DGAs is closed under filtered colimits (see \cite{Jardine}) so that $\tch (X; m)$ is a DGA. For each morphism $f: X \to Y$ in $\Sch_k$, one checks $f^*: \tch (Y;m)\to \tch(X;m)$ is a morphism of DGAs. The other assertions follow easily using Theorem \ref{thm:Witt over k}.
\end{proof}

Before we discuss Witt-complexes over $R$, we state the following behavior of various operators under finite push-forward maps.

\begin{prop}\label{prop:pushforward}
Let $f: X \to Y$ be a finite map in $\SmAff^{\ess}_k$. Then for $r \geq 1$, we have: $(a)~ f_* \mathfrak{R} = \mathfrak{R} f_*; ~ (b) ~ f_* \delta = \delta f_*; \ (c)~ f_* F_r = F_r f_*; \ (d)~ f_* V_r = V_r f_*$.
\end{prop}

\begin{proof}
The item (a) is obvious and (b) and (c) follow at once from the fact that these operators are defined as push-forward under closed immersion and finite maps and they preserve the faces. For (d), we consider the commutative diagram 
\begin{equation}\label{eqn:pfFV}
\xymatrix{ X \times \A^1 \ar[r] ^{{\rm Id} \times \phi_r} \ar[d] ^{ f \times {\rm Id}} & X \times \A^1 \ar[d] ^{ f \times {\rm Id}} \\
Y \times \A^1 \ar[r] ^{{\rm Id} \times \phi_r} & Y \times \A^1.}
\end{equation} 

Since this diagram is Cartesian and $f$ as well as $\phi$ preserve the faces, we conclude from \cite[Proposition 1.7]{Fulton} that $f_* \circ \phi^*_r = \phi^*_r \circ f_*$.
\end{proof}

\subsection{Witt-complex structure over $R$}\label{sec:Witt affine}
Let $X= \Spec (R)\in \SmAff_k^{\ess}$. The objective of this section is to strengthen Theorem \ref{thm:Witt over k} by showing that $\TH (X)$ is a restricted Witt-complex over $R$. 

\bigskip
Let $m \geq 1$ be an integer. We first define a group homomorphism $\tau_R: \mathbb{W}_m (R) \to \TH^1 (R, 1 ; m)$ for any $k$-algebra $R$. Recall that the underlying abelian group of $\mathbb{W}_m (R)$ identifies with the multiplicative group $(1 + t R[[t]])^{\times} / (1 + t^{m+1} R[[t]])^{\times}$. For each polynomial $p(t) \in (1 + R[[t]])^{\times}$, consider the closed subscheme of $ \Spec (R[t])$ given by the ideal $(p(t))$, and let $\Gamma_{(p(t))}$ be its associated cycle. By definition, $\Gamma _{(p(t))} \cap \{ t = 0 \} = \emptyset$ so that $\Gamma_{(p(t))} \in \TZ^1 (R, 1; m)$. We set $\Gamma_{a,n} = \Gamma_{(1-at^n)}$ for $n \ge 1$ and $a \in R$.

\begin{lem}\label{lem:Witt homo}
Let $f(t), g(t)$ be polynomials in $R[t]$, and let $h(t) \in R[t]$ be the unique polynomial such that $(1-tf(t))(1-tg(t)) = 1-t h(t)$. Then $\Gamma_{(1-th(t))} = \Gamma_{(1-tf(t))} + \Gamma _{(1-tg(t))}$ in $\TZ^1 (R, 1; m)$. 
\end{lem}

\begin{proof}This is obvious by $(1-tf(t))(1-tg(t)) = 1-t h(t)$.
\end{proof}

\begin{lem}\label{lem:mWitt well-def} 
For $n \geq m+1$, we have $\Gamma_{(1-t^nf(t))} \equiv 0$ in $\TH^1 (R, 1; m)$. 
\end{lem}

\begin{proof}
Consider the closed subscheme $\Gamma \subset X \times \A^1 \times \square$ given by $y_1 = 1- t^nf(t)$. Let $\nu: \ov{\Gamma} ^N \to \ov{\Gamma} \hookrightarrow  X \times \mathbb{A}^1 \times \mathbb{P}^1$ be the normalization of the Zariski closure $\ov{\Gamma}$ in $X \times \mathbb{A}^1 \times \mathbb{P}^1$. Since $f(t) t^n = 1 - y_1$ on $\ov{\Gamma}$, we see that $n \nu^* \{ t = 0 \} \leq \nu^* \{ y_1 = 1 \}$ on $\ov{\Gamma} ^N$. Since $n \geq m+1$, this shows that $\Gamma$ satisfies the modulus $m$ condition. Since $\partial_1 ^{\infty} (\Gamma) = 0$ and $\partial _1 ^0 (\Gamma) = \Gamma_{(1-t^nf(t))}$ (which is of codimension $1$), the cycle $\Gamma$ is an admissible cycle in $\TZ^1 (R, 2; m)$ such that $\partial \Gamma = \Gamma_{(1-t^nf(t))}$. This shows that $\Gamma_{(1-t^nf(t))} \equiv 0$ in $\TH^1 (R, 1; m)$.
\end{proof}

\begin{prop}\label{prop:tau_R}
Let $R$ be a $k$-algebra. Then the map $\tau_R: (1 + tR[t]) \to \TZ^1 (R, 1; m)$ that sends a polynomial $1- tf(t)$ to $\Gamma_{(1-tf(t))}$, defines a group homomorphism $\tau_R: \mathbb{W}_m (R) \to \TH^1 (R, 1; m)$. 
\end{prop}

\begin{proof}
Every element $p(t) \in (1 + t R[[t]])^{\times}$ has a unique expression $p(t) = \prod_{ n \geq 1} (1 - a_n t^n)$ for $a_n \in R$. For any such $p(t)$, set $p^{\le m}(t) = \prod_{n=1} ^m  (1-a_n t^n)$. We define $\tau_R(p(t)) =  \Gamma_{(p^{\le m}(t))}$. It follows from Lemmas~\ref{lem:Witt homo} and ~\ref{lem:mWitt well-def} that this map descends to a group homomorphism from $\W_m(R)$.
\end{proof}

Recall from \cite[Appendix A]{R} that for each $r \geq 1$, we have the Frobenius $F_r: \mathbb{W}_{rm+r-1} (R) \to \mathbb{W}_m (R)$ and the Verschiebung $V_r: \mathbb{W}_m (R) \to \mathbb{W}_{rm+r-1}(R)$. They are given by $F_r (1- at^n) = (1- a^{\frac{r}{s}} t^{\frac{n}{s}})^s$, where $s= \gcd (r, n)$ and $V_r (1-at^n) = 1-at^{rn}$. On the other hand, as seen in Section \ref{sec:Witt over k}, we have operations $F_r$ and $V_r$ on $\{ \TH^1 (R, 1;m) \}_{m \in \mathbb{N}}$.

\begin{lem}\label{lem:tau FV}
Let $R$ be a $k$-algebra. Then the maps $\tau_R: \mathbb{W}_m (R) \to \TH^1 (R, 1; m)$ of \propref{prop:tau_R} commute with the $F_r$ and $V_r$ operators on both sides.
\end{lem}

\begin{proof}
That $\tau_R V_r = V_r \tau _R$, is easy: we have $V_r (\tau_R (1-at^n)) = V_r (\Gamma_{a,n}) = \Gamma_{a, rn}$, while $\tau_R(V_r (1-a t^n)) = \Gamma_{(1 - a t^{rn})} = \Gamma_{a, rn}$.

That $\tau_R F_r = F_r \tau_R$, is slightly more involved. Recall that $F_r (1- at^n) = (1- a^{\frac{r}{s}} t^{\frac{n}{s}})^s$, where $s= \gcd (r, n)$. Write $n= n' s$ and $r=r' s$, where $1= (r', n')$. Hence, we have $\tau_R F_r (1-at^n) = s \Gamma _{a^{\frac{r}{s}}, \frac{n}{s}} = s V_{\frac{n}{s}} (\Gamma_{a^{\frac{r}{s}}, 1}) = s V_{n'} (\Gamma _{a ^{r'}, 1}) =: \clubsuit$, while $F_r \tau_R (1-at^n) = F_r \Gamma_{a, n} = F_r V_n (\Gamma_{a, 1})=:\heartsuit$. 

First observe that when $n=1$, we have $s=1$, $r=r', n=n'=1$, and we have $\heartsuit=F_r (\Gamma_{a,1}) = \Gamma_{a^r, 1}= \clubsuit$, so that $\tau_R F_r (1-at)= F_r \tau_R (1-at)$, indeed.

For a general $n \geq 1$, we have $F_r V_n = F_{r'}F_{s} V_{s} V_{n'} = F_{r'} \circ (s \cdot {\rm Id}) \circ  V_{n'} = s F_{r'} V_{n'} =^{\dagger} s V_{n'} F_{r'}$, where $\dagger$ holds because $(r', n') = 1$. Since $F_{r'} (\Gamma_{a,1}) = \Gamma_{a ^{r'}, 1}$ (by the first case), we have $\heartsuit=F_r V_n (\Gamma_{a,1}) = s V_{n'} F_{r'} (\Gamma_{a,1}) = s V_{n'} (\Gamma_{a^{r'}, 1})=\clubsuit$. This shows $\tau_R F_r = F_r \tau_R$.
\end{proof}

\begin{remk}
In the proof of Lemma \ref{lem:tau FV}, we saw that for $s= (r,n)$, 
\begin{equation}\label{eqn:FV identity}F_r (\Gamma_{a,n} ) = s \Gamma_{ a^{\frac{r}{s}}, \frac{n}{s}}, \ \ 
V_r (\Gamma_{a,n}) = \Gamma_{a, rn}.
\end{equation}
\end{remk}

\begin{prop}\label{prop:smooth tau R}
For $X = \Spec (R) \in \SmAff_k^{\ess}$, the maps $\tau_R : \mathbb{W} _m (R) \to \TH^1 (R, 1; m)$ form a morphism of pro-rings that commutes with $F_r$ and $V_r$ for $r \geq 1$. 
\end{prop}

\begin{proof} 
It is clear from the definition of $\tau_R$ in \propref{prop:tau_R} that it commutes with $\mathfrak{R}$. We saw that $\tau_R$ commutes with $F_r$ and $V_r$ in Lemma \ref{lem:tau FV}. So, we only need to show that $\tau_R$ respects the products. By \cite[Proposition (1.1)]{Bloch crys}, it is enough to prove that for $a, b \in R$ and $u, v \geq 1$,
\begin{equation}\label{eqn:ring homo}
\Gamma_{a, u} \wedge \Gamma_{b, v} = w \Gamma_{ a^{\frac{v}{w}} b^{\frac{u}{w}}, \frac{uv}{w}} \ \ \mbox{ in } \TH^1 (R, 1; m),
\end{equation}
where $w = \gcd (u, v)$ and $\wedge = \wedge_X$ is the product structure on the ring $\TH^1 (R, 1;m)$ as in Theorem \ref{thm:Witt over k}. 
%(N.B. The smoothness of $R$ is essential to have the pull-back $\Delta^*$ in 
%the definition of $\wedge$.) We prove it in three steps.

\noindent \textbf{Step 1.} First, consider the case when $u=v=1$, \emph{i.e.}, we prove $\Gamma_{a,1} \wedge \Gamma_{b,1} = \Gamma_{ab,1}.$ Recall that $\wedge$ is defined as the composition $\Delta^* \circ \mu_* \circ \times$ in
\[
X \times \A^1 \times X \times \A^1 
\overset{\mu}{\to} X \times X \times \A^1 
\overset{\Delta}{\leftarrow} X \times \A^1.
\]

Under the identification $X \times X \simeq \Spec (R \otimes_k R)$, we have $\mu_* (\Gamma_{a,1} \times \Gamma_{b,1}) = \Gamma_{(a \otimes 1)(1 \otimes b), 1}$, and $\Delta^* (\Gamma_{(a \otimes 1)(1 \otimes b), 1}) = \Gamma_{ab, 1}$, because $\Delta$ is given by the multiplication $R \otimes_k R \to R$. This proves \eqref{eqn:ring homo} for {Step 1.}

For the following remaining two steps, we use the projection formula: $x \wedge V_s (y) = V_s (F_s (x) \wedge y)$, which we can use by Theorem \ref{thm:Witt over k}.

 \noindent \textbf{Step 2.} Consider the case when $v=1$, but $u \geq 1$ is any integer. We apply the projection formula to $x= \Gamma_{b,1}$ and $y = \Gamma_{a,1}$ with $s=u$. Since $\TH^1 (R, 1;m)$ is a commutative ring, by the projection formula, we get $V_u (\Gamma_{a,1}) \wedge \Gamma_{b,1} = V_u ( \Gamma_{a,1} \wedge F_u (\Gamma_{b,1})).$ Here, the left hand side is $\Gamma_{a, u} \wedge \Gamma_{b,1}$ by eqn:FV identity, while the right hand side is $=^1 V_u (\Gamma_{a,1} \wedge \Gamma_{b^u, 1}) =^2 V_u ( \Gamma_{ab^u, 1}) =^3 \Gamma_{ab^u, u},$ where $=^1$ and $=^3$ hold by \eqref{eqn:FV identity} and $=^2$ holds by {Step 1.} This proves \eqref{eqn:ring homo} for {Step 2.}

\noindent \textbf{Step 3.} Finally, let $u,v \geq 1$ be any integers. Let $w = \gcd (u,v)$. We again apply the projection formula to $x = V_u (\Gamma_{a, 1})$, $y = \Gamma_{b,1}$, $s=v$, so that $V_u (\Gamma_{a,1}) \wedge V_v (\Gamma_{b,1}) = V_v (F_v (V_u (\Gamma_{a,1})) \wedge \Gamma_{b,1}).$ Its left hand side coincides with that of \eqref{eqn:ring homo} by \eqref{eqn:FV identity}. Its right hand side is $= ^1 V_v (F_v ( \Gamma_{a, u}) \wedge \Gamma_{b,1}) =^2 V_v (w \Gamma_{a^{\frac{v}{w}}, \frac{u}{w}} \wedge \Gamma_{b,1})$, where $=^1$ and $ =^2$ hold by \eqref{eqn:FV identity}. But, Step 2 says that $\Gamma_{a^{\frac{v}{w}}, \frac{u}{w}} \wedge \Gamma_{b,1} = \Gamma_{a ^{\frac{v}{w}} b^{\frac{u}{w}}, \frac{u}{w}}$ so that $V_v (w \Gamma_{a^{\frac{v}{w}}, \frac{u}{w}} \wedge \Gamma_{b,1}) = w V_v (\Gamma_{a ^{\frac{v}{w}} b^{\frac{u}{w}}, \frac{u}{w}}) =^\dagger w \Gamma_{a ^{\frac{v}{w}} b^{\frac{u}{w}}, \frac{uv}{w}}$, where $=\dagger$ holds by \eqref{eqn:FV identity}. This last expression is the right hand side of \eqref{eqn:ring homo}. Thus, we obtain the equality \eqref{eqn:ring homo} and this finishes the proof.
\end{proof}

\begin{thm}\label{thm:sm Witt}
For $\Spec (R) \in \SmAff_k^{\ess}$, $\TH(R)$ is a restricted Witt-complex over $R$, and its sub-pro-system $\TH^M (R)$ is a restricted sub-Witt-complex over $R$. 
\end{thm}

\begin{proof}
As saw in the proof of Theorem \ref{thm:Witt over k}, we already have the restriction $\mathfrak{R}$, the differential $\delta$, the Frobenius $F_r$ and the Verschiebung $V_r$ defined by the same formulas. Furthermore, by Proposition \ref{prop:smooth tau R}, now we have ring homomorphisms $\lambda = \tau_R: \mathbb{W}_m (R) \to \TH^1 (R, 1;m)$ for $m \geq 1$. The properties (i), (ii), (iii), (iv) in Section \ref{sec:DRW} are independent of the choice of the ring, so that what we checked in Theorem \ref{thm:Witt over k} still work. To prove the theorem, the only thing left to be checked is the property (v) that for all $a \in R$ and $r \geq 1$,
\begin{equation}\label{eqn:(v)}
F_r \delta \tau_R ([a]) = \tau_R ([a] ^{r-1}) \delta \tau_R ([a]),
\end{equation}
where we have shrunk the product notation $\wedge$ and taken the ring homomorphism $\lambda$ to be $\tau_R$. To check this, we identify $\mathbb{W}_m (R)$ with $(1+ t R[[t]])^{\times} / (1 + t^{m+1} R[[t]])^{\times}$. 

If $a=0$, then $\tau_R ([a]) = \Gamma _{ (1- 0 \cdot t)}= \emptyset$. So, both sides of \eqref{eqn:(v)} are zero.

If $a=1$, then $\tau_R ([a]) = \tau_R (1- t) = \Gamma_{(1-t)}$. But, in our definition of $\delta$, to compute it, we should first restrict the cycle $\Gamma_{(1-t)} \subset \Spec (R) \times \mathbb{G}_m$ onto $\Spec (R) \times (\mathbb{G}_m \setminus \{ 1 \})$, which becomes empty. Hence, $\delta \tau_R ([a]) = \delta \Gamma_{(1-t)} = 0$, so again both sides of \eqref{eqn:(v)} are zero.

Let $a \in R \setminus\{ 0, 1 \}$. Then $\tau_R ([a]) = \Gamma_{ (1-at)} \subset \Spec (R) \times \mathbb{A}^1$, and $\delta \tau_R ([a])$ is given by the ideal $(1-at, 1-ty_1)$ in $R[ t, y_1]$. Since $t$ is not a zero-divisor in $R[t, y_1]$, we have $(1-at, 1-ty_1) = (1-at, y_1 -a)$ as ideals. Hence, $F_r \delta \tau_R ([a])$ is given by the ideal $(1- a^r t, y_1 -a)$ in $R[t, y_1]$. On the other hand, 
\begin{equation}\label{eqn:explicit}
\begin{array}{lll}
&& \tau_R ([a]^{r-1}) \delta \tau_R ([a]) = \Gamma _{(1- a^{r-1} t)} \wedge \Spec  \left( \frac{R[t, y_1]}{(1-at, y_1 -a)}\right)\\
&=& \Delta^* \left( \frac{ (R \otimes_k R)[ t, y_1] }{  (1- (a^{r-1} \otimes 1) (1 \otimes a) , y_1 -(1 \otimes a))} \right) =^{\dagger} \Spec \left( \frac{R[t, y_1]}{ ( 1-a^r t, y_1 -a)} \right),
\end{array}
\end{equation}
where $\dagger$ holds because $\Delta$ is induced by the product homomorphism $R \otimes_k R \to R$. Hence, both hand sides of \eqref{eqn:(v)} coincide. This completes the proof.
\end{proof}

\begin{thm}\label{cor:sm Witt}
For $\Spec (R) \in \SmAff_k^{\ess}$ and $n, m \geq 1$, there is a unique homomorphism $\tau_{n,m} ^R: \mathbb{W}_m \Omega_R ^{n-1} \to \TH^n (R, n;m)$ that defines a morphism of restricted Witt-complexes over $R$, $\{ \tau^R _{\bullet, m} : \mathbb{W}_m \Omega_R ^{\bullet-1} \to \TH^{\bullet} (R, \bullet; m)\}_m$,
%\[
%\{\tau^R_{\bullet, m}\}: \{\W_{m}\Omega^{\bullet -1}_R\}_{m \ge 1} \to
%\{\TH^{\bullet} (R, \bullet; m)\}_{m \ge 1}
%\]
such that $\tau^R_{1,m} = \tau_R$.
\end{thm}

\begin{proof}
The theorem follows from Theorem \ref{thm:sm Witt} and \cite[Proposition~1.15]{R}. We have $\tau_{1,m} ^R = \tau_R$ because the map $\lambda$ of \S \ref{sec:DRW} is given by $\tau_R$ in Theorem \ref{thm:sm Witt}.
\end{proof}

We have shown in Propositions~\ref{prop:tau_R} and ~\ref{prop:smooth tau R} that $\tau_R$ is a group homomorphism for any $k$-algebra $R$ and is a ring homomorphism if $R$ is smooth. Here, we provide the following information on $\tau_R$.

\begin{thm}\label{thm:integral Witt}
Let $R$ be an integral domain which is an essentially of finite type $k$-algebra. Then $\tau_R$ is injective. It is an isomorphism if $R$ is a UFD. 
\end{thm}

\begin{proof}
Let $K := {\rm Frac}(R)$ and $\iota: R \hookrightarrow K$ be the inclusion. This induces a commutative diagram
$$
\xymatrix{ \mathbb{W}_m (R) \ar[r] ^{\mathbb{W}_m (\iota)}  \ar[d] _{\tau_R} & \mathbb{W}_m (K) \ar[d] ^{\simeq} _{\tau_K} \\
\TH^1 (R, 1; m) \ar[r] & \TH^1 (K, 1; m),}
$$
where the bottom map is the flat pull-back via $\Spec (K) \to \Spec (R)$, and $\tau_K$ is the isomorphism by \cite[Corollary~3.7]{R}. Since $\mathbb{W}_m (\iota)$ is clearly injective (see \cite[Properties A.1.(i)]{R}), it follows that $\tau_R$ is injective, too.

Suppose now $R$ is a UFD and $V$ is an irreducible admissible cycle in $\TZ^1 (R, 1;m)$. Then we must have $(I(V), t) = R[t]$, where $I(V)$ is the ideal of $V$. Since $R[t]$ is a UFD, using basic commutative algebra, one checks that $I(V) = (1-t f(t))$ for some non-zero polynomial $f(t) \in R[t]$. In particular, the map $\tau_R$ is surjective and hence an isomorphism. 
\end{proof}

\subsection{\'Etale descent} 
Finally:

\begin{proof}[Proof of Theorem \ref{thm:descent}]
By \corref{cor:mod-1}, we can assume $|\un{m}| \ge 2$. We set $Y = X/G$, $\lambda = |G|$ and consider the diagram 
\begin{equation}\label{eqn:quot-0}
\xymatrix@C1pc{
G \times X \ar[r]^>>>{\gamma} \ar[d]_{p} & X \ar[d]^{f} \\
X \ar[r]_{f} & Y,}
\end{equation}
where $\gamma$ is the action map and $p$ is the projection. Since $G$ acts freely on $X$, this square is Cartesian and $f$ is {\'e}tale of degree $\lambda$. By \cite[Proposition~1.7]{Fulton}, we have $f^* \circ f_* = p_* \circ \gamma^*: \CH^q(X[r]|D_{\un{m}},n) \to 
\CH^q(X[r]|D_{\un{m}},n)$.

Since $f$ is $G$-equivariant with respect to the trivial $G$-action on $Y$, we see that $f^*$ induces a map $f^*: \CH^q(Y[r]|D_{\un{m}},n) \to \CH^q(X[r]|D_{\un{m}},n)^G$. Moreover, it follows from \cite[Theorem~3.12]{KPv} that $f_* \circ f^*$ is multiplication by $\lambda$. 

On the other hand, it follows easily from the action map $\gamma$ that $p_* \circ \gamma^*(\alpha) = {\underset{g \in G}\sum} \ g^*(\alpha)$. In particular, $p_* \circ \gamma^*(\alpha) = \lambda \cdot \alpha$ if $\alpha \in \CH^q(X[r]|D_{\un{m}},n)^G$.

Since $\lambda \in k^{\times}$ and the Teichm{\"u}ller map is multiplicative with $|\un{m}| \ge 2$, we see that $\lambda \in (\W_{(|\un{m}|-1)}(k))^{\times}$. We conclude from \thmref{thm:algebra}(3) and \corref{cor:Witt-module} that the composite $\CH^q(Y[r]|D_{\un{m}},n) \xrightarrow{f^*} \CH^q(X[r]|D_{\un{m}},n)^G \xrightarrow{\lambda^{-1}f_*} \CH^q(Y[r]|D_{\un{m}},n)$ yields the desired isomorphism.
\end{proof}

%\bigskip

\noindent\emph{Acknowledgments.} The authors thank Wataru Kai for letting them know about his moving lemma for affine schemes. The authors also feel very grateful to the referee whose careful and detailed comments had improved the paper. During this research, JP was partially supported by the National Research Foundation of Korea (NRF) grants No. 2013042157 and No. 2015R1A2A2A01004120, Korea Institute for Advanced Study (KIAS) grant, all funded by the Korean government (MSIP), and TJ Park Junior Faculty Fellowship funded by POSCO TJ Park Foundation.

\end{document}